\documentclass[12pt]{article}

\usepackage{times}

\usepackage[utf8]{inputenc}
\usepackage{commath}
\usepackage{amsthm, amscd}
\usepackage{amsmath}
\usepackage{amssymb}
\usepackage{cite}
\usepackage{setspace}

\usepackage{bigints}
\usepackage{comment}
\usepackage{mathtools}
\usepackage{mathrsfs}
\usepackage{fancyhdr}

\allowdisplaybreaks[4]

\numberwithin{equation}{section}
\usepackage{etoolbox}
\patchcmd{\thebibliography}
  {\settowidth}
  {\setlength{\itemsep}{0pt plus -10pt}\settowidth}
  {}{}
\apptocmd{\thebibliography}
  {
  }
  {}{}
\makeatletter
\newtheorem*{rep@theorem}{\rep@title}
\newcommand{\newreptheorem}[2]{%
\newenvironment{rep#1}[1]{%
 \def\rep@title{#2 \ref{##1}}%
 \begin{rep@theorem}}%
 {\end{rep@theorem}}}
\makeatother

\theoremstyle{theorem}

\newreptheorem{theorem}{Theorem}
\newtheorem{thm}{Theorem}[section]
\newtheorem*{thm*}{Theorem}
\theoremstyle{definition}
\newtheorem{prop}[thm]{Proposition}
\newtheorem*{prop*}{Proposition}
\newtheorem{defn}[thm]{Definition}
\newtheorem{lem}[thm]{Lemma}
\newtheorem{cor}[thm]{Corollary}
\newtheorem*{cor*}{Corollary}
\theoremstyle{remark}

\newtheorem{rem}[thm]{Remark}

\title{\vspace*{-0.5cm}Semiclassical Ohsawa-Takegoshi extension theorem \\
and asymptotics of the orthogonal Bergman kernel}

\author
{Siarhei Finski
}

\date{}

\usepackage[%
    left=1in,%
    right=1in,%
    top=1.1in,%
    bottom=0.8in,%
    paperheight=11in,%
    paperwidth=8.5in%
]{geometry}

\newcommand{\imun} {\sqrt{-1}}

\newcommand{\res}{{\rm{Res}}}
\newcommand{\ext}{{\rm{E}}}

\newcommand{\comp}{\mathbb{C}}
\newcommand{\real}{\mathbb{R}}

\newcommand{\nat}{\mathbb{N}}

\newcommand{\spec}{{\rm{Spec}}}
\newcommand{\dist}{{\rm{dist}}}

\newcommand{\enmr}[1]{\text{End}{(#1)}}

\newcommand{\ccal}{\mathscr{C}}

\newcommand{\dbar}{ \overline{\partial} }

\newcommand{\laplcomp}{\Box}

\newcommand{\rk}[1]{{\rm{rk}} ( #1 )}

\renewcommand{\Im}{\operatorname{Im}}
\newcommand{\scal}[2]{\big< #1, #2 \big>}

\newenvironment{sciabstract}{}

\begin{document} 
\maketitle

\begin{sciabstract}
  \textbf{Abstract.} We study the asymptotics of Ohsawa-Takegoshi extension operator and orthogonal Bergman projector associated with high tensor powers of a positive line bundle.
  \par 
  More precisely, for a fixed complex submanifold in a complex manifold, we consider the operator which associates to a given holomorphic section of a positive line bundle over the submanifold the holomorphic extension of it to the ambient manifold with the minimal $L^2$-norm.
  When the tensor power of the line bundle tends to infinity, we obtain an explicit asymptotic expansion of this operator.
  This is done by proving an exponential estimate for the associated Schwartz kernel and showing that this Schwartz kernel admits a full asymptotic expansion.
We prove similar results for the projection onto holomorphic sections orthogonal to those which vanish along the submanifold.
%As an application, we obtain asymptotically optimal $L^{\infty}$-bounds on the holomoprhic extensions from submanifolds, refining previous results of Zhang and Bost.

\end{sciabstract}

\pagestyle{fancy}
\lhead{}
\chead{Semiclassical Ohsawa-Takegoshi extension theorem}
\rhead{\thepage}
\cfoot{}

%\fancypagestyle{mypagestyle}{%
%  \fancyhf{}% Clear header/footer
%  \fancyhead[OC]{An Author}% Author on Odd page, Centred
%  \fancyhead[EC]{A titlesdfdsfdsfds}% Title on Even page, Centred
%  \fancyfoot[C]{\thepage}%
%  \renewcommand{\headrulewidth}{.4pt}% Header rule of .4pt
%}
%\pagestyle{mypagestyle}

\newcommand{\Addresses}{{% additional braces for segregating \footnotesize
  \bigskip
  \footnotesize
  \noindent \textsc{Siarhei Finski, Institut Fourier - Université Grenoble Alpes, France.}\par\nopagebreak
  \noindent  \textit{E-mail }: \texttt{finski.siarhei@gmail.com}.
}} 

\vspace*{-0.4cm}

\tableofcontents

\section{Introduction}\label{sect_intro}
	One of the main goals of this article is to give an asymptotic version of Ohsawa-Takegoshi extension theorem when the powers of the twisting positive line bundle tend to infinity. Another goal is to establish asymptotic expansion of the Schwartz kernel of the orthogonal Bergman projector onto the holomorphic sections orthogonal to those which vanish along a submanifold.
	\par 
	More precisely, we fix two (not necessarily compact) complex manifolds $X, Y$, of dimensions $n$ and $m$ respectively.
	We fix also a complex embedding $\iota : Y \to X$, a positive line bundle $(L, h^L)$ over $X$ and an arbitrary Hermitian vector bundle $(F, h^F)$ over $X$.
	In particular, we assume that for the curvature $R^L$ of the Chern connection on $(L, h^L)$, the closed real $(1, 1)$-form
	\begin{equation}\label{eq_omega}
		\omega := \frac{\imun}{2 \pi} R^L
	\end{equation}
	is positive.
	We denote by $g^{TX}$ the Riemannian metric on $X$ induced by $\omega$ as follows
	\begin{equation}\label{eq_gtx_def}
		g^{TX}(\cdot, \cdot) := \omega(\cdot, J \cdot),
	\end{equation}
	where $J : TX \to TX$ is the complex structure on $X$.
	We denote by $g^{TY}$ the induced metric on $Y$. 	
	\par 
	\textit{We assume throughout the whole article that the triple $(X, Y, g^{TX})$, and the Hermitian vector bundles $(L, h^L)$, $(F, h^F)$, are of bounded geometry in the sense of Definitions \ref{defn_bnd_subm}, \ref{defn_vb_bg}.}
	\begin{comment} 
	This means that we assume uniform lower bounds $r_X, r_Y > 0$ on the injectivity radii of $X$, $Y$, the existence of the geodesic tubular neighborhood of $Y$ of uniform size $r_{\perp} > 0$ in $X$, and some uniform bounds on related curvatures and the second fundamental form of the embedding.
	\end{comment} 
	\par 
	Now, we fix some positive (with respect to the orientation given by the complex structure) volume forms  $dv_X$, $dv_Y$ on $X$ and $Y$.
	For smooth sections $f, f'$ of $L^p \otimes F$, $p \in \nat$, over $X$, we define the $L^2$-scalar product using the pointwise scalar product $\langle \cdot, \cdot \rangle_h$, induced by $h^L$ and $h^F$, by
	\begin{equation}\label{eq_l2_prod}
		\scal{f}{f'}_{L^2(X)} := \int_X \scal{f(x)}{f'(x)}_h dv_X(x).
	\end{equation}
	Similarly, using $dv_Y$, we introduce the $L^2$-scalar product for sections of $\iota^*( L^p \otimes F)$ over $Y$.
	We denote by $H^0_{(2)}(X, L^p \otimes F)$ and $H^0_{(2)}(Y, \iota^*( L^p \otimes F))$ the vector spaces of holomorphic sections of $L^p \otimes F$ over $X$ and $Y$ respectively with bounded $L^2$-norm.
	\par 
	We assume that for the Riemannian volume forms $dv_{g^{TX}}$, $dv_{g^{TY}}$ of $(X, g^{TX})$, $(Y, g^{TY})$, for any $k \in \nat$, there is $C_k > 0$, such that over $X$ and $Y$, the following bounds hold
	\begin{equation}\label{eq_vol_comp_unif}
		\Big\| \frac{dv_{g^{TX}}}{dv_X} \Big\|_{\ccal^k(X)}, 
		\Big\| \frac{dv_X}{dv_{g^{TX}}} \Big\|_{\ccal^k(X)}, 
		\Big\| \frac{dv_{g^{TY}}}{dv_Y} \Big\|_{\ccal^k(Y)}, 
		\Big\| \frac{dv_Y}{dv_{g^{TY}}} \Big\|_{\ccal^k(Y)} \leq C_k.
	\end{equation}
	By extending Ohsawa-Takegoshi theorem, \cite{OhsTak1}, \cite{Ohsawa}, \cite[\S 13]{DemBookAnMet}, in Theorem \ref{thm_ot_weak}, we prove that there is $p_0 \in \nat$, such that for any $p \geq p_0$, $g \in H^{0}_{(2)}(Y, \iota^*(L^p \otimes F))$, there is $f \in H^{0}_{(2)}(X, L^p \otimes F)$, satisfying $f|_Y = g$.
	Then, for the \textit{Bergman projector} $B_p^{Y}$, given by the orthogonal projection from the space of $L^2$-sections $L^2(Y, \iota^*(L^p \otimes F))$ to $H^{0}_{(2)}(Y, \iota^*(L^p \otimes F))$, we define the \textit{extension} operator
	\begin{equation}\label{eq_ext_op}
		\ext_p :  L^2(Y, \iota^*(L^p \otimes F)) \to H^{0}_{(2)}(X, L^p \otimes F),
	\end{equation}
	by putting $\ext_p g = f$, where $f|_Y = B_p^{Y} g$, and $f$ has the minimal $L^2$-norm among those $f' \in H^{0}_{(2)}(X, L^p \otimes F)$ satisfying $f'|_Y = B_p^{Y} g$ (the set of such $f'$ is closed and convex, hence $f$ exists and it is unique by Hilbert projection theorem; moreover, the operator $\ext_p$ is easily seen to be linear).
	In particular, for $g \in H^{0}_{(2)}(Y, \iota^*(L^p \otimes F))$, we have $(\ext_p g)|_Y = g$.
	Ohsawa-Takegoshi extension theorem in this context means precisely that the operator $\ext_p$ is bounded.
	In this article, we find an explicit asymptotic expansion of $\ext_p$, as $p \to \infty$.
	Remark that $\| \ext_p \|$ is the \textit{optimal constant} for the estimate of the $L^2$-norm of the extension.
	\par 
	More precisely, we identify the normal bundle $N$ of $Y$ in $X$ as an orthogonal complement of $TY$ in $TX$ (with respect to $g^{TX}$), so that we have the following orthogonal decomposition
	\begin{equation}\label{eq_tx_rest}
		TX|_Y \to TY \oplus N.
	\end{equation}
	We denote by $g^{N}$ the metric on $N$ induced by $g^{TX}$.
	Let $P^N : TX|_Y \to N$,  $P^Y : TX|_Y \to TY$, be the projections induced by (\ref{eq_tx_rest}).
	Clearly, $\nabla^{N} := P^N \nabla^{TX}|_{Y}$ defines a connection on $N$.
	\par 
	For $y \in Y$, $Z_N \in N_y$, let $\real \ni t \mapsto \exp_y^{X}(tZ_N) \in X$ be the geodesic in $X$ in direction $Z_N$.
	Bounded geometry condition means, in particular, that this map induces a diffeomorphism of $r_{\perp}$-neighborhood of the zero section in $N$ with a tubular neighborhood $U$ of $Y$ in $X$.
	From now on, we use this identification implicitly.
	Of course, points $(y, 0)$, for $y \in Y$ then correspond to $Y$.
	\par 
	We denote by $\pi_0 : U \to Y$ the natural projection $(y, Z_N) \mapsto y$. 
	Over $U$, we identify $L, F$ to $\pi_0^* (L|_Y), \pi_0^* (F|_Y)$ by the parallel transport with respect to the respective Chern connections along the geodesic $[0, 1] \ni t \mapsto (y,t  Z_N) \in X$, $|Z_N| < r_{\perp}$. We also define a function $\kappa_N$ as follows
	\begin{equation}\label{eq_kappan}
		dv_X = \kappa_N dv_Y \wedge dv_N,
	\end{equation}
	where $dv_N$ is the relative Riemannian volume form on $(N, g^N)$.
	Of course, we have $\kappa_N|_Y = 1$ if
	\begin{equation}\label{eq_comp_vol_omeg}
		dv_X = dv_{g^{TX}}, \qquad dv_Y = dv_{g^{TY}}.
	\end{equation}
	\par 
	\begin{sloppypar}
	Using the above identification, we define the operator $\ext_p^0 : L^2(Y, \iota^*(L^p \otimes F)) \to L^2(X, L^p \otimes F)$ as follows.
	For $g \in L^2(Y, \iota^*(L^p \otimes F))$, we let $(\ext_p^0 g)(x) = 0$, $x \notin U$, and in $U$, we put
	\begin{equation}\label{eq_ext0_op}
		(\ext_p^{0} g)(y, Z_N) = (B_p^Y g)(y) \exp \Big(- p \frac{\pi}{2} |Z_N|^2 \Big) \rho \Big(  \frac{|Z_N|}{r_{\perp}} \Big).
	\end{equation}
	where the norm $|Z_N|$, $Z_N \in N$, is taken with respect to $g^{N}$, and $\rho : \real_{+} \to [0, 1]$ satisfies
	\begin{equation}\label{defn_rho_fun}
		\rho(x) =
		\begin{cases}
			1, \quad \text{for $x < \frac{1}{4}$},\\
			0, \quad \text{for $x > \frac{1}{2}$}.
		\end{cases}
	\end{equation}
	Now, for $g \in H^{0}_{(2)}(Y, \iota^*(L^p \otimes F))$, the section $\ext_p^{0} g$ satisfies $(\ext_p^{0} g)|_Y = g$, but $\ext_p^{0} g$ is not holomorphic over $X$ unless $g$ is zero.
	Nevertheless, as we shall see, $\ext_p^{0} g$ can be used to approximate very well the holomorphic section $\ext_p g$.
	More precisely, we have the following result.
	\end{sloppypar}
	\begin{thm}\label{thm_high_term_ext}
		There are $C > 0$, $p_1 \in \nat^*$, such that for any $p \geq p_1$, we have 
		\begin{equation}\label{eq_ext_as}
			\big\| \ext_p - \ext_p^{0} \big\| \leq \frac{C}{p^{\frac{n - m+ 1}{2}}}.
		\end{equation}
		where $\| \cdot \|$ is the operator norm.
		Also, as $p \to \infty$, we have
		\begin{equation}\label{eq_norm_asymp}
			\big\|  \ext_p^{0} \big\| \sim   \frac{1 }{p^{\frac{n - m}{2}}} \cdot  \sup_{y \in Y} \kappa_N^{\frac{1}{2}}(y).
		\end{equation}
		Moreover, under assumption (\ref{eq_comp_vol_omeg}), in (\ref{eq_ext_as}), one can replace $p^{\frac{n - m+ 1}{2}}$ by an asymptotically better estimate if and only if $Y$ is a totally geodesic submanifold of $(X, g^{TX})$, i.e. the second fundamental form, see (\ref{eq_sec_fund_f}), vanishes.
	\end{thm}
	\begin{rem}
		a) From the bounded geometry condition $\kappa_N$ is bounded, see Section \ref{sect_bnd_geom_cf}.
		\par 
		b) By the remark after (\ref{eq_ext_op}), we see that this theorem gives an asymptotic of the optimal constant in Ohsawa-Takegoshi extension theorem.
		For other versions of optimal Ohsawa-Takegoshi extension theorem, see Błocki \cite{BlockiOToptimal}, Guan-Zhou \cite{GuanZhouOToptimal}.
		\par 
		c)
		Our result refines a theorem of Randriambololona \cite[Théorème 3.1.10]{RandriamTh}, stating in the compact case that for any $\epsilon > 0$, there is $p_1 \in \nat^*$, such that $\big\|  \ext_p \big\| \leq \exp(\epsilon p)$ for $p \geq p_1$.
		Remark, however, that theorem of Randriambololona also works for smooth nonreduced subvarieties $Y$.
		For a generalization of Theorem \ref{thm_high_term_ext} in this context, see the subsequent work \cite{FinOTRed}.
	\end{rem}
	\begin{cor}\label{cor_ot_higher}
		There is $p_1 \in \nat^*$, such that for any $k \in \nat^*$, there is $C > 0$, such that for any $p \geq p_1$, $g \in H^0_{(2)}(Y, \iota^*(L^p \otimes F))$, we have
	\begin{equation}\label{eq_ot_higher}
		\big\| \nabla^k ( \ext_p g) \big\|_{L^2(X)} \leq \frac{C}{p^{\frac{n - m - k}{2}}} \cdot \norm{g}_{L^2(Y)},
	\end{equation}
	where the connection $\nabla$ is induced by the Chern and Levi-Civita connections.
	\end{cor}
	\begin{rem}\label{rem_dem_rem}
		The estimate of type (\ref{eq_ot_higher}) was lacking in Demailly's approach to the invariance of plurigenera for Kähler families, see \cite[(4.19)]{DemBerg}.
		Remark that in \cite{DemBerg}, the manifold $X$ is an open strictly pseudoconvex domain $U$, given by the neighborhood of the diagonal of certain product manifold (hence, $X$ is never compact).
		Our theorem applies to $X$ of this form because the Bergman metric on any strictly pseudoconvex domain has bounded geometry, see Klembeck \cite[Theorem 1 and p. 279]{KlembBergm} and Greene-Krantz \cite[p. 8]{GreenKrantz}.
		We, however, couldn't find a complex-geometric description of submanifolds $Y$ for which there is a metric $g^{TX}$ on $X$ coming from a strictly plurisubharmonic exhaustion function, such that the triple $(X, Y, g^{TX})$ is of bounded geometry.
	\end{rem}
	\par 
	Both Theorem \ref{thm_high_term_ext} and Corollary \ref{cor_ot_higher} appear as almost direct consequences of more precise results about the asymptotics of the \textit{Schwartz kernel} $\ext_p(x, y) \in (L^p \otimes F)_x \otimes (L^p \otimes F)_y^{*}$, $x \in X$, $y \in Y$, of $\ext_p$ with respect to $dv_Y$.
	To state them, recall that $\ext_p(x, y)$ is defined so that for any $g \in L^2(Y, \iota^*(L^p \otimes F))$, $x \in X$, we have
	\begin{equation}\label{eq_sch_ker_ext}
		(\ext_p g)(x) = \int_Y \ext_p(x, y) \cdot g(y) \cdot dv_Y(y).
	\end{equation}
	As we show, this Schwartz kernel shares similar asymptotic behavior with the Bergman kernel, previously studied by Dai-Liu-Ma \cite{DaiLiuMa} and Ma-Marinescu \cite{MaMarOffDiag}.
	More precisely, similarly to \cite{MaMarOffDiag}, we first show that this Schwartz kernel has exponential decay.
	\begin{thm}\label{thm_ext_exp_dc}
		There are $c > 0$, $p_1 \in \nat^*$,  such that for any $k \in \nat$, there is $C > 0$, such that for any $p \geq p_1$, $x \in X$, $y \in Y$, the following estimate holds
		\begin{equation}\label{eq_ext_exp_dc}
			\Big|  \ext_p(x, y) \Big|_{\ccal^k(X \times Y)} \leq C p^{m + \frac{k}{2}} \exp \big(- c \sqrt{p} \dist(x, y) \big),
		\end{equation}
		where the pointwise $\ccal^{k}$-norm of an element from $\ccal^{\infty}(X \times Y, (L^p \otimes F) \boxtimes (L^p \otimes F)^*)$ at a point $(x, y) \in X \times Y$ is the sum of the norms induced by $h^L, h^F$ and $g^{TX}$, evaluated at $(x, y)$, of the derivatives up to order $k$ with respect to the connection induced by the Chern connections on $L, F$ and the Levi-Civita connection on $TX$.
	\end{thm}
	\par 
	Theorem \ref{thm_ext_exp_dc} implies that to understand fully the asymptotics of the Schwartz kernel of the extension operator, it suffices to do so in a neighborhood of a fixed point $(y_0, y_0) \in Y \times Y$ in $X \times Y$.
	Our next result shows that after a reparametrization, given by a homothety with factor $\sqrt{p}$ in certain coordinates around $(y_0, y_0)$, the kernel of this extension operator admits a complete asymptotic expansion in integer powers of $\sqrt{p}$, as $p \to \infty$, similarly to the off-diagonal asymptotic expansion of Bergman kernel \cite{DaiLiuMa}.
	To state it, let us fix some notation.
	\par 
	We define the \textit{second fundamental form} $A \in \ccal^{\infty}(Y, T^*Y \otimes \enmr{TX|_Y})$ by
	\begin{equation}\label{eq_sec_fund_f}
		A := \nabla^{TX}|_{Y} - \nabla^{TY} \oplus \nabla^N.
	\end{equation}
	Trivially, $A$ takes values in skew-symmetric endomorphisms of $TX|_Y$, interchanging $TY$ and $N$.
	\par 
	\begin{sloppypar}
	We fix a point $y_0 \in Y$ and an orthonormal frame $(e_1, \ldots, e_{2m})$ (resp. $(e_{2m+1}, \ldots, e_{2n})$) in $(T_{y_0}Y, g_{y_0}^{TY})$ (resp. in $(N_{y_0}, g^{N}_{y_0})$), such that for $i = 1, \ldots, n$, the following identity is satisfied
	\begin{equation}\label{eq_cond_jinv}
		J e_{2i - 1} = e_{2i}.
	\end{equation}
		Define the coordinate system $\psi_{y_0} : \mathbb{B}_0^{\real^{2m}}(r_Y) \times \mathbb{B}_0^{\real^{2(n - m)}}(r_{\perp}) \to X$, for $Z = (Z_Y, Z_N)$, $Z_Y \in \real^{2m}$, $Z_N \in \real^{2(n - m)}$, $Z_Y = (Z_1, \ldots, Z_{2m})$, $Z_N = (Z_{2m + 1}, \ldots, Z_{2n})$, $|Z_Y| < r_Y$, $|Z_N| < r_{\perp}$, by 
		\begin{equation}\label{eq_defn_fermi}
			\psi_{y_0}(Z_Y, Z_N) := \exp_{\exp_{y_0}^{Y}(Z_Y)}^{X}(\tau_{Z_Y}(Z_N)),
		\end{equation}
		where we identified $Z_Y$, $Z_N$ to elements in $T_{y_0}Y$, $N_{y_0}$, using the fixed frames $(e_1, \ldots, e_{2m})$ and $(e_{2m+1}, \ldots, e_{2n})$, $\tau_{Z_Y}(Z_N) \in N_{\exp_{y_0}^{Y}(Z_Y)}$ is the parallel transport of $Z_N$ along the geodesic $\exp_{y_0}^{Y}(tZ_Y)$, $t = [0, 1]$, with respect to the connection $\nabla^{N}$ on $N$, and $\mathbb{B}_0^{\real^k}(\epsilon)$, $\epsilon > 0$ means the euclidean ball of radius $\epsilon$ around $0 \in \real^k$.
		The coordinates $\psi_{y_0}$ are called the \textit{Fermi coordinates} at $y_0$.
		Define the functions $\kappa_{X, y_0} : \mathbb{B}_0^{\real^{2m}}(r_Y) \times \mathbb{B}_0^{\real^{2(n - m)}}(r_{\perp}) \to \real$, $\kappa_{Y, y_0} : \mathbb{B}_0^{\real^{2m}}(r_Y) \to \real$, by
	\begin{equation}\label{eq_defn_kappaxy}
	\begin{aligned}
		&
		(\psi_{y_0}^* dv_X) (Z)
		=
		\kappa_{X, y_0}(Z)
		d Z_1 \wedge \cdots \wedge d Z_{2n},
		\\
		&
		((\exp_{y_0}^{Y})^* dv_Y) (Z_Y)
		=
		\kappa_{Y, y_0}(Z_Y)
		d Z_1 \wedge \cdots \wedge d Z_{2m}.
	\end{aligned}
	\end{equation}
	Clearly, under assumptions (\ref{eq_comp_vol_omeg}), we have $\kappa_{X, y_0}(0) = \kappa_{Y, y_0}(0) = 1$.
	Once the point $y_0$ is fixed, we usually omit it from the subscript of functions.
	\end{sloppypar}
	\par 
	We fix an orthonormal frame $f_1, \ldots, f_r \in F_{y_0}$, and define $\tilde{f}_1, \ldots, \tilde{f}_r$ by the parallel transport of $f_1, \ldots, f_r$ with respect to the Chern connection $\nabla^F$ of $(F, h^F)$, done first along the path $\psi(t Z_Y, 0)$, $t \in [0, 1]$, and then along the path $\psi(Z_Y, tZ_N)$, $t \in [0, 1]$, $Z_Y \in \real^{2m}$, $Z_N \in \real^{2(n-m)}$, $|Z_Y| < r_Y$, $|Z_N| < r_{\perp}$.
	Similarly, we trivialize $L$ in the neighborhood of $y_0$.
	These frames, as well as the induced frames of the dual vector bundles, allow us to interpret $\ext_p(x, y)$ as an element of $\enmr{F_{y_0}}$ for $x \in X$, $y \in Y$ in a $\min(r_{\perp}, r_Y)$-neighborhood of $y_0$.
	\par 
	We also define the function $\mathscr{E}_{n, m}$ over $\real^{2n} \times \real^{2m}$ as follows
	\begin{equation}\label{eq_ext_defn_fun}
		\mathscr{E}_{n, m}(Z, Z'_Y)
		=
		 \exp \Big(
				-\frac{\pi}{2} \sum_{i = 1}^{m} \big( 
					|z_i|^2 + |z'_i|^2 - 2 z_i \overline{z}'_i
				\big)
				-
				\frac{\pi}{2}
				 \sum_{i = m+1}^{n} |z_i|^2
				\Big),
	\end{equation}
	where $Z = (Z_Y, Z_N)$, $Z_Y, Z'_Y \in \real^{2m}$, $Z_N \in \real^{2(n-m)}$ and $z_i$, $z'_i$ are given by $z_i = Z_{2i - 1} + \imun Z_{2i}$, $z'_j = Z'_{2j - 1} + \imun Z'_{2j}$, for $i = 1, \ldots, n$ and $j = 1, \ldots, m$.
	As we show in Section \ref{sect_model_calc}, $\mathscr{E}_{n, m}(Z, Z'_Y)$ is the Schwartz kernel of the extension operator for a model space. 
	For general manifolds, the extension operator is comparable to this model one, as the following theorem shows.
	\begin{thm}\label{thm_ext_as_exp}
		For any $r \in \nat$, $y_0 \in Y$, there are $J_r^E(Z, Z'_Y) \in \enmr{F_{y_0}}$ polynomials in $Z \in \real^{2n}$, $Z'_Y \in \real^{2m}$, with the same parity as $r$ and $\deg J_r^E \leq 3r$, 
		whose coefficients are polynomials in $R^{TX}$, $A$, $R^F$, $(dv_X / dv_{g^{TX}})^{\pm \frac{1}{2n}}$,  $(dv_Y / dv_{g^{TY}})^{\pm \frac{1}{2n}}$,  and their derivatives of order $\leq 2r$, all evaluated at $y_0$, such that for the functions $F_r^E := J_r^E \cdot \mathscr{E}_{n, m}$ over $\real^{2n} \times \real^{2m}$, the following holds.
		\par 
		There are $\epsilon, c > 0$, $p_1 \in \nat^*$, such that for any $k, l, l' \in \nat$, there is $C  > 0$, such that for any $y_0 \in Y$, $p \geq p_1$, $Z = (Z_Y, Z_N)$, $Z_Y, Z'_Y \in \real^{2m}$, $Z_N \in \real^{2(n - m)}$, $|Z|, |Z'_Y| \leq \epsilon$, $\alpha \in \nat^{2n}$, $\alpha' \in \nat^{2m}$, $|\alpha| + |\alpha'| \leq l$, for  $Q^1_{k, l, l'} := 6(16(n+2)(k+1) + l') + 2l$, the following bound holds
		\begin{multline}\label{eq_ext_as_exp}
			\bigg| 
				\frac{\partial^{|\alpha|+|\alpha'|}}{\partial Z^{\alpha} \partial Z'_Y{}^{\alpha'}}
				\bigg(
					\frac{1}{p^m} \ext_p \big(\psi_{y_0}(Z), \psi_{y_0}(Z'_Y) \big)
					-
					\sum_{r = 0}^{k}
					p^{-\frac{r}{2}}						
					F_r^E(\sqrt{p} Z, \sqrt{p} Z'_Y) 
					\kappa_{X}^{-\frac{1}{2}}(Z)
					\kappa_{Y}^{-\frac{1}{2}}(Z'_Y)
				\bigg)
			\bigg|_{\ccal^{l'}(Y)}
			\\
			\leq
			C p^{- \frac{k + 1 - l}{2}}
			\Big(1 + \sqrt{p}|Z| + \sqrt{p} |Z'_Y| \Big)^{Q^1_{k, l, l'}}
			\exp\Big(- c \sqrt{p} \big( |Z_Y - Z'_Y| + |Z_N| \big) \Big),
		\end{multline}
		where the $\ccal^{l'}$-norm is taken with respect to $y_0$.
		Also, the following identity holds
		\begin{equation}\label{eq_je0_exp}
			J_0^E(Z, Z'_Y) = {\rm{Id}}_{F_{y_0}} \cdot \kappa_N^{\frac{1}{2}}(y_0).
		\end{equation}
		Moreover, under assumption (\ref{eq_comp_vol_omeg}), we have
		\begin{equation}\label{eq_je1_exp}
			J_1^E(Z, Z'_Y) =  {\rm{Id}}_{F_{y_0}} \cdot \pi  \cdot
			 g^{TX}_{y_0} \big(z_N, A(\overline{z}_Y - \overline{z}'_Y) (\overline{z}_Y - \overline{z}'_Y) \big),
		\end{equation}
		where we implicitly identified $Z \in \real^{2n}$ to an element in $T_{y_0}X$ as $Z := \sum Z_i \cdot e_i$, similar notations have been used for $Z_Y, Z_N, Z'_Y$, and $z_Y, z_N$, $z'_Y$ are the induced complex coordinates.
 	\end{thm}
 	%\begin{rem}
 	%	When $Y$ is a point, we recover the asymptotic expansion of peak sections.
 	%\end{rem}
	\par 
	The operator $E_p$ is very much related to the \textit{orthogonal Bergman projector}. The last operator is defined as the orthogonal projector onto the holomorphic sections of $L^p \otimes F$ over $X$, which are orthogonal to those vanishing along $Y$.
	To prove Theorems \ref{thm_ext_exp_dc}, \ref{thm_ext_as_exp}, and out of independent interest, we establish analogous results for this projector.
	More precisely, consider the vector space
	\begin{equation}\label{eq_h00_defn}
		 H^{0, 0}_{(2)}(X, L^p \otimes F) := \Big\{ f \in  H^{0}_{(2)}(X, L^p \otimes F) : f|_Y = 0 \Big\}.
	\end{equation}
	An easy verification shows that (\ref{eq_h00_defn}) is a closed subspace.
	Let $H^{0, \perp}_{(2)}(X, L^p \otimes F)$ be the orthogonal complement of $H^{0, 0}_{(2)}(X, L^p \otimes F)$ in $H^{0}_{(2)}(X, L^p \otimes F)$ with respect to the $L^2$-scalar product.
	\par 
	 Denote by $B_p^{\perp}$, $B_p^{0}$, $B_p^X$ the orthogonal projection from $L^2(X, L^p \otimes F)$ to $H^{0, \perp}_{(2)}(X, L^p \otimes F)$, $H^{0, 0}_{(2)}(X, L^p \otimes F)$ and $H^{0}_{(2)}(X, L^p \otimes F)$ respectively.
	 As we shall explain in Lemma \ref{lem_ext_op_inf_sum}, there is an algebraic relation between $E_p$ and $B_p^{\perp}$.
	 Similarly to (\ref{eq_sch_ker_ext}), we denote by $B_p^{\perp}(x_1, x_2)$, $B_p^{0}(x_1, x_2)$, $B_p^X(x_1, x_2)$ the Schwartz kernels of $B_p^{\perp}$, $B_p^{0}$, $B_p^X$ with respect to $dv_X$.
	\begin{thm}\label{thm_logbk_exp_dc}
		There are $c > 0$, $p_1 \in \nat^*$, such that for any $k \in \nat$, there is $C > 0$, such that for any $p \geq p_1$, $x_1, x_2 \in X$, the following estimate holds
		\begin{equation}\label{eq_logbk_exp_dc}
			\Big|  B_p^{\perp}(x_1, x_2) \Big|_{\ccal^k(X \times X)} \leq C p^{n + \frac{k}{2}} \exp \Big(- c \sqrt{p} \big(  \dist(x_1, x_2) + \dist(x_1, Y) + \dist(x_2, Y) \big) \Big),
		\end{equation}
		where the norm $\ccal^k$ is interpreted in the same way as in Theorem \ref{thm_ext_exp_dc}.
	\end{thm}
	%\begin{rem}
	%	Similar exponential decay property was established by Ma-Marinescu \cite[Theorem 0.1]{MaMarOffDiag} for the Bergman kernel.
	%\end{rem}
	 \par 
	 Hence, for the asymptotics of the Schwartz kernel of $B_p^{\perp}$, it is only left to study it in the neighborhood of a fixed point $(y_0, y_0) \in Y \times Y$ in $X \times X$.
	 To state our result in this direction, we define the function $\mathscr{P}_{n, m}^{\perp}$ over $\real^{2n} \times \real^{2n}$ as follows
	\begin{equation}\label{eq_pperp_defn_fun}
		\mathscr{P}_{n, m}^{\perp}(Z, Z')
		=
		 \exp \Big(
				-\frac{\pi}{2} \sum_{i = 1}^{n} \big( 
					|z_i|^2 + |z'_i|^2
				\big)
				+
				\pi
				 \sum_{i = 1}^{m} z_i \overline{z}'_i
				\Big),
	\end{equation}
	where $Z, Z' \in \real^{2n}$, and $z_i$, $z'_i$ are given by $z_i = Z_{2i - 1} + \imun Z_{2i}$, $z'_i = Z'_{2i - 1} + \imun Z'_{2i}$, for $i = 1, \ldots, n$.
	As we show in Section \ref{sect_model_calc}, $\mathscr{P}_{n, m}^{\perp}(Z, Z')$ is the Schwartz kernel of the orthogonal Bergman projector for the model space. 
	The following theorem shows that the general situation is comparable to this model one.
	We use the same trivialization of $F, F^*$ and $L, L^*$ in the neighborhood of $y_0$ as in Theorem \ref{thm_ext_as_exp}, and interpret $B_p^{\perp}(x_1, x_2)$ as an element of $\enmr{F_{y_0}}$, for $x_1, x_2 \in X$ in a $\min(r_{\perp}, r_Y)$-neighborhood of $y_0$.
	\begin{thm}\label{thm_berg_perp_off_diag}
		For any $r \in \nat$, $y_0 \in Y$, there are polynomials $J_r^{\perp}(Z, Z') \in \enmr{F_{y_0}}$, $Z, Z' \in \real^{2n}$, with the same properties as in Theorem \ref{thm_ext_as_exp}, such that for $F_r^{\perp} := J_r^{\perp} \cdot \mathscr{P}_{n, m}^{\perp}$, the following holds.
		\par 
		There are $\epsilon, c > 0$, $p_1 \in \nat^*$, such that for any $k, l, l' \in \nat$, there is $C  > 0$, such that for any $y_0 \in Y$, $p \geq p_1$, $Z = (Z_Y, Z_N)$, $Z' = (Z'_Y, Z'_N)$, $Z_Y, Z'_Y \in \real^{2m}$, $Z_N, Z'_N \in \real^{2(n-m)}$, $|Z|, |Z'| \leq \epsilon$, $\alpha, \alpha' \in \nat^{2n}$, $|\alpha|+|\alpha'| \leq l$,  for  $Q^2_{k, l, l'} := 3(8(n+2)(k+1) + l') + l$, we have
		\begin{multline}\label{eq_berg_perp_off_diag}
			\bigg| 
				\frac{\partial^{|\alpha|+|\alpha'|}}{\partial Z^{\alpha} \partial Z'{}^{\alpha'}}
				\bigg(
					\frac{1}{p^n} B_p^{\perp}\big(\psi_{y_0}(Z), \psi_{y_0}(Z') \big)
					-
					\sum_{r = 0}^{k}
					p^{-\frac{r}{2}}						
					F_r^{\perp}(\sqrt{p} Z, \sqrt{p} Z') 
					\kappa_{X}^{-\frac{1}{2}}(Z)
					\kappa_{X}^{-\frac{1}{2}}(Z')
				\bigg)
			\bigg|_{\ccal^{l'}(Y)}
			\\
			\leq
			C p^{- \frac{k + 1 - l}{2}}
			\Big(1 + \sqrt{p}|Z| + \sqrt{p} |Z'| \Big)^{Q^2_{k, l, l'}}
			\exp\Big(- c \sqrt{p} \big( |Z_Y - Z'_Y| + |Z_N| + |Z'_N| \big) \Big),
		\end{multline}
		where the $\ccal^{l'}$-norm is taken with respect to $y_0$.
		Also, we have
		\begin{equation}\label{eq_jopep_0}
			J_0^{\perp}(Z, Z') = {\rm{Id}}_{F_{y_0}}.
		\end{equation}
		Moreover, under the assumptions (\ref{eq_comp_vol_omeg}) and notations as in (\ref{eq_je1_exp}), we have
		\begin{equation}\label{eq_jopep_1}
			J_1^{\perp}(Z, Z') = {\rm{Id}}_{F_{y_0}} \cdot 
			 \pi
			 \cdot
				 g^{TX}_{y_0} \big(z_N + \overline{z}'_N, A(Z_Y - Z'_Y) (Z_Y - Z'_Y) \big).
		\end{equation}
	\end{thm}
	\begin{rem}
		We present an algorithmic way to compute the polynomials $J_r^E, J_r^{\perp}$.
		%\par 
		%b)
		%Ma-Zhang in \cite[Theorem 0.2]{MaZhBKSR} established an analogue of Theorem \ref{thm_berg_perp_off_diag} is symplectic reduction setting.
	\end{rem}
	\par 
	In the final part of the introduction, we describe an application of Theorem \ref{thm_ext_as_exp} to the $L^{\infty}$-estimates of holomorphic extensions.
	Remark that Theorem \ref{thm_ext_as_exp} only studies the holomorphic extensions which are optimal for the $L^2$-norm.
	Nevertheless, as an explicit asymptotic formula is given, our theorem can be applied to study extension problems associated with other norms.
	\begin{thm}\label{thm_opt_linf_bnd}
		There are $C > 0$, $p_1 \in \nat^*$, such that for any $p \geq p_1$, $f \in H^0_{(2)}(Y, \iota^*( L^p \otimes F))$,
		\begin{equation}\label{eq_ext_asinfty}
			\big\| \ext_p f \big\|_{L^{\infty}(X)} \leq \Big( 1 + \frac{C}{\sqrt{p}} \Big) \cdot \big\|  f \big\|_{L^{\infty}(Y)}.
		\end{equation}
	\end{thm}
	\begin{rem}
		a) Clearly, for any extension $\tilde{f}$ of $f$ to $X$, we have $\| \tilde{f} \|_{L^{\infty}(X)} \geq \|  f \|_{L^{\infty}(Y)}$.
		Hence, the theorem above tells that $\ext_p f$ saturates the optimal $L^{\infty}$-bound asymptotically. 
		\par 
		b) Theorem \ref{thm_opt_linf_bnd} refines the results of Zhang \cite[Theorem 2.2]{ZhangPosLinBun} and Bost \cite[Proposition 3.6, Theorem A.1]{BostDwork}, cf. also Randriambololona \cite[Théorème B]{RandriamCrelle}, which prove in various settings and generalities that for any $\epsilon > 0$, there is $p_1 \in \nat^*$, such that for any $f \in H^0_{(2)}(Y, \iota^*( L^p \otimes F))$, $p \geq p_1$, there is a holomorphic extension $\tilde{f}$ of $f$ to $X$, verifying $\|  \tilde{f} \|_{L^{\infty}(X)} \leq \exp(\epsilon p) \|  f \|_{L^{\infty}(Y)}$.
		As explained in \cite{ZhangPosLinBun}, \cite{BostDwork}, \cite{RandriamCrelle}, these $L^{\infty}$-bounds are important in Arakelov geometry.
		%Under laxer assumptions on positivity of $(L, h^L)$ and smoothness of $X, Y$, the corresponding statement was studied by Moriwaki \cite[Theorem 4.1]{MoriwExtSup}.
		\par 
		c)
		In realms of non-Archimedean geometry, a related extension problem has been considered by Chen-Moriwaki \cite{ChenMoriw}.
		\par 
		d) It is interesting to study to which extent the $L^2$-optimal extension becomes optimal for other $L^p$-norms, $p \in [1,+\infty[$. See Beatrous \cite{Beatrous} for related results.
	\end{rem}
	\par 
	Let us finally say few words about the tools we use in this article.
	The proofs of Theorems \ref{thm_ext_exp_dc}, \ref{thm_logbk_exp_dc} rely on the exponential estimate for the Bergman kernel, cf. Ma-Marinescu \cite[Theorem 0.1]{MaMarOffDiag}, and the refinement -- in our asymptotic setting -- of the Ohsawa-Takegoshi extension theorem, see Theorem \ref{thm_ot_as_sp}, the proof of which is inspired by Bismut-Lebeau \cite{BisLeb91} and Demailly \cite{Dem82}.
	\par 
	The proofs of Theorems \ref{thm_ext_as_exp}, \ref{thm_berg_perp_off_diag} rely on some techniques from spectral analysis, inspired by \cite{BisLeb91}; on the existence of the so-called uniform Stein atlases (see Definition \ref{eq_stein_uniform_man}) over Kähler manifolds of bounded geometry, which we establish using Hörmander's $L^2$-estimates; on the full off-diagonal asymptotic expansion of the Bergman kernel due to Dai-Liu-Ma \cite{DaiLiuMa}, and on some technical results about the algebras of operators with Taylor-type expansion of the Schwartz kernel, which are inspired by the work of Ma-Marinescu \cite{MaMarToepl}, cf. \cite[\S 7]{MaHol}.
	The general strategy for dealing with semi-classical limits here is inspired by Bismut \cite{BisDem} and Bismut-Vasserot \cite{BVas}.
	\par 
	The technical novelty of this paper, compared to \cite{DaiLiuMa}, is that, unlike $B_p^X$, the operator $B_p^{\perp}$ \textit{is not} the spectral projector associated to the Kodaira Laplacian.
	This breaks apart most of the techniques used in \cite{DaiLiuMa}, \cite{MaHol}, as, for example, the relation with the heat kernel is no longer available. 
	To remedy this, instead of Laplacian, we construct an “ad hoc" operator, see (\ref{eq_bpperp_a_ap_defn}), based on the restriction map and $B_p^X$, so that $B_p^{\perp}$  \textit{is} the spectral projector associated to this new operator.
	\par 
	\begin{sloppypar}
	A similar idea of using spectral theory of operators other than Laplacian in the study of Bergman kernel has been used in symplectic reduction setting by Ma-Zhang \cite{MaZhBKSR}. 
	The role of normal direction in the decay of Schwartz kernels as in Theorems \ref{thm_ext_as_exp} and \ref{thm_berg_perp_off_diag} was already present in \cite{MaZhBKSR}, even though the contexts of the two problems are completely different.
	In their manuscript, authors use the deformation of Laplacian by the Casimir operator, coming from Hamiltonian action of a compact Lie group. 
	This deformation, differently from our setting, is a differential operator itself.
	This makes technical details of our article different, as here the spectral theory is applied to an operator, which is no longer local, and actually has a smooth Schwartz kernel.
	\end{sloppypar}
	\par 
	We note that Theorems \ref{thm_logbk_exp_dc}, \ref{thm_berg_perp_off_diag} can be reformulated in terms of the so-called \textit{logarithmic Bergman kernel}, which corresponds to the Schwartz kernel of $B_p^{0}$.
	This is due to the obvious relation $B_p^X = B_p^{0} + B_p^{\perp}$, and the fact that the Schwartz kernel of $B_p^X$ is already well-understood by the results of Tian \cite{TianBerg}, Catlin \cite{Caltin}, Zelditch \cite{ZeldBerg}, Dai-Liu-Ma \cite{DaiLiuMa} and Ma-Marinescu \cite{MaMarOffDiag}.
	\par 	
	This paper is organized as follows.
	In Section \ref{sect_20}, we recall the definitions of manifolds (resp. pairs of manifolds, vector bundles) of bounded geometry.
	We compare Fermi coordinates to geodesic coordinates.
	We then introduce and study the notion of a uniform Stein atlases.
	In Section \ref{sect_2}, we prove that the set of operators over manifolds of bonded geometry, admitting certain bounds on the Schwartz kernels, forms an algebra under the composition.
	In Section \ref{sect_spec_bnd_res}, we establish a spectral bound for the restriction operator.
	Finally, in Section \ref{sect_lgbk_prfs}, by using all the above results, we establish the results announced in this section.
	\par {\bf{Notations.}}
	We use notations $X, Y$ for complex manifolds and $M, H$ for real manifolds.
	The complex (resp. real) dimensions of $X, Y$ (resp. $M, H$) are denoted here by $n, m$.
	An operator $\iota$ always means an embedding $\iota : Y \to X$ (resp. $\iota : H \to M$).
	We denote by $\res_Y$ (resp. $\res_H$) the operator of restriction of sections of a certain vector bundle over $X$ to $Y$ (resp. $M$ to $H$). 
	\par 
	For a Riemannian manifold $(M, g^{TM})$, we denote the Levi-Civita connection by $\nabla^{TM}$, and by $dv_{g^{TM}}$ the Riemannian volume form.
	For a closed subset $W \subset M$, $r \geq 0$, let $\mathbb{B}_{W}^{M}(r)$ be the geodesic tubular neighborhood of radius $r$ around $W$. 
	For a Hermitian vector bundle $(E, h^E)$, note
	$
		\mathbb{B}_{r}(E) := \{ Z \in E : |Z|_{h^E} < r \}
	$.
	\par 
	For a fixed volume form $dv_M$ on $M$, we denote by $L^2(dv_M, h^E)$ the space of $L^2$-sections of $E$ with respect to $dv_M$ and $h^E$. When $dv_M = dv_{g^{TM}}$, we also use the notation $L^2(g^{TM}, h^E)$. When there is no confusion about the data, we also use the simplified notation $L^2(M)$ or even just $L^2$.
	\par 
	We denote by $dv_{\comp^n}$ the standard volume form on $\comp^n$.
	We view $\comp^m$ (resp. $\real^{2m}$) embedded in $\comp^n$ (resp. $\real^{2n}$) by the first $m$ coordinates.
	For $Z \in \real^k$, we denote by $Z_l$, $l = 1, \ldots, k$, the coordinates of $Z$.
	If $Z \in \real^{2n}$, we denote by $z_i$, $i = 1, \ldots, n$ the induced complex coordinates $z_i = Z_{2i - 1} + \imun Z_{2i}$.
	We frequently use the decomposition $Z = (Z_Y, Z_N)$, where $Z_Y = (Z_1, \ldots, Z_{2m})$ and $Z_N = (Z_{2m + 1}, \ldots, Z_{2n})$. 
	For a fixed frame $(e_1, \ldots, e_{2n})$ in $T_yX$, $y \in Y$, we implicitly identify $Z$ (resp. $Z_Y$, $Z_N$) to an element in $T_yX$ (resp. $T_yY$, $N_y$) by
	\begin{equation}\label{eq_Z_ident}
		Z = \sum_{i = 1}^{2n} Z_i e_i, \quad Z_Y = \sum_{i = 1}^{2m} Z_i e_i, \quad Z_N = \sum_{i = 2m + 1}^{2n} Z_i e_i.
	\end{equation}
	If the frame $(e_1, \ldots, e_{2n})$ satisfies (\ref{eq_cond_jinv}), we denote $\frac{\partial}{\partial z_i} := \frac{1}{2} (e_{2i-1} - \imun e_{2i})$, $\frac{\partial}{\partial \overline{z}_i} := \frac{1}{2} (e_{2i-1} + \imun e_{2i})$, and identify $z, \overline{z}$ to vectors in $T_yX \otimes_{\real} \comp$ as follows
	\begin{equation}\label{eq_z_ovz_id}
		z = \sum_{i = 1}^{n} z_i \cdot \frac{\partial}{\partial z_i}, \qquad
		\qquad
		\overline{z} = \sum_{i = 1}^{n} \overline{z}_i \cdot \frac{\partial}{\partial \overline{z}_i}.
	\end{equation}
	Clearly, in this identification, we have $Z = z + \overline{z}$, $({\rm{Id}} - \imun J) Z = 2z$, and $({\rm{Id}} + \imun J) Z = 2 \overline{z}$.
	We define $z_Y, \overline{z}_Y \in T_yY \otimes_{\real} \comp$, $z_N, \overline{z}_N \in N_y \otimes_{\real} \comp$ in a similar way. 
	\par 
	For $\alpha = (\alpha_1, \ldots, \alpha_k) \in \nat^k$, $B = (B_1, \ldots, B_k) \in \comp^k$, we note 
	\begin{equation}
		|\alpha| = \sum_{i = 1}^{k} \alpha_i, \quad \alpha! = \prod_{i = 1}^{k} \alpha_i!, \quad B^{\alpha} = \prod_{i = 1}^{k} B_i^{\alpha_i}.
	\end{equation}
	\par {\bf{Acknowledgements.}} Author would like to warmly thank Jean-Pierre Demailly, the numerous enlightening discussions with whom inspired this article.
	We also thank Jingzhou Sun for his interest in this article and for pointing out several misprints.
	Finally, we would like to thank the anonymous referees for many helpful comments.
	This work is supported by the European Research Council grant ALKAGE number 670846 managed by Jean-Pierre Demailly.

\section{Manifolds of bounded geometry and uniform Stein atlases}\label{sect_20}
	We recall the definitions of manifolds (resp. pairs of manifolds, vector bundles) of bounded geometry and study some of their properties. 
	Then, we introduce and study uniform Stein atlases. 
	\par Manifolds of bounded geometry are certain complete manifolds for which some uniform boundness conditions on the curvature and injectivity radii are assumed.
	The reason for considering those types of manifolds in this article is twofold.
	\par 
	First, they appear naturally already in the study of compact manifolds.
	This is due to the fact that our main philosophy here is to reduce all the statements from general manifolds to $\comp^n$. 
	As we would like our theory to work with $\comp^n$ as well, which is no longer compact, the setting of compact manifolds is not appropriate for our needs.
	It turns out that the class of manifolds of bounded geometry is wide enough to contain both compact manifolds and $\comp^n$, and restrictive enough to satisfy some essential estimates in our approach.
	\par
	Second, the full force of Ohsawa-Takegoshi extension theorem comes from the fact that it can be applied to local problems over general weakly pseudoconvex domains.
	Also, many geometric applications of the extension theorem, as \cite{DemBerg} (which partially motivated the current article), are formulated for non-compact manifolds, see Remark \ref{rem_dem_rem}.
	As we would like our theorems to be useful in these contexts as well, we need to abandon the compactness assumption.
	\par 
	This section is organized as follows.
	In Section \ref{sect_bnd_geom_cf}, we recall the definitions of manifolds (resp. pairs of manifolds, vector bundles) of bounded geometry.
	In Section \ref{sect_coord_syst}, we compare the geodesic and Fermi coordinates.
	In Section \ref{sect_par_transport}, we calculate the holonomy of the vector bundle along the paths, adapted to the two coordinate systems.
	In Section \ref{sect_cmplx_fermi}, we give a formula for the complex structure in Fermi coordinates and study quasi-plurisubharmonicity of some functions on manifolds of bounded geometry.
	Finally, in Section \ref{sect_stein_atl}, we introduce and study uniform Stein atlases, and prove that any Kähler manifold of bounded geometry admits uniform Stein atlas.
	\par 
	A reader, willing to understand only the compact case, might safely skip Sections \ref{sect_bnd_geom_cf} and \ref{sect_stein_atl}, and skim through most of Section \ref{sect_cmplx_fermi}, as for compact manifolds, the statements are well-known.

	\subsection{Coordinate-free and coordinate-wise descriptions of bounded geometry}\label{sect_bnd_geom_cf}
	In this section, we recall the definitions of manifolds (resp. pairs of manifolds, vector bundles) of bounded geometry. 
	There are mainly two ways to define objects of bounded geometry. 
	Either one uses some bounds on the curvature (and related objects) -- this is the coordinate-free description -- or one uses charts constructed by geodesics and bounds the relevant structures (as the metric tensor) and its derivatives in these coordinates -- the coordinate-wise description. 
	The equivalence of these two perspectives is established in the works of Eichhorn \cite{EichBoundG}, Schick \cite{SchBound} and Große-Schneider \cite{GrosSchnBound}. 
	In this section, we recall these statements precisely.
	\begin{defn}\label{defn_bnd_g_man}
		We say that a Riemannian manifold $(M, g^{TM})$ is of bounded geometry if the following two conditions hold.
		\\ \hspace*{0.3cm} \textit{(i)} The injectivity radius of $(M, g^{TM})$ is bounded below by a positive constant $r_M$. 
		\\ \hspace*{0.3cm} \textit{(ii)} Every covariant derivative of the Riemann curvature tensor $R^{TM}$ of $M$ is bounded, i.e. for any $k \in \nat$, there is a constant $C_k > 0$ such that for any $l = 0, \ldots, k$, we have
		\begin{equation}\label{eq_bnd_curv_tm}
			| (\nabla^{TM})^l R^{TM} | \leq C_k,
		\end{equation}
		where $\nabla^{TM}$ is the connection induced by the Levi-Civita connection, and the pointwise norm is taken with respect to $g^{TM}$.
	\end{defn}
	\begin{rem}
		The condition \textit{(i)} from Definition \ref{defn_bnd_g_man} implies that $(M, g^{TM})$ is complete.	
	\end{rem}
	\par
	Let us fix $x_0 \in M$ and an orthonormal frame $(e_1, \ldots, e_n)$ of $(T_{x_0}M, g^{TM}_{x_0})$.
	We identify $\real^n$ to $T_{x_0}M$ implicitly as in (\ref{eq_Z_ident}).
	We introduce the map $\phi_{x_0} : \real^n \to M$, $x_0 \in M$, as follows
	\begin{equation}\label{eq_phi_defn}
		\phi_{x_0}(Z) := \exp^{M}_{x_0}(Z).
	\end{equation}
	As a general rule, whenever the point $x_0$ is implicit, we drop it out from the subscript.
	As it was proved, for example, in \cite[Theorem A]{EichBoundG} and \cite[Proposition 3.3]{SchBound}, the metric tensor, written in geodesic coordinates, has bounded derivatives, and the transition maps between two different geodesic coordinates can be bounded uniformly in a similar way. This gives us \textit{the coordinate-wise approach} through geodesic coordinates to manifolds of bounded geometry.
	\par 
	\par 
	Now, let $(H, g^{TH})$ be an embedded submanifold of $(M, g^{TM})$, $g^{TH} := g^{TM}|_H$.
	We identify the normal bundle $N$ of $H$ in $M$ as an orthogonal complement of $TH$ in $TM$ as in (\ref{eq_tx_rest}).
	We denote by $g^{N}$ the metric on $N$ induced by $g^{TM}$, and define $\nabla^{N}$ as after (\ref{eq_tx_rest}).
	We denote by $A$ the second fundamental form of the embedding of $H$ in $M$, defined as in  (\ref{eq_sec_fund_f}).
	\begin{defn}\label{defn_bnd_subm}
		We say that the triple $(M, H, g^{TM})$ is of bounded geometry if the following conditions are fulfilled.
		\\ \hspace*{0.3cm} \textit{(i)} The manifold $(M, g^{TM})$ is of bounded geometry.
		\\ \hspace*{0.3cm} \textit{(ii)} The injectivity radius of $(H, g^{TH})$ is bounded below by a positive constant $r_H$. 
		\\ \hspace*{0.3cm} \textit{(iii)} There is a collar around $H$ (a tubular neighborhood of fixed radius), i.e. there is $r_{\perp} > 0$ such that for any $x, y \in H$, the normal geodesic balls $B^{\perp}_{r_{\perp}}(x), B^{\perp}_{r_{\perp}}(y)$, obtained by the application of the exponential mapping to vectors, orthogonal to $H$, of norm, bounded by $r_{\perp}$, are disjoint.
		\\ \hspace*{0.3cm} \textit{(iv)} Every covariant derivative of $A$ of $M$ is bounded, i.e. for all $k \in \nat$, there is a constant $C_k > 0$ such that for any $l = 0, \ldots, k$, we have
		\begin{equation}\label{eq_bnd_a_ck}
			| \nabla^l A | \leq C_k,
		\end{equation}
		where $\nabla$ is the connection induced by the Levi-Civita connection $\nabla^{TH}$ and $\nabla^N$, and the pointwise norm is taken with respect to $g^{TM}$.
	\end{defn}	
	Clearly, condition \textit{(iii)} from Definition \ref{defn_bnd_subm} means that the map $\mathbb{B}_{r_{\perp}}(N) \to M$, $(y, Z_N) \mapsto \exp^{M}_{y}(Z_N)$, $y \in H$, $Z_N \in N_y$, $|Z_N| < r_{\perp}$, is a diffeomorphism onto a tubular neighborhood $U := \mathbb{B}_H^M(r_{\perp})$ of $H$ in $M$.
	We define the projection
	\begin{equation}\label{eq_pi_proj}
		\pi_0 : U \to H, \qquad \qquad \exp^{M}_{y}(Z_N) \mapsto y.
	\end{equation}
	\par 
	\begin{sloppypar}
		We now fix a point $y_0 \in H$ and an orthonormal frame $(e_1, \ldots, e_m)$ (resp. $(e_{m+1}, \ldots, e_n)$) in $(T_{y_0}H, g^{TH})$ (resp. in $(N, g^{N})$).
		We define the \textit{Fermi coordinates}, $\psi_{y_0}$, at $y_0$ as in (\ref{eq_defn_fermi}).
		Again, as both manifolds $(M, g^{TM})$, $(H, g^{TH})$, are complete, one can extend the domain of $\psi_{x_0}$ to $\real^n$.
		As it was established, for example in \cite[Lemma 3.9]{SchBound} and \cite[Theorem 4.9]{GrosSchnBound}, the derivatives of the metric tensor in $\psi_{x_0}$ coordinates are uniformly bounded. Similarly, the derivatives of the transition maps are uniformly bounded.
		From this, we see directly that the function $\kappa_N$, defined as in (\ref{eq_kappan}), is uniformly bounded on the submanifold.
	\end{sloppypar}
	Define $R > 0$ as follows
	\begin{equation}\label{eq_r_defn_const}
		R := \min \Big\{ 
			\frac{r_M}{2}, \frac{r_H}{4}, \frac{r_{\perp}}{4}		
		\Big\}.
	\end{equation}
	\begin{comment}
	\begin{prop}\label{prop_fin_overl_subman}
		Assume $(M, H, g^{TM})$ is of bounded geometry.
		Then for any $0 < \epsilon < \frac{R}{4}$, there is a set of points $y_i \in Y$, $i \in \nat$, such that $\mathbb{B}_{y_i}^H(\epsilon)$ have  a uniform finite overlap (independent of $\epsilon$) and cover $\mathbb{B}_H^M( \frac{\epsilon}{2} )$.
	\end{prop}
	\begin{proof}
		The proof is identical to the proof of Proposition \ref{prop_fin_overl_man} with only one change: instead of (\ref{eq_bnd_metr_tens0}), one should use (\ref{eq_bndtr_metr_tens}).
	\end{proof}
	\end{comment}
	Now, there is a diffeomorphism $h_{y_0} : \mathbb{B}_0^{\real^n}(R) \to \real^n$, such that the following holds
	\begin{equation}\label{eq_h_defn_tr_m}
		\psi_{y_0} = \phi_{y_0} \circ h_{y_0}.
	\end{equation}
	As it was established, for example, in {\cite[Lemma 4.7]{GrosSchnBound}, the derivatives of $h_{y_0}$ are then uniformly bounded. All in all, this gives the \textit{coordinate-wise approach} through Fermi coordinates to triples of bounded geometry.
	\par 
	\begin{defn}\label{defn_vb_bg}
		Let $(E, \nabla^E, h^{E})$ be a Hermitian vector bundle with a fixed Hermitian connection over a manifold $(M, g^{TM})$ of bounded geometry.
		We say that $(E, \nabla^E, h^{E})$ is of bounded geometry if for any $k \in \nat$, there is a constant $C_k > 0$ such that for any $l = 0, \ldots, k$, we have
		\begin{equation}\label{eq_bnd_re_ck}
			| \nabla^l R^E | < C_k,
		\end{equation}
		where $\nabla$ is the connection induced by the Levi-Civita connection $\nabla^{TM}$ and $\nabla^E$, and the pointwise norm is taken with respect to $g^{TM}$.
		\par 
		If $(E, h^{E})$ is a Hermitian vector bundle over a \textit{complex manifold}, we say that it is of bounded geometry if $(E, \nabla^E, h^{E})$ is of bounded geometry for the Chern connection $\nabla^E$ on $(E, h^{E})$.
	\end{defn}
	\par 
	Let us now give the coordinate-wise description for vector bundles of bounded geometry.
	Let us first construct a trivialization of vector bundle $(E, \nabla^E)$, $\rk{E} = r$, as follows.
	We fix a point $x_0 \in M$ and an orthonormal frame $f_1, \ldots, f_r \in E_{x_0}$. 
	Let $\tilde{f}'_1, \ldots, \tilde{f}'_r$ be a frame of $E$ over $\mathbb{B}_{x_0}^{M}(r_M)$, obtained by the parallel transport of $f_1, \ldots, f_r$ along the curve $\phi_{x_0}(tZ)$, $t \in [0, 1]$, $Z \in T_{x_0}M$, $|Z| < r_M$.
	Clearly, as $\nabla^E$ is Hermitian,  $\tilde{f}'_1, \ldots, \tilde{f}'_r$ is an orthonormal frame over $\mathbb{B}_{x_0}^{M}(r_M)$.
	We denote by $\Gamma^E{}'$ the connection form of $(E, \nabla^E)$ with respect to this frame.
	As it was established, for example in \cite[Theorem B]{EichBoundG}, the derivatives of $\Gamma^E{}'$ are uniformly bounded for vector bundles of bounded geometry.
	\begin{comment}
	\begin{prop}[{\cite[Theorem B]{EichBoundG}}]\label{prop_bndg_vect_1}
		Assume that for some $k \in \nat^*$, the bounds (\ref{eq_bnd_curv_tm}), (\ref{eq_bnd_re_ck}) hold.
		Then there is $D_k > 0$, such that for any $x_0 \in X$, $l = 0, \ldots, k-1$, we have 
		\begin{equation}\label{eq_bnd_metr_tens_gamma}
			\| \Gamma^E{}' \|_{\ccal^{l}(\mathbb{B}_0^{\real^n}(r_M))} \leq D_k.
		\end{equation}
		Moreover, $D_k$ depends only on $k$, $C_k$ from (\ref{eq_bnd_curv_tm}), (\ref{eq_bnd_re_ck}) and $r_M$.
	\end{prop}
	\end{comment}
	\par
	Now, let us consider another trivialization of $(E, \nabla^E)$.
	We place ourselves in a setting where $(M, H, g^{TM})$ is a triple of bounded geometry.
	We fix a point $y_0 \in H$  and an orthonormal frame $f_1, \ldots, f_r \in E_{y_0}$. 
	We define $\tilde{f}_1, \ldots, \tilde{f}_r$ by the parallel transport as it was done before Theorem \ref{thm_ext_as_exp}.
	We denote by $\Gamma^E$ the connection form of $(E, \nabla^E)$ with respect to this frame.
	As it was established, for example in \cite[Lemma 5.13]{GrosSchnBound}, the derivatives of $\Gamma^E{}$ are uniformly bounded for triples of bounded geometry.
	\begin{comment}
	\begin{prop}[{\cite[Lemma 5.13]{GrosSchnBound}}]\label{prop_bndg_vect_2}
		Assume that for some $k \in \nat^*$, the bounds (\ref{eq_bnd_curv_tm}), (\ref{eq_bnd_a_ck}), (\ref{eq_bnd_re_ck}) hold.
		Then there is $D_k > 0$, such that for any $y_0 \in Y$, $l = 0, \ldots, k-1$, we have 
		\begin{equation}\label{eq_bnd_metr_tens_gamma2}
			\| \Gamma^E \|_{\ccal^{l}(\mathbb{B}_0^{\real^n}(R))} \leq D_k.
		\end{equation}
		Moreover, $D_k$ depends only on $k$, $C_k$ from (\ref{eq_bnd_curv_tm}), (\ref{eq_bnd_a_ck}), (\ref{eq_bnd_re_ck}) and $R$.
	\end{prop}
	\end{comment}
	Let $\xi_E$ be the function, defined in $\mathbb{B}_{y_0}^{M}(R)$, with values in $\enmr{\comp^r}$, such that
	\begin{equation}\label{eq_frame_tilde}
		(\tilde{f}_1, \ldots, \tilde{f}_r) = \exp(\xi_E) \cdot (\tilde{f}'_1, \ldots, \tilde{f}'_r),
	\end{equation}
	where we view $(\tilde{f}_1, \ldots, \tilde{f}_r)$ and $(\tilde{f}'_1, \ldots, \tilde{f}'_r)$ as $r \times 1$ matrices.
	Clearly, for triples of bounded geometry, the derivatives of $\xi_E$ are uniformly bounded.
	\begin{comment}
	\begin{prop}\label{prop_xi_un_bnd}
		Assume that for some $k \in \nat^*$, the bounds (\ref{eq_bnd_curv_tm}), (\ref{eq_bnd_a_ck}), (\ref{eq_bnd_re_ck}) hold.
		Then there is $D_k > 0$, such that for any $y_0 \in Y$, $l = 0, \ldots, k-1$, we have 
		\begin{equation}\label{eq_bnd_xi_tens}
			\| \xi_E \|_{\ccal^{l}(\mathbb{B}_{y_0}^{M}(R))} \leq D_k.
		\end{equation}
		Moreover, $D_k$ depends only on $k$, $C_k$ from (\ref{eq_bnd_curv_tm}), (\ref{eq_bnd_a_ck}), (\ref{eq_bnd_re_ck}) and $R$.
	\end{prop}
	\begin{proof}
		The proof is analogous to the proof of Proposition \ref{prop_tr_map}, cf. \cite[Lemma 5.13]{GrosSchnBound}.
	\end{proof}
	\end{comment}
	
\subsection{Diffeomorphism between Fermi and geodesic coordinates}\label{sect_coord_syst}
	
	In this section, we study the Taylor expansion of the diffeomorphism comparing geodesic and Fermi coordinates.
	\par 
	We conserve the notations from Section \ref{sect_bnd_geom_cf} and we place ourselves in the setting of a triple $(M, H, g^{TM})$ of bounded geometry.
	Let $A \in \ccal^{\infty}(H, T^*H \otimes \enmr{TM|_H})$ be as in (\ref{eq_sec_fund_f}).
	We define an auxiliary form $B \in \ccal^{\infty}(H, {\rm{Sym}}^2(T^*M|_H) \otimes TM|_H)$ in the notations (\ref{eq_sec_fund_f}) by
	\begin{equation}\label{eq_b_defn}
		B(u) := B(u, u) := \frac{1}{2} A(P^H u)P^H u + A(P^H u)P^N u, \qquad  u \in TM|_H,
	\end{equation}
	where $P^H : TM|_H \to TH$ is the orthogonal projection.
	\par 
	We fix a point $y_0 \in H$ and an orthonormal frame $(e_1, \ldots, e_m)$ (resp. $(e_{m+1}, \ldots, e_n)$) in $(T_{y_0}H, g^{TH})$ (resp. in $(N, g^{N})$). 
	Recall that in (\ref{eq_defn_fermi}) and (\ref{eq_phi_defn}), we defined two coordinate systems $\psi_{y_0}$, $\phi_{y_0}$ in a neighborhood of $y_0$, and in (\ref{eq_h_defn_tr_m}), we defined a diffeomorphism $h_{y_0}: \mathbb{B}_0^{\real^{n}}(R) \to U \subset \real^{n}$, for $R$ from (\ref{eq_r_defn_const}) and a certain open subset $U$.
	We drop out $y_0$ from the subscripts from now on.
	The main goal of this section is to study the Taylor expansion of $h$ at $0 \in \real^{2n}$.
	\begin{prop}\label{prop_diff_exp}
		The diffeomorphism $h$ has the following Taylor expansion
		\begin{equation}\label{eq_prop_diff_exp}
			h(Z) = Z + B(Z) + O(|Z|^3).
		\end{equation}
		Moreover, the coefficients of order $r \in \nat$ in the above Taylor expansion can be expressed in terms of $R^{TM}$, $A$, and their derivatives up to order $r-2$ with respect to $\nabla^{TM}$, evaluated at $y_0$.
	\end{prop}
	Before presenting the proof of Proposition \ref{prop_diff_exp}, let us prove some auxiliary results.
	\begin{lem}\label{lem_cal_1}
		Let $\gamma(t)$ be a geodesic in $(M, g^{TM})$, and let $v := \gamma'(t)$. Then for any $k \in \nat$, the following identity between operators on smooth functions on $M$ holds
		\begin{equation}\label{eq_lem_cal_1}
			v^{\otimes k} \cdot (\nabla^{TM})^{\otimes k}  = \Big( \frac{\partial}{\partial v} \Big)^{k},
		\end{equation}
		where we view $(\nabla^{TM})^{\otimes k}$ as an operator $(\nabla^{TM})^{\otimes k} : \ccal^{\infty}(M) \to \ccal^{\infty}(M, (T^*M)^{\otimes k})$, and $\cdot$ on the left-hand side of (\ref{eq_lem_cal_1}) means the contraction.
 	\end{lem}
 	\begin{proof}
 		The proof is a consequence of the simple fact that $\nabla_v^{TM} v = 0$ and the formula
 		\begin{equation}\label{eq_der_diff_form}
 			(\nabla^{TM} \alpha)(X_1, \ldots, X_{k + 1})
 			=
 			X_1 \alpha(X_2, \ldots, X_{k + 1})
 			-
 			\sum_{i = 2}^{k + 1} \alpha(X_2, \ldots, \nabla^{TM}_{X_1} X_i, \ldots, X_{k + 1}),
 		\end{equation}
 		where $\alpha$ is a $k$-form and $X_1, \ldots, X_{k + 1}$ are some vector fields.
 	\end{proof}
 	\begin{cor}\label{cor_cal_1}
 		Let $u$ be a smooth function on $M$. Then for any smooth function $u$ on $M$ on the level of formal Taylor expansions, the following identity holds
 		\begin{equation}\label{eq_cor_cal_1}
 			u \big( \exp^M_{y_0}(Z) \big) = \sum_{k = 0}^{\infty} \frac{1}{k!} Z^{\otimes k} \cdot (\nabla^{TM})^k u(y_0).
 		\end{equation}
 	\end{cor}
 	\begin{proof}
 		We have the following formal identity
 		\begin{equation}
			u \big( \exp^M_{y_0}(Z) \big) 
			=
			\sum_{k = 0}^{\infty} \frac{1}{k!} \frac{d^n}{dt^n} u \big( \exp^M_{y_0}(tZ) \big).
 		\end{equation}
 		The proof now is a direct consequence of Lemma \ref{lem_cal_1}.
 	\end{proof}
 	Both Lemma \ref{lem_cal_1} and Corollary \ref{cor_cal_1} appeared in Gavrilov \cite[\S 2]{Gavr}.
 	Let us now define the connection $\nabla^a$ on $TM|_H$ as follows
 	\begin{equation}
 	\nabla^a = \nabla^{TH} \oplus \nabla^N,
 	\end{equation}
 	where we used the notation as in (\ref{eq_sec_fund_f}).
 	\begin{lem}\label{lem_cal_2}
 		Let $\gamma(t)$ be a geodesic in $(H, g^{TH})$, and let $v := \gamma'(t)$. Let $Z_1(t), \ldots, Z_l(t) \in N$, $l \in \nat$ be vector fields along $\gamma(t)$, which are parallel with respect to $\nabla^{N}$. Then for any $k \in \nat$, the following identity between operators on smooth functions on $M$ holds
		\begin{multline}\label{eq_lem_cal_2}
			\Big( v^{\otimes k} \cdot (\nabla^{TH})^{\otimes k} \Big) 
			\circ
			{\rm{Res}}_H
			\circ 
			\Big( (Z_1 \otimes \cdots \otimes Z_l) \cdot (\nabla^{TM})^{\otimes l} \Big)
			\\
			= 
			{\rm{Res}}_H \circ \Big( (v^{\otimes k} \otimes Z_1 \otimes \cdots \otimes Z_l) \cdot (\nabla^a)^{\otimes k} (\nabla^{TM})^{\otimes l} \Big),
		\end{multline}
		where the compositions of the connections are interpreted in the same way as in Lemma \ref{lem_cal_1}, and ${\rm{Res}}_H$ is a map restricting sections over $M$ to sections over $H$.
 	\end{lem}
 	\begin{proof}
 		The proof is a direct consequence of $\nabla^a_{v} Z_{i} = \nabla^N_{v} Z_{i} = 0$, $\nabla_v^{TH} v = 0$ and (\ref{eq_der_diff_form}).
 	\end{proof}
 	\begin{proof}[Proof of Proposition \ref{prop_diff_exp}]
 		The main idea is to fix a smooth function $u$ on $M$ and to use Lemmas \ref{lem_cal_1}, \ref{lem_cal_2} to get two different expressions for the Taylor expansion of $u(\psi(Z))$.
 		By comparing the two expansions, we will get the Taylor expansion (\ref{eq_prop_diff_exp}).
 		\par 
 		In one way, since $\psi(Z_H, Z_N) = \phi(h(Z_H, Z_N))$, we may apply Lemma \ref{lem_cal_1} to get
 		\begin{equation}\label{eq_h_t_exp0}
 			u(\psi(Z_H, Z_N)) 
 			=
 			\sum_{k = 0}^{\infty} \frac{1}{k!} \Big( h(Z_H, Z_N)^{\otimes k} \cdot (\nabla^{TM})^k u \Big) (y_0). 
 		\end{equation}
 		\par 
 		In another way, let us first apply Lemma \ref{lem_cal_1} to get
 		\begin{equation}\label{eq_h_t_exp1}
 			u \big(\psi(Z_H, Z_N)\big) = \sum_{k = 0}^{\infty} \frac{1}{k!} \Big( \tau_{Z_H}(Z_N)^{\otimes k} \cdot (\nabla^{TM})^k u \Big) \big(\psi(Z_H, 0)\big),
 		\end{equation}
 		where $\tau_{Z_H}(Z_N)$ is the parallel transport of $Z_N$ along the curve $\exp_{y_0}^{H}(tZ_H)$, $t \in [0, 1]$, with respect to $\nabla^N$.
 		We now apply Lemma \ref{lem_cal_1} again, but this time on $H$, to get
 		\begin{multline}\label{eq_h_t_exp2}
 			\Big( \tau_{Z_H}(Z_N)^{\otimes k} \cdot (\nabla^{TM})^{\otimes k} u \Big) \big(\psi(Z_H, 0)\big)
 			\\
 			=
 			 \sum_{l = 0}^{\infty} \frac{1}{l!}
 			  Z_H^{\otimes l} \cdot (\nabla^{TH})^{\otimes l}
 			 \Big( \tau_{Z_H}(Z_N)^{\otimes k} \cdot (\nabla^{TM})^{\otimes k} u \Big)(y_0). 
 		\end{multline}
 		Finally, let us apply Lemma \ref{lem_cal_2}, (\ref{eq_h_t_exp1}) and (\ref{eq_h_t_exp2}), to get
 		\begin{equation}\label{eq_h_t_exp3}
 			u(\psi(Z_H, Z_N)) = 
 			\sum_{l = 0}^{\infty} 
 			\sum_{k = 0}^{\infty} 
			\frac{1}{l!} 			
 			\frac{1}{k!} 
 			  \Big( (Z_H^{\otimes l} \otimes Z_N^{\otimes k}) \cdot (\nabla^a)^{\otimes l} (\nabla^{TM})^{\otimes k} u \Big)(y_0). 
 		\end{equation}
 		\par 
 		Now, we denote by $h^{[r]}(Z)$, $r \in \nat$, the homogeneous polynomial in $(Z)$ of degree $r$ such that the Taylor expansion of $h(Z)$ is given by $\sum_{r = 0}^{\infty} h^{[r]}(Z)$. Let us now take the homogeneous parts of (\ref{eq_h_t_exp0}) and compare them with (\ref{eq_h_t_exp3}). 
 		The comparison of the first degree gives us
 		\begin{equation}\label{eq_h_t_exp4}
			\big( h^{[1]}(Z) u \big)(y_0)
			=
			(Z u)(y_0).
 		\end{equation}
 		By comparing now the second degree, we get
 		\begin{multline}\label{eq_h_t_exp5}
 			\Big( h^{[2]}(Z_H, Z_N) \cdot \nabla^{TM} \Big) u(y_0)
 			+
 			\frac{1}{2}
 			\Big(
 			h^{[1]}(Z_H, Z_N)^{\otimes 2} \cdot (\nabla^{TM})^{\otimes 2} \Big) u(y_0)
 			\\
 			=
 		 	\Big( \frac{1}{2} Z_H^{\otimes 2} \cdot (\nabla^a)^{\otimes 2} 
 			+
 			Z_H \otimes Z_N \cdot \nabla^a \nabla^{TM}
 			+
 			\frac{1}{2} Z_N^{\otimes 2} \cdot (\nabla^{TM})^{\otimes 2} \Big) u(y_0).
 		\end{multline}
 		From (\ref{eq_sec_fund_f}), (\ref{eq_der_diff_form}), (\ref{eq_h_t_exp5}) and the fact that $\nabla^{TM}$ has no torsion, we deduce
 		\begin{equation}\label{eq_h_t_fin2}
 			h^{[2]}(Z_H, Z_N) = \frac{1}{2} A_{y_0}(Z_H)Z_H + A_{y_0}(Z_H)Z_N.
 		\end{equation}
 		From (\ref{eq_h_t_exp4}) and (\ref{eq_h_t_fin2}), we deduce (\ref{eq_prop_diff_exp}). 
 		\par 
 		Now, by the definition of curvature, we also have
 	 	\begin{multline}\label{eq_h_t_exp7}
 	 		U \otimes V \otimes W \cdot (\nabla^{TM})^{\otimes 3} 
 	 		=
 	 		U \otimes W \otimes V \cdot (\nabla^{TM})^{\otimes 3} 
 	 		\\
 	 		=
 	 		V \otimes U \otimes W \cdot (\nabla^{TM})^{\otimes 3} 
 	 		-
			R^{TM}(U, V)W \cdot \nabla^{TM}.
 	 	\end{multline}
 	 	From (\ref{eq_h_t_exp0}), (\ref{eq_h_t_exp3}) and (\ref{eq_h_t_exp7}), we see that the coefficients of $h^{[r]}(Z_H, Z_N)$ can be expressed in terms of $A, R^{TM}$, and their derivatives up to order $r-2$.
 	\end{proof}

\subsection{Holonomy along the paths adapted to the two coordinate systems}\label{sect_par_transport}
	This section is devoted to the comparison of two different trivializations of vector bundles, done in a neighborhood of a submanifold. 
	One is for the parallel transport adapted to the Fermi coordinate system, another one is for the geodesic coordinates.
	\par 
	We conserve the notations and assumptions from Section \ref{sect_coord_syst}.
	Let $(E, \nabla^E, h^{E})$ be a Hermitian vector bundle of bounded geometry and rank $r$ over $(M, g^{TM})$.
	We fix $y_0 \in H$ and an orthonormal frame $f_1, \ldots, f_r \in (E_{y_0}, h^E_{y_0})$. 
	Recall that in Section \ref{sect_bnd_geom_cf}, over $\mathbb{B}_{y_0}^{M}(R)$, for $R$ defined as in (\ref{eq_r_defn_const}), using the parallel transports, we defined two orthonormal frames $\tilde{f}'_1, \ldots, \tilde{f}'_r$ and $\tilde{f}_1, \ldots, \tilde{f}_r$.
	In (\ref{eq_frame_tilde}), we defined the matrix function $\xi_E$ which relates them. 
	\begin{prop}\label{prop_phi_fun_exp}
		The following bound holds
		\begin{equation}\label{eq_phi_fun_exp1}
			\xi_E(\psi(Z)) = O(|Z|^2).
		\end{equation}
		If, moreover, we assume that $(E, \nabla^E, h^E) := (L, \nabla^L, h^L)$ is a line bundle, and that there is a skew-adjoint endomorphism $Q$ of $TM$, which is parallel with respect to $\nabla^{TM}$ (i.e. $\nabla^{TM} Q = 0$), which commutes with $A$, the restriction of which to $H$ respects the decomposition (\ref{eq_tx_rest}), and such that for the curvature $R^L$ of $\nabla^L$, and for any $u, v \in TM$, we have
		\begin{equation}\label{eq_rl_q_op}
			\frac{\imun}{2 \pi}
			R^L (u, v) 
			=
			g^{TM}(Q u, v),
		\end{equation}
		then the following more precise bound holds
		\begin{equation}\label{eq_phi_fun_exp12}
			\xi_L(\psi(Z)) = -\frac{1}{6} R^L_{y_0} \big(Z, B(Z) \big) + O(|Z|^4).
		\end{equation}
		Moreover, the coefficients of order $r$, $r \in \nat$, in above Taylor expansions can be expressed in terms of $R^{TM}$, $R^E$, $R^L$, $A$, and their derivatives up to order $r-2$ with respect to $\nabla^{TM}$, $\nabla^{E}$, $\nabla^{L}$, at $y_0$.
	\end{prop}
	\begin{rem}
		Assume $(M, g^{TM})$ is endowed with a complex structure $J$, and $g^{TM}$ is invariant under the action of it. 
		Assume, moreover, that (\ref{eq_rl_q_op}) holds for $Q := J$ as in (\ref{eq_gtx_def}). 
		Then all the requirements are satisfied for $Q := J$.
		This is because the invariance of $g^{TM}$ under the action of $J$, and the fact that $H$ is a complex submanifold, imply that the restriction of $J$ to $H$ respects the decomposition (\ref{eq_tx_rest}).
		Also (\ref{eq_rl_q_op}) implies that $(M, J, g^{TM})$ is Kähler, which implies that $J$ is parallel with respect to $\nabla^{TM}$, cf. \cite[Theorem 1.2.8]{MaHol}.
		Finally, by the definition of $A$ and the fact that $J|_H$ preserves $TH$ and $N$, we see that $J$ commutes with $A$.
		Proposition \ref{prop_phi_fun_exp} will only be applied in this particular situation.
	\end{rem}
	\par 
	The proof is given in the end of this section. 
	Before, let us state some auxiliary results.
	\begin{lem}[{\cite[Lemma 1.2.3]{MaHol} or \cite[p. 38]{BGV}}]\label{lem_comp_phi_par}
		Let $\tilde{t}'_i$, $i = 1, \ldots, n$ be the vector fields, constructed by the parallel transport of $e_i$ with respect to $\nabla^{TM}$ along the curve $\phi(tZ)$, $t \in [0, 1]$, $Z \in T_{y_0}M$, $|Z| < r_M$.
 		Then
 		\begin{equation}\label{eq_par_tr_coord_comp}
 			\frac{\partial \phi}{\partial Z_i} = \tilde{t}'_{i} + \sum_{j = 1}^{n} O(|Z|^2) \tilde{t}'_{j},
 		\end{equation}
 		Moreover, the coefficients of order $r$, $r \in \nat$, of the Taylor expansion (\ref{eq_par_tr_coord_comp}) can be expressed in terms of $R^{TM}$, and their derivatives up to order $r-2$ with respect to $\nabla^{TM}$, evaluated at $y_0$.
	\end{lem}
	We denote by $e_H, e'_H \in TM$ (resp. $e_N, e'_N$) the vector fields $\frac{\partial \psi}{\partial Z_i}$ for $i = 1, \ldots, m$ (resp.  $i = m+1, \ldots, n$).
 	\begin{lem}\label{lem_scal_prod}
 		For any endomorphism $T$ of $TM$, which is parallel with respect to $\nabla^{TM}$, which commutes with $A$, the restriction of which to $H$ respects the decomposition (\ref{eq_tx_rest}), we have
 		\begin{equation}\label{eq_scal_prod_1}
 		\begin{aligned}
 			& g^{TM}_{\psi (Z_H, Z_N)}(T e_N, e'_N) = g^{TM}_{y_0}(T e_N, e'_N) + O(|Z|^2), 
 			\\
 			& g^{TM}_{\psi (Z_H, Z_N)}(T e_N, e_H)= O(|Z|^2).	
 			\\
 			& g^{TM}_{\psi (Z_H, Z_N)}(T e_H, e'_H)= g^{TM}_{y_0}(T e_H, e'_H) + g^{TM}_{y_0} \big(T A(e_H)Z_N, e'_H \big) 
 			\\
 			&
 			\qquad \qquad \qquad \qquad \qquad \qquad \qquad \qquad 
 			+ g^{TM}_{y_0} \big(Te_H , A(e'_H)Z_N  \big)  + O(|Z|^2).			
 		\end{aligned}
 		\end{equation}
 		Also, the following identity holds
 		\begin{equation}\label{eq_scal_prod_12}
 			 g^{TM}_{\psi (Z_H, 0)}(T e_N, e'_N) =  g^{TM}_{y_0}(T e_N, e'_N).
 		\end{equation}
		Moreover, the coefficients of order $r$, $r \in \nat$, in above Taylor expansions can be expressed in terms of $R^{TM}$, $A$, their derivatives up to order $r-2$ with respect to $\nabla^{TM}$, $\nabla^{L}$, evaluated at $y_0$, and endomorphism $T_{y_0}$.
 	\end{lem}
 	\begin{rem}
 		For $H = \{y_0\}$ and $T = {\rm{Id}}$, the result follows from \cite[Proposition 1.28]{BGV}.
 		Similar calculations were done by Ma-Zhang in \cite{MaZhBKSR} and Bismut-Lebeau \cite{BisLeb91}.
 	\end{rem}
 	\begin{proof}
 		First of all, since $T$ is parallel with respect to $\nabla^{TM}$, we see that
 		\begin{equation}
 			\frac{\partial}{\partial e_H} g^{TM}_{\psi (Z_H, 0)}(T e_N, e'_N)
 			=
 			g^{TM}_{\psi (Z_H, 0)}(T \nabla^{TM}_{e_H}e_N, e'_N)
 			+
 			g^{TM}_{\psi (Z_H, 0)}(T e_N, \nabla^{TM}_{e_H} e'_N).
 		\end{equation}
 		However, by the definition of $\psi$, see (\ref{eq_defn_fermi}), the restrictions of $e_N$, $e'_N$ to $H$ are parallel with respect to $\nabla^N$.
 		By this, the fact that the form $A$ exchanges $N$ and $TH$, and the endomorphism $T$ preserves the decomposition (\ref{eq_tx_rest}), we deduce that
 		$
 			\frac{\partial}{\partial e_H} g^{TM}_{\psi (Z_H, 0)}(T e_N, e'_N) = 0,
 		$
 		which readily implies (\ref{eq_scal_prod_12}).
 		\par 
 		Recall that the vector fields $\tilde{t}'_i$, $i = 1, \ldots, n$, were introduced in Lemma \ref{lem_comp_phi_par}.
 		As $T$ is parallel, $T \tilde{t}'_i$ is equal to the parallel transport of $T e_i$, along the same curve as the one used in the definition of $\tilde{t}'_i$.
 		From this and (\ref{eq_par_tr_coord_comp}), we deduce that
 		\begin{equation}\label{eq_par_tr_coord_comp22}
 			g^{TM}_{\psi (Z_H, Z_N)}\Big( T \frac{\partial \phi}{\partial Z_i}, \frac{\partial \phi}{\partial Z_j} \Big)
 			=
 			g^{TM}_{y_0}\big( T e_i, e_j \big)
 			+
 			O(|Z|^2).
 		\end{equation}
 		\par 
 		Now, we see that Proposition \ref{prop_diff_exp}, along with the fact that $A$ exchanges $N$ with $TH$, implies that for $i = 1, \ldots, m$ and $j = m + 1, \ldots, n$, we have
 		\begin{equation}\label{eq_psi_phi_vect_f_com}
 		\begin{aligned}
 			& 
 			\frac{\partial \psi}{\partial Z_j}
 			=
 			\frac{\partial \phi}{\partial Z_j}
 			+
 			\sum_{l = 1}^{m}
 			\frac{\partial \phi}{\partial Z_l}
 			g^{TM}_{y_0}
 			\big(
 				A(
 				Z_H
 				) e_j,
 				e_l
 			\big)
 			+
 			O(|Z|^2),
 			\\
 			&
 			\frac{\partial \psi}{\partial Z_i}
 			=
 			\frac{\partial \phi}{\partial Z_i}
 			+
 			\sum_{l = 1}^{m}
 			\frac{\partial \phi}{\partial Z_l}
 			g^{TM}_{y_0}
 			\big(
 				A(
 					e_i
 				)Z_N
 				,
 				e_l
 			\big)
 			\\
 			& \qquad \qquad \qquad \qquad \qquad
 			+
 			\sum_{l = m + 1}^{n}
 			\frac{\partial \phi}{\partial Z_l}
 			g^{TM}_{y_0}
 			\big(
 				A\big(
 					e_i
 				\big)Z_H,
 					e_l
 			\big)
 			+
 			O(|Z|^2).
 		\end{aligned}
 		\end{equation}
 		From this, Proposition \ref{prop_diff_exp} and (\ref{eq_par_tr_coord_comp22}), we establish the first equation in (\ref{eq_scal_prod_1}) and the remark after it.
 		We also deduce that 
 		\begin{multline}
 			g^{TM}_{\psi (Z_H, Z_N)}(T e_N, e_H)
 			=
 			g^{TM}_{y_0}(T e_N, e_H)
 			\\
 			+
 			g^{TM}_{y_0}(T A(Z_H)e_N, e_H)
 			+
 			g^{TM}_{y_0}(T e_N, A(e_H) Z_H)
 			+
 			O(|Z|^2).
 		\end{multline}
 		Now, since both $\nabla^{TM}$ and $\nabla^{TH}$ have no torsion, we have
 		\begin{equation}\label{eq_a_no_tors}
 			A(u) v = A(v) u, \quad \text{for any} \quad u, v \in TH.
 		\end{equation}
 		From (\ref{eq_a_no_tors}) and the fact that $A$ is skew-adjoint and commutes with $T$, we deduce the second part of (\ref{eq_scal_prod_1}).
 		The third part of (\ref{eq_scal_prod_1}) follows directly from (\ref{eq_par_tr_coord_comp22}) and (\ref{eq_psi_phi_vect_f_com}).
 	\end{proof}
 	\par 
 	Recall that the projection $\pi_0$ and the identification of $E$ to $\pi_0^* (E|_H)$ in a tubular neighborhood $U$ of $H$ were defined before (\ref{eq_ext0_op}).
	We define the $1$-form $\Gamma^{E}$ with values in $\enmr{\pi_0^* (E|_H)}$ as follows
		\begin{equation}\label{eq_gamma_l_def}
			\Gamma^{E} = \nabla^E - \pi_0^*(\nabla^E|_{H}),
		\end{equation}
	where we implicitly used the above isomorphism.
	We introduce similar notations for $L$.
	\begin{lem}\label{lem_gamma_exp}
		Under the same assumptions as in Proposition \ref{prop_phi_fun_exp}, the following holds
		\begin{equation}\label{eq_gamma_exp_ref}
		\begin{aligned}
			& \Gamma^{E}_{\psi (Z)}(e_N) = \frac{1}{2} R^E_{\pi_0(\psi (Z))}(Z_N, e_N) + O(|Z|^2), 
			\, \Gamma^{E}_{\psi(Z)}(e_H) = R^E_{\pi_0(\psi (Z))}(Z_N, e_H) + O(|Z|^2),
			\\
			& \Gamma^{L}_{\psi(Z)}(e_N) = \frac{1}{2} R^L_{\pi_0(\psi (Z))}(Z_N, e_N) + O(|Z|^3), 
			\, \Gamma^{L}_{\psi(Z)}(e_H) = O(|Z|^3).
		\end{aligned}
		\end{equation}
		Moreover, the coefficients of order $r$, $r \in \nat$, in above Taylor expansions can be expressed in terms of $R^{TM}$, $R^L$, $R^E$, $A$, and their derivatives up to order $r-2$ with respect to $\nabla^{TM}$, $\nabla^L$, $\nabla^E$, $\nabla^a$, evaluated at $y_0$.
	\end{lem}
	\begin{rem}
		Similar identities were obtained in \cite[(13.65)-(13.66)]{BisLeb91}.
	\end{rem}
	\begin{proof}
		We follow closely the proof from \cite[Proposition 1.18]{BGV}, cf. \cite[Lemma 1.2.4]{MaHol}, which establishes Lemma \ref{lem_gamma_exp} for $H = \{ y_0 \}$.
		First of all, by (\ref{eq_gamma_l_def}), we have
		\begin{equation}\label{eq_lem_1}
			(\pi_0^*(\nabla^E|_{H}) \Gamma^E) + \Gamma^E \wedge \Gamma^E + \pi_0^* ( R^E|_{H})
			=
			R^E.
		\end{equation}
		\par 
		We denote by $\mathcal{R} \in \ccal^{\infty}(N, TN)$ the tautological section of $TN$.
		As we identified a neighborhood $U$ of $H$ in $M$ with a neighborhood of the zero section in $N$, we may look at $\mathcal{R}$ as at the vector in the tangent space of $U$. If we write it in Fermi coordinates, we get
		\begin{equation}\label{eq_lem_3}
			\mathcal{R} = 
			\sum_{i = m+1}^{n} Z_i \frac{\partial}{\partial Z_i}.
		\end{equation}
		By our choice of the trivialization, $\iota_{\mathcal{R}} \Gamma^E = 0$. 
		Also, $\iota_{\mathcal{R}} (\pi_0^* R^E|_{H}) = 0$. Hence, by (\ref{eq_lem_1}), we get
		\begin{equation}\label{eq_lem_2}
			L_{\mathcal{R}} \Gamma^E = \Big[ \iota_{\mathcal{R}}, \pi_0^*(\nabla^E|_{H}) \Big] \Gamma^E = \iota_{\mathcal{R}} R^E.
		\end{equation}
		From (\ref{eq_lem_3}), for $i = 1, \ldots, m$, and $j = m+1, \ldots, n$, we also obtain the following identities
		\begin{equation}\label{eq_lem_4}
			L_{\mathcal{R}} dZ_i = 0, \qquad L_{\mathcal{R}} dZ_j = dZ_j.
		\end{equation}
		\par 
		Using (\ref{eq_lem_4}) and expanding both sides of (\ref{eq_lem_2}) in Taylor series at $Z_N = 0$, we get
		\begin{equation}\label{eq_lem_5}
		\begin{aligned}
			&\sum_{\alpha} (|\alpha| + 1) (\partial^{\alpha} \Gamma^E)_{\psi(Z_H, 0)} ( e_N ) \frac{Z_N^{\alpha}}{\alpha!}
			=
			\sum_{\alpha} (\partial^{\alpha} R^E)_{\psi(Z_H, 0)} ( \mathcal{R},  e_N ) \frac{Z_N^{\alpha}}{\alpha!},
			\quad 
			\text{ }
			\\
			&\sum_{\alpha} |\alpha| (\partial^{\alpha} \Gamma^E)_{\psi(Z_H, 0)} ( e_H ) \frac{Z_N^{\alpha}}{\alpha!}
			=
			\sum_{\alpha} (\partial^{\alpha} R^E)_{\psi(Z_H, 0)} ( \mathcal{R},  e_H ) \frac{Z_N^{\alpha}}{\alpha!}.
		\end{aligned}
		\end{equation}
		From (\ref{eq_lem_5}), we deduce the first two equations of Lemma \ref{lem_gamma_exp}.
		From the fact that $N$ is orthogonal to $TH$, (\ref{eq_rl_q_op}) and the fact that the restriction of $Q$ to $H$ respects the decomposition (\ref{eq_tx_rest}), we see that $R^L_{\psi (Z_H, Z_N)} ( e_N,  e_H) = O(|Z|)$.
		The last two equations of Lemma \ref{lem_gamma_exp} follow directly from Lemma \ref{lem_scal_prod}, (\ref{eq_rl_q_op}) and (\ref{eq_lem_5}), applied for $E := L$.
	\end{proof}
	\begin{proof}[Proof of Proposition \ref{prop_phi_fun_exp}]
		We denote by $\nabla$ the standard covariant derivative in $E$, defined locally in the frame $\tilde{f}_1, \ldots, \tilde{f}_r$. 
		We define the $1$-form $\Gamma^E_{0}$ with values in $\enmr{E_{y_0}}$ as follows
		\begin{equation}
			\Gamma^E_{0} := \nabla^E - \nabla.
		\end{equation}
		By (\ref{eq_gamma_l_def}), we have the following identity
		\begin{equation}\label{eq_gamma_0_ident}
			\Gamma^E_{0, \psi (Z_H, Z_N)}
			=
			\pi_0^* ( \Gamma^E_{0, \psi (Z_H, 0)})
			+
			\Gamma^E_{\psi (Z_H, Z_N)},
		\end{equation}
		where we implicitly identified $\Gamma^E_{\psi (Z_H, Z_N)}$ with a $1$-form with values in $\enmr{E_{y_0}}$ using $\tilde{f}_1, \ldots, \tilde{f}_r$.
		By (\ref{eq_rl_q_op}) and Lemma \ref{lem_gamma_exp}, applied once for the second summand in the right-hand side of (\ref{eq_gamma_0_ident}), and once for $M := H, H := y_0$, to treat the first summand, we see that for $e := \frac{\partial \psi}{\partial Z_i}$, $i = 1, \ldots, n$,
		\begin{equation}\label{eq_gamma_exp_ref2}
		\begin{aligned}
			& \Gamma^L_{0, \psi (Z)}(e) = \frac{1}{2} \Big( R^L_{y_0}(Z_H, e) + R^L_{\psi (Z_H, 0)}(Z_N, e) \Big) + O(|Z|^3),  
			\\
			& \Gamma^E_{0, \psi (Z)}(e) = O(|Z|).
		\end{aligned}
		\end{equation}
		Now, by (\ref{eq_scal_prod_12}) and (\ref{eq_gamma_exp_ref2}), we conclude that 
		\begin{equation}\label{eq_gamma_exp_ref2222}
			\Gamma^L_{0, \psi (Z)}(e) = \frac{1}{2} R^L_{y_0}(Z, e) + O(|Z|^3),  
		\end{equation}
		\par 
		Now, we denote by $\gamma(t)$ the geodesic $\phi (t h(Z))$, $t \in [0, 1]$. By the definition of $h$, we have $\gamma(1) = \psi(Z)$. 
		From Proposition \ref{prop_diff_exp}, we easily see
		\begin{equation}\label{eq_gamma_in_psi_ccc}
			\psi^{-1}(\gamma(t)) = h^{-1}(t h(Z)) 
			= tZ + (t - t^2) B(Z) + O(|Z|^3).
		\end{equation}
		\par 
		Directly from the definition of $\xi_E$, the fact that $\nabla^E_{\gamma'(t)} \tilde{f}'_i = 0$, (\ref{eq_frame_tilde}), and the usual law of derivation of the exponential, we have
		\begin{equation}\label{eq_phi_der_11}
			\int_0^{1}
			\exp \big((1 - s) \xi_E(\gamma(t)) \big)
			\cdot
			\Big(
			\frac{\partial}{\partial t} \xi_E(\gamma(t))
			\Big)
			\cdot
			\exp \big((s - 1) \xi_E(\gamma(t)) \big)
			ds
			=
			\Gamma^E_{0, \gamma(t)}(\gamma'(t)).
		\end{equation}
		From (\ref{eq_phi_der_11}) and the last equation from (\ref{eq_gamma_exp_ref2}), we deduce (\ref{eq_phi_fun_exp1}).
		The statement about the coefficients of the Taylor expansions readily follows from (\ref{eq_phi_der_11}) and the respective statement from Proposition \ref{prop_phi_fun_exp}.
		\par 
		Now, it is only left to establish (\ref{eq_phi_fun_exp12}).
		Since now $E := L$ is a line bundle, (\ref{eq_phi_der_11}) simplifies to
		\begin{equation}\label{eq_phi_der_111}
			\frac{\partial}{\partial t} \xi_L(\gamma(t))
			=
			\Gamma^L_{0, \gamma(t)}(\gamma'(t)).
		\end{equation}
		From (\ref{eq_gamma_in_psi_ccc}), we deduce the following asymptotics
		\begin{equation}\label{eq_phi_der_12}
			\int_0^{1} R^L_{y_0} \Big( \psi^{-1}(\gamma(t)), \frac{\partial \psi^{-1}(\gamma(t))}{\partial t} \Big) dt 
			=
			- \frac{1}{3} R^L_{y_0} \Big(Z, B(Z) \Big) + O(|Z|^4).
		\end{equation}
		From  (\ref{eq_gamma_exp_ref2222}), (\ref{eq_phi_der_111}) and (\ref{eq_phi_der_12}), we deduce the result. The statement about the form of the coefficients follows from the respective statements in Proposition \ref{prop_diff_exp} and Lemma \ref{lem_gamma_exp}.
	\end{proof}
	
\subsection{Complex structure in Fermi coordinates and quasi-plurisubharmonicity}\label{sect_cmplx_fermi}
	The main goal of this section is to give for Kähler manifolds of bounded geometry an approximate formula for the complex structure in Fermi coordinates.
	As a consequence, we establish quasi-plurisubharmonicity of several functions, crucial for certain $L^2$-estimates.
	\par 
	More precisely, assume $(X, Y, g^{TX})$ is a Kähler triple of bounded geometry.
	We fix a point $y_0 \in Y$ and an orthonormal frame $(e_1, \ldots, e_{2m})$ (resp. $(e_{2m+1}, \ldots, e_{2n})$) in $(T_{y_0}Y, g_{y_0}^{TY})$ (resp. in $(N_{y_0}, g^{N}_{y_0})$), satisfying (\ref{eq_cond_jinv}).
	Recall that in (\ref{eq_defn_fermi}) and (\ref{eq_phi_defn}), we defined two coordinate systems $\psi_{y_0}$, $\phi_{y_0}$ in a neighborhood of $y_0$.
	We denote by $J$ the complex structure of $X$, and by $\textbf{J} = (\textbf{J}_{ij})_{i,j = 1}^{2n}$ its coordinates with respect to the basis $\frac{\partial \psi}{\partial Z_i}$, $i = 1, \ldots, 2n$.
	It is a function, defined in $\mathbb{B}_0^{\real^{2n}}(R)$, where $R$ is as in (\ref{eq_r_defn_const}), with matrix values of size $2n \times 2n$.
	We write $\textbf{J}$ in a block form $\big(\begin{smallmatrix}
		  \textbf{J}_0 &  \textbf{J}_1 \\
		  \textbf{J}_2 &  \textbf{J}_3
		\end{smallmatrix}\big)$,
	where $\textbf{J}_0$ has size $2m \times 2m$.
	The first result of this section goes as follows.
	\begin{lem}\label{lem_cmplx_str_exp}
		The matrix  $\textbf{J}$ has the following Taylor expansion
		\begin{equation}\label{eq_cmplx_str_exp}
			\textbf{J} = J^0 
			+
			 \bigg(
			 \begin{matrix}
			  J^1 & 0\\ 
			  0 & 0
			\end{matrix}
			\bigg)
			+
			 O \big( |Z|^2 \big),
		\end{equation}
		where $J^0$ is the diagonal block matrix with blocks
		$\big(\begin{smallmatrix}
		  0 & -1\\
		  1 & 0
		\end{smallmatrix}\big)$,
		and $J^1 = (J^1_{ij})_{i,j = 1}^{2m}$ is given by 
		\begin{equation}
			J^1_{ij} 
			=
			-2
			 g^{TX}_{y_0}(A(Je_i) Z_N, e_j),
		\end{equation}
		for $A$ defined in (\ref{eq_sec_fund_f}).
		Moreover, the constant in the $O$-terms is uniform on $y_0 \in Y$, and depends only on $C_1$ from the bounds (\ref{eq_bnd_curv_tm}), (\ref{eq_bnd_a_ck}) for $k = 1$.
	\end{lem}
	\begin{proof}
		The proof is essentially based on Lemma \ref{lem_comp_phi_par} and (\ref{eq_psi_phi_vect_f_com}).
		Recall that the vector fields $\tilde{t}'_i$, $i = 1, \ldots, 2n$, were constructed in Lemma \ref{lem_comp_phi_par} by the parallel transport of $e_i$ with respect to $\nabla^{TX}$ along the curve $\phi(tZ)$, $t \in [0, 1]$, $Z \in T_{y_0}X$.
		As $(X, g^{TX})$ is Kähler, by \cite[Theorem 1.2.8]{MaHol}, $J$ is parallel with respect to $\nabla^{TX}$.
		This means that by (\ref{eq_cond_jinv}), for $i = 1, \ldots, n$, we have 
		\begin{equation}\label{eq_cond_jinvtvect}
			J \tilde{t}'_{2i - 1} = \tilde{t}'_{2i}.
		\end{equation}
		From (\ref{eq_par_tr_coord_comp}) and (\ref{eq_cond_jinvtvect}), we conclude that
		\begin{equation}\label{eq_jphi}
			J \frac{\partial \phi}{\partial Z_{2i - 1 }}
			=
			\frac{\partial \phi}{\partial Z_{2i}}
			+
			O(|Z|^2).
		\end{equation}
		Remark that (\ref{eq_jphi}) already implies Lemma \ref{lem_cmplx_str_exp} for $Y := \{x_0\}$.
		Now, let us establish the general case.
		By the fact that $A$ takes values in skew-adjoint matrices and commutes with $J$, the $J$-invariance of $g^{TX}$ and (\ref{eq_a_no_tors}), we deduce that for $i = 1, \ldots, 2m$, $j = 2m + 1, \ldots, 2n$, $l = 1, \ldots, 2n$, for $Z_Y \in \real^{2m}$, $Z_N \in \real^{2(n-m)}$, we have
		\begin{equation}\label{eq_gtx_commut}
		\begin{aligned}
			&
			g^{TX}_{y_0}
 			\big(
 				A(
 				Z_Y
 				) e_j,
 				e_l
 			\big)
 			=
 			g^{TX}_{y_0}
 			\big(
 				A(
 				Z_Y
 				) J e_j,
 				J e_l
 			\big),
 			\\
 			&
			g^{TX}_{y_0}
 			\big(
 				A(
 				e_i
 				) Z_Y,
 				e_l
 			\big)
 			=
 			g^{TX}_{y_0}
 			\big(
 				A(
 				J e_i
 				) Z_Y,
 				J e_l
 			\big),
 			\\
 			&
			g^{TX}_{y_0}
 			\big(
 				A(
 				e_i
 				) Z_N,
 				e_l
 			\big)
 			=
 			-
 			g^{TX}_{y_0}
 			\big(
 				A(
 				J e_i
 				) Z_N,
 				J e_l
 			\big),
 		\end{aligned}
		\end{equation}
		By (\ref{eq_psi_phi_vect_f_com}), (\ref{eq_jphi}) and (\ref{eq_gtx_commut}), we deduce (\ref{eq_cmplx_str_exp}).
	\end{proof}
	\par 
	Now, recall that a function $f : X \to [-\infty, +\infty[$ on a complex Hermitian manifold $(X, \omega)$ is called quasi-plurisubharmonic if it is upper-semicontinuous, and there is a constant $C \in \real$, such that the following inequality holds in the distributional sense
	\begin{equation}\label{eq_distr_ineq}
		\imun \partial \dbar f \geq -C \omega.
	\end{equation}
	We denote by ${\rm{PSH}}(X, C \omega)$ the set of quasi-plurisubharmonic functions $f$, verifying (\ref{eq_distr_ineq}).
	\par 
	Recall that by standard properties of the plurisubharmonic functions, cf. \cite[Theorem I.5.6]{DemCompl}, for any convex function $\chi : \real \to \real$, verifying $0 \leq \chi' \leq 1$, $f \in {\rm{PSH}}(X, C \omega)$, $C \geq 0$, we have
	\begin{equation}\label{eq_chi_f_psh}
		\chi \circ f \in {\rm{PSH}}(X, C \omega) \quad \text{for any} \quad f \in {\rm{PSH}}(X, C \omega).
	\end{equation}
	\par 
	Later on, on several occasions, we will need the following lemma.
	\begin{lem}[{\cite[Proposition 4.2]{HarvLawsPotTh}}]\label{lem_harv_law_form}
		Suppose that $J$ is an almost complex structure on an open subset $\Omega \subset \real^{2n}$. 
		Let $v$ be a constant coefficient vector field on $\Omega$ (i.e. $v = \sum v_i \frac{\partial}{\partial Z_i}$ where $v_i$ are constants).
		Then for any smooth function $f : \Omega \to \real$, the following identity holds
		\begin{multline}\label{eq_harv_law_form}
			\partial \dbar f
			\Big(
				v - \imun Jv,
				v + \imun Jv
			\Big)
			=
			(D^2 f)(v, v)
			+
			(D^2 f)(Jv, Jv)
			\\
			+
			(D f)
			\Big(
				D_{Jv}J (v) - D_vJ (Jv)
			\Big),
		\end{multline}
		where $D^2 f$ is the standard double derivative of $f$, and $D_uJ$ is the derivative in the direction $u \in \real^{2n}$, of the almost complex structure, written in the standard coordinates of $\real^{2n}$.
	\end{lem}
	\par 
	For any $x_0 \in X$, we consider the function $\alpha_{x_0} : \mathbb{B}_{x_0}^{X}(r_X) \to \real$, given by 
	\begin{equation}\label{eq_defn_alpha}
		\alpha_{x_0}(x) := \dist_X(x_0, x)^2. 
	\end{equation}
	The proof of the following lemma is a direct consequence of Lemmas \ref{lem_cmplx_str_exp} and \ref{lem_harv_law_form}.
	\begin{lem}\label{prop_plurisub_pt}
		Assume that a Kähler manifold $(X, g^{TX})$ is of bounded geometry. 
		Then there is $r_0 > 0$, such that for any $x_0 \in X$, $\alpha_{x_0} \in {\rm{PSH}}(\mathbb{B}_{x_0}^{X}(r_0), - \frac{1}{2} \omega)$.
		Moreover, $r_0$ depends only on $C_1$ from the bound (\ref{eq_bnd_curv_tm}) and $r_X$.
	\end{lem}
	\par 
	Assume that a Kähler triple $(X, Y, g^{TX})$ is of bounded geometry. 
	Let us consider $\delta_Y : X \setminus Y \to \real$, $\alpha_Y : X \to \real$,  defined as
	\begin{equation}\label{eq_delta_defn_y}
	\begin{aligned}
		&
		\delta_Y(x) := \log \big(\dist_X(x, Y) \big) \cdot \rho 
		\Big(
			 \frac{\dist_X(x, Y)}{r_{\perp}} 
		\Big),
		\\
		&
		\alpha_Y(x) := \dist_X(x, Y)^2 \cdot \rho 
		\Big(
			 \frac{\dist_X(x, Y)}{r_{\perp}} 
		\Big)
		+
		\Big(
		1
		-
		\rho 
		\Big(
			 \frac{\dist_X(x, Y)}{r_{\perp}} 
		\Big)
		\Big)
		,
	\end{aligned}
	\end{equation}
	where $\rho$ was defined in (\ref{defn_rho_fun}).
	Remark that since $\dist_X(x, Y)^2$ is smooth over $\mathbb{B}_Y^X(r_{\perp})$, both functions $\delta_Y$ and $\alpha_Y$ are smooth.
	\begin{lem}\label{thm_plurisub}
		There is $C > 0$, such that $\delta_Y \in {\rm{PSH}}(X, C \omega)$, $\alpha_Y \in {\rm{PSH}}(X, C \omega)$, $-\alpha_Y \in {\rm{PSH}}(X, C \omega)$.
		Moreover, $C$ depends only on $C_1$ from the bounds (\ref{eq_bnd_curv_tm}), (\ref{eq_bnd_a_ck}) and $r_{\perp}$.
	\end{lem}
	\begin{proof}
		The proofs for $\delta_Y$, $\alpha_Y$, $-\alpha_Y$ are similar, so we concentrate on $\delta_Y$, which is slightly more difficult than the other two.
		From Lemma \ref{lem_cmplx_str_exp} and bounded geometry condition, we see that $\imun \partial \dbar \delta_Y$ is uniformly bounded away from in $\mathbb{B}_Y^X(\frac{r_{\perp}}{4})$.
		So it is enough to study this quantity only on $\mathbb{B}_Y^X(\frac{r_{\perp}}{4})$.
		We fix a point $y_0 \in Y$, consider Fermi coordinates $\psi_{y_0}$, and introduce the vector field
		\begin{equation}\label{eq_vvect_dec}
			v 
			=
			\sum_{i = 1}^{2m}
			a_i
			\frac{\partial \psi}{\partial Z_i}
			+
			\sum_{j = 2m + 1}^{2n}
			b_j
			\frac{\partial \psi}{\partial Z_j},
		\end{equation}
		where $a_i, b_j \in \comp$, $i = 1, \ldots, 2m$, $j = 2m + 1, \ldots, 2n$, are certain constants, $\sum |a_i|^2 + \sum |b_j|^2  = 1$.
		It is enough to verify that there is a constant $C > 0$, as described in the statement of the proposition we're proving, such that for any choice of $a_i$, $b_j$ above, over $\{ \psi_{y_0}(Z_N) : |Z_N| < \frac{r_{\perp}}{4} \}$, we have
		\begin{equation}\label{eq_j_asymp0}
			\partial \dbar \delta_Y
			\Big(
				v - \imun Jv,
				v + \imun Jv
			\Big)
			\geq
			- C.
		\end{equation}
		\par 
		Let us now calculate each term on the right-hand side of (\ref{eq_harv_law_form}) up to negligible terms.
		First of all, remark that in Fermi coordinates, over $\mathbb{B}_Y^X(\frac{r_{\perp}}{4})$, $\delta_Y$ has particularly simple form
		\begin{equation}
			\delta_Y(Z) = \log |Z_N| = \frac{1}{2} \log \Big( \sum_{j = 2m + 1}^{2n} |Z_j|^2 \Big).
		\end{equation}
		Hence, by (\ref{eq_vvect_dec}), we deduce
		\begin{equation}\label{eq_j_asymp1}
			(D^2 \delta_Y)(v, v)
			=
			\frac{1}{|Z_N|^4} \Big( 	
				\Big( \sum_{j = 2m + 1}^{2n} |b_j|^2 \Big)
				|Z_N|^2
				-
				2
				\Big( \sum_{j = 2m + 1}^{2n} b_j Z_j \Big)^2
			\Big).
		\end{equation}
		A similar calculation, using Lemma \ref{lem_cmplx_str_exp}, gives that over $\{ \psi_{y_0}(Z_N) : |Z_N| <  \frac{r_{\perp}}{4}  \}$, we have
		\begin{multline}\label{eq_j_asymp2}
			(D^2 \delta_Y)(Jv, Jv)
			=
			\\
			\frac{1}{|Z_N|^4} \Big( 	
				\Big( \sum_{j = 2m + 1}^{2n} |b_j|^2 \Big)
				|Z_N|^2
				-
				2
				\Big( \sum_{j = m + 1}^{n} \big( b_{2 j - 1} Z_{2j} - b_{2 j} Z_{2j - 1} \big) \Big)^2
			\Big)
			+
			O(1),
		\end{multline}
		where the $O$-term can be bounded uniformly in terms of $C_1$ from the bounds (\ref{eq_bnd_curv_tm}), (\ref{eq_bnd_a_ck}) and $r_{\perp}$.
		From Lemma \ref{lem_cmplx_str_exp}, we see that the last two terms in (\ref{eq_harv_law_form}) are negligible.
		From this (\ref{eq_harv_law_form}), (\ref{eq_j_asymp1}) and (\ref{eq_j_asymp2}), we see that to deduce (\ref{eq_j_asymp0}), it is now only left to prove the bound
		\begin{equation}\label{eq_j_asymp4}
			\Big( \sum_{j = 2m + 1}^{2n} |b_j|^2 \Big) |Z_N|^2
			\geq 
			\Big( \sum_{j = m + 1}^{n} \big( b_{2 j - 1} Z_{2j} - b_{2 j} Z_{2j - 1} \big) \Big)^2
			+
			\Big( \sum_{j = 2m + 1}^{2n} b_j Z_j \Big)^2.
		\end{equation}
		But (\ref{eq_j_asymp4}) follows easily from the fact that the vectors $(Z_{2m + 2}, - Z_{2m + 1}, \ldots,  Z_{2n}, -Z_{2n-1})$ and $(Z_{2m + 1}, \ldots, Z_{2n})$  are orthogonal.
	\end{proof}

	\subsection{Existence of uniform Stein atlases and related extension theorems}\label{sect_stein_atl}
	The main goal of this section is to introduce and study manifolds and submanifolds with uniform Stein atlases.
	In particular, we prove that any Kähler manifold of bounded geometry admits a uniform Stein atlas and we make a relation between the existence of a uniform Stein atlas and the existence of holomoprhic coordinates, defined on sufficiently big geodesic balls.
	\par 
	Before this, recall that by Cauchy formula, the $L^2$-bound on holomoprhic functions on the complex plane implies $\ccal^k$-bounds for any $k \in \nat$.
	As the following lemma shows, an analogous statement holds for general Kähler manifolds of bounded geometry.
	\begin{lem}\label{lem_ell_reg_sob_bnd_g}
		Assume that a Kähler manifold $(X, g^{TX})$ and  a Hermitian vector bundle $(E, h^E)$ over it have bounded geometry.
		Then for any $k \in \nat$, $r_1 > r_0 > 0$, $r_1 < \frac{r_X}{2}$, there is $C > 0$, such that for any $x_0 \in X$, for any $f \in H^0_{(2)}(\mathbb{B}_{x_0}^{X}(r_1), E)$, we have
		\begin{equation}
			\|
				f
			\|_{\ccal^k(\mathbb{B}_{x_0}^{X}(r_0))}
			\leq
			C
			\|
				f
			\|_{L^2(\mathbb{B}_{x_0}^{X}(r_1))}.
		\end{equation}
		Moreover, $C$ depends only on $r_0, r_1$ and $C_{k + n + 4}$ from (\ref{eq_bnd_curv_tm}) and (\ref{eq_bnd_re_ck}).
	\end{lem}
	\begin{proof}
		It is a direct consequence of Sobolev embedding and elliptic estimates, cf. \cite[(A.1.15), (A.1.17)]{MaHol}.
	\end{proof}
	\par 
	Recall that a complex manifold is said to be \textit{Stein} if it is holomorphically convex, its global holomorphic functions separate points and give local coordinates at every point, cf. \cite[Definition 1.6.16]{DemCompl}.
	One can see, cf. \cite[Theorem 1.6.18]{DemCompl}, that any Stein manifold is strongly pseudoconvex, meaning that it carries a smooth strictly plurisubharmonic exhaustion function.
	The famous result of Grauert from \cite[Theorem 2]{GrauertLevi}, states that the converse for relatively compact subdomains of complex manifolds holds as well.
	In this section, we will use this perspective through strictly plurisubharmonic exhaustions on Stein manifolds.
	\par 
	\begin{defn}\label{eq_stein_uniform_man}
	 We say that a Hermitian manifold $(X, g^{TX})$ admits a uniform Stein atlas if there are $r_1 > r_0 > 0$, $c, C > 0$, such that for any $x_0 \in X$, there is a Stein neighborhood $\Omega$ of $x_0$, such that  $\mathbb{B}_{x_0}^{X}(r_0) \subset \Omega \subset \mathbb{B}_{x_0}^{X}(r_1)$, and a smooth strictly plurisubharmonic exhaustion function $\delta$ on $\Omega$, verifying the following bounds
	 	\begin{equation}\label{eq_psi_norm_man}
	 	\begin{aligned}
	 		& 
	 		\delta > 0 \quad \text{ over } \, \Omega, \qquad \delta < C  \quad \text{ over } \, \mathbb{B}_{x_0}^{X}(r_0),
	 		\\
			&
			\imun \partial \dbar \delta > c \omega,
		\end{aligned}
		\end{equation}
		where $\omega$ is the Hermitian form associated to $g^{TX}$.
	\end{defn}
	\begin{prop}\label{thm_st_atl_ext}
		Any Kähler manifold of bounded geometry $(X, g^{TX})$ admits uniform Stein atlas.
		Moreover, $r_1 > r_0 > 0$, $c, C > 0$ from Definition \ref{eq_stein_uniform_man} depend only on $C_1$ from the bound (\ref{eq_bnd_curv_tm}) and $r_X$.
	\end{prop}
	\begin{proof}
		Recall that $\alpha_{x_0}$ was defined in (\ref{eq_defn_alpha}).
		By Lemma \ref{prop_plurisub_pt}, there is a uniform constant $r_1 > 0$, which depends only on $C_1$ from the bound (\ref{eq_bnd_curv_tm}) and $r_X$, such that $\imun \partial \dbar \alpha_{x_0} > \frac{1}{2} \omega$ over $\mathbb{B}_{x_0}^X(r_1)$.
		We now take $r_0 := \frac{r_1}{2}$, and let $\delta := \frac{1}{r_1^{2} -  \alpha_{x_0}}$.
		By the above and (\ref{eq_chi_f_psh}), we deduce that $\delta$ is strictly plurisubharmonic.
		Since it is clearly exhaustive, we deduce that $\Omega := \mathbb{B}_{x_0}^{X}(r_1)$ is strongly pseudoconvex.
		By (\ref{eq_chi_f_psh}), we conclude that the requirements of Definition \ref{eq_stein_uniform_man} are satisfied for the above choice and $c := \frac{1}{2 r_1^2}$, $C := \frac{2}{r_1^2}$.
	\end{proof}
	%\begin{rem}
	%	An easy modification of the proof of Proposition \ref{prop_plurisub_pt} shows that the conclusion of Theorem \ref{thm_st_atl_ext} continues to hold even if one replaces the Kähler assumption by a weaker one that for any $k \in \nat$, there is $C_k > 0$, such that for the complex structure $J$, we have $| \nabla^k J | < C_k$.
	%\end{rem}
	\begin{prop}\label{thm_hol_coord_exst1}
		For any Kähler manifold of bounded geometry $(X, g^{TX})$, there are $r_c, C > 0$, such that for any $x_0 \in X$, there are holomorphic coordinates $\chi := (h_1, \ldots, h_n) : \mathbb{B}_{x_0}^{X}(r_c) \to \comp^n$, verifying $h_i(x_0) = 0$, $|h_i| < C$, $i = 1, \ldots, n$, and such that $dh_i(x_0)$ form an orthonormal frame of $(T^{(1, 0)*}X, g^{T^*X})$.
		Moreover, $r_c, C > 0$, depend only on $C_{n + 7}$ from the bound (\ref{eq_bnd_curv_tm}) and $r_X$.
	\end{prop}
	\begin{comment}
	\begin{rem}\label{rem_wuyau}
		From Lemma \ref{lem_ell_reg_sob_bnd_g}, all higher derivatives of $h_i$ (with respect to $\nabla^{TX}$) can be bounded uniformly on $X$.
		Hence, there is $r_c^1 > 0$, which depends only on $C_{n + 7}$ from the bound (\ref{eq_bnd_curv_tm}) and $r_X$, such that $\mathbb{B}_0^{\comp^n}(r_c^1) \subset \Im \chi$.
		\par 
		The above observation with the results of Section \ref{sect_bnd_geom_cf} imply that $\chi$ are coordinates of bounded geometry in the sense of Cheng-Yau \cite[Definition 1.1]{ChengYau}.
		The existence of such coordinates on Kähler manifold of bounded geometry was proved by Wu-Yau \cite[Theorem 9]{WuYau} by different methods.
	\end{rem}
	\end{comment}
	\begin{proof}
		The main idea of the proof is to first construct smooth functions $h'_i$, $i = 1, \ldots, n$, as in Proposition \ref{thm_hol_coord_exst1} and then by using Hörmander's $L^2$-estimates to perturb $h'_i$ to holomorphic coordinates $h_i$.
		\par 
		Let $r_0, r_1 > 0$, $\Omega$, $\delta$ be as in Definition \ref{eq_stein_uniform_man}.
		First, we define $h'_i$ in geodesic coordinates $\phi_{x_0}$ by
		\begin{equation}\label{eq_hpr_zi_defn}
			h'_i := z_i \rho \Big(  \frac{|Z|}{r_0} \Big),
		\end{equation}
		where $\rho$ is a bump function as in (\ref{defn_rho_fun}).
		By bounded geometry condition, there is $C_1 > 0$, which depends only on $r_0$ and $C_0$ from (\ref{eq_bnd_curv_tm}), such that 
		\begin{equation}\label{eq_hi_l2_nrm}
			\int_{\Omega} |h'_i|^2 dv_{g^{TX}} \leq C_1.
		\end{equation}
		Directly from Lemma \ref{lem_cmplx_str_exp}, applied for $Y := \{ x_0 \}$, we see that there is $C_2 > 0$, which depends only on $r_0$ and $C_1$ from (\ref{eq_bnd_curv_tm}), such that over $\mathbb{B}_{x_0}^{X}(r_0)$, we have
		\begin{equation}\label{eq_dbarhi_pr_bnd}
			|\dbar h'_i| \leq C_2 |Z|^2.
		\end{equation}
		\par 
		From this, the fact that ${\rm{supp}} \, \dbar h'_i \subset \mathbb{B}_{x_0}^{X}(r_0)$ and the first estimate from (\ref{eq_psi_norm_man}), we conclude that for any $c > 0$, there is a uniform constant $C_3$, which depends only on $r_0$, $C$ from the first line in (\ref{eq_psi_norm_man}), and $C_1$ from (\ref{eq_bnd_curv_tm}), such that for $\delta_{ \{x_0\} }$, from Lemma \ref{thm_plurisub}, we have
		\begin{equation}\label{eq_est_dbarhpro}
		\int_{\Omega} |\dbar h'_i|^2 \exp(-c \delta) \cdot \exp(- (2n + 1) \delta_{ \{x_0\} }) dv_{g^{TX}}
		\leq
		C_3.
		\end{equation}
		\par 
		Now, as $(X, g^{TX})$ is of bounded geometry, by Lemma \ref{thm_plurisub}, there is $c > 0$ depending only on $C_1$ from (\ref{eq_bnd_curv_tm}) and $c$ from the second line of (\ref{eq_psi_norm_man}), such that for $E := K_X^{-1}$, $h^E := (h^{K_X})^{-1} \cdot \exp(-c \delta) \cdot \exp(- (2n + 1) \delta_{ \{x_0\} })$, the curvature $R^E$ of $(E, h^E)$ over $\Omega$ satisfies $\imun R^E \geq \omega \otimes {\rm{Id}}_E$,
		where the positivity here is in the sense of Nakano, cf. \cite[Definition VII.6.3]{DemCompl}.
		By (\ref{eq_est_dbarhpro}) and the result \cite[Théorème 0.2]{Dem82}, which states that $\Omega$ admits a complete Kähler metric, we may apply \cite[Théorème 5.1]{Dem82} for $X := \Omega$, $g := \dbar h'_i$ and $(E, h^E)$ as above to deduce that there are sections $f_i \in L^2(\Omega, K_{\Omega} \otimes E)$, such that we have $\dbar f_i = g$, and 
		\begin{multline}\label{eq_est_fi}
			\int_{\Omega} |f_i|^2 \exp(-c \delta) \cdot \exp(- (2n + 1) \delta_{ \{x_0\} }) dv_{g^{TX}}
			\\
			\leq
			\int_{\Omega} |\dbar h'_i|^2 \exp(-c \delta) \cdot \exp(- (2n + 1) \delta_{ \{x_0\} }) dv_{g^{TX}}.
		\end{multline}
		\par 
		We now put $h_i := h'_i - f_i$.
		We trivially have $\dbar h_i = 0$. 
		Also, by (\ref{eq_psi_norm_man}), (\ref{eq_hi_l2_nrm}), (\ref{eq_est_dbarhpro}) and (\ref{eq_est_fi}), we obtain
		\begin{equation}\label{eq_hi_bnd_l2}
			\int_{\mathbb{B}_{x_0}^X(r_0)} |h_i|^2 dv_{g^{TX}}
			\leq
			C_4,
		\end{equation}
		where $C_4$ depends only on $r_0$, $C$ from the first line in (\ref{eq_psi_norm_man}), and $C_1$ from (\ref{eq_bnd_curv_tm}).
		By the standard regularity result, \cite[Lemme 6.9]{Dem82}, $h_i$ is smooth and holomorphic on $\Omega$, and the equation $\dbar h_i = 0$ holds in the usual sense.
		From this, (\ref{eq_est_dbarhpro}), (\ref{eq_est_fi}) and the non-integrability of $|Z| \exp(- (2n + 1) \delta_{ \{x_0\} })$ and $ \exp(- (2n + 1) \delta_{ \{x_0\} })$ near $x_0$, we deduce that $f_i(x_0) = 0$ and $d f_i(x_0) = 0$.
		Hence $h_i(x_0) = h'_i(x_0)$ and $d h_i(x_0) = d h'_i(x_0)$.
		Thus, $\partial h_i(x_0)$ form an orthonormal basis of $(T^{(1, 0)*}X, g^{T^*X})$.
		\par 
		From Lemma \ref{lem_ell_reg_sob_bnd_g} and (\ref{eq_hi_bnd_l2}), we conclude that there is a constant $C_5 > 0$, which depends only on $r_0, r_1$ and $C_{n + 7}$ from (\ref{eq_bnd_curv_tm}), such that
		$
			\|
				h_i
			\|_{\ccal^3(\mathbb{B}_{x_0}^{X}(r_0))}
			\leq
			C_5
		$.
		We conclude that there is $r_c > 0$, which depends only on $C_{n + 7}$ from the bound (\ref{eq_bnd_curv_tm}) and $r_X$, such that $(h_1, \ldots, h_n)$ form holomorphic coordinates over $\mathbb{B}_{x_0}^{X}(r_c)$.
		This finishes the proof.
	\end{proof}
	\par 
	\begin{lem}\label{cor_hl_fr_bnd_g}
		Let $(E, h^E)$ be a Hermitian vector bundle of bounded geometry over a Kähler manifold of bounded geometry $(X, g^{TX})$. 
		Then there are $r_0, C > 0$, such that for any $x_0 \in X$, there is a local holomorphic frame $(f_1, \ldots, f_r) \in H^0(\mathbb{B}_{x_0}^{X}(r_0), E)$ of $E$, such that
		\begin{equation}\label{eq_f_un_bnd_geom}
			\big\| 
				f_i
			\big\|_{L^2(\mathbb{B}_{x_0}^{X}(r_0))}
			\leq 
			C,
		\end{equation}
		and $(f_1(x_0), \ldots, f_r(x_0))$ is an orthonormal frame of $(E_{x_0}, h^E_{x_0})$.
		Moreover, $C$ depends only on $r_0$, $c, C$ from (\ref{eq_psi_norm_man}) and $C_{n + 7}$ from (\ref{eq_bnd_curv_tm}) and (\ref{eq_bnd_re_ck}).
	\end{lem}
	\begin{proof}
		The proof can be deduced from Hörmander's $L^2$-estimates and Lemmas \ref{lem_gamma_exp}, \ref{lem_cmplx_str_exp}, \ref{thm_plurisub} in exactly the same way as in the proof of Proposition \ref{thm_hol_coord_exst1}.
	\end{proof}
	\par 
	\begin{thm}\label{lem_hol_coord_exst}
		Assume that the Kähler triple $(X, Y, g^{TX})$ is of bounded geometry. Then there are $r_c, C > 0$, such that for any $y_0 \in Y$, there are holomorphic coordinates $\chi := (h_1, \ldots, h_n) : \mathbb{B}_{y_0}^{X}(r_c) \to \comp$, $h_i(x_0) = 0$, $|h_i| < C$, $i = 1, \ldots, n$, $h_{m+1}|_Y, \ldots, h_n|_Y = 0$, and such that $\partial h_i(x_0)$ form an orthonormal frame of $(T^{(1, 0)*}X, g^{T^*X})$.
		Moreover, $r_c, C$ depend only on $r_X$ and $C_{n + 7}$ from (\ref{eq_bnd_curv_tm}) and (\ref{eq_bnd_a_ck}).
	\end{thm}
	\begin{rem}\label{rem_hol_coord_exst}
		We see that there is $r_c^1 > 0$, which depends only on $C_{n + 7}$ from the bounds (\ref{eq_bnd_curv_tm}), (\ref{eq_bnd_a_ck}), $r_Y$, $r_X$, $r_{\perp}$, such that $\mathbb{B}_0^{\comp^n}(r_c^1) \subset \Im \chi$.
	\end{rem}
	\begin{proof}
		The proof is very similar to the proof of Proposition \ref{thm_hol_coord_exst1}. The only difference is that instead of the weight $(2n + 1) \delta_{ \{x_0\} } $ in (\ref{eq_est_dbarhpro}), one has to consider $(2m + 1) \delta_{ \{x_0\} } + 2(n-m) \delta_Y$.
	\end{proof}

\section{Operator algebras on manifolds of bounded geometry}\label{sect_2}
	In this section, we prove that the set of operators on manifolds of bounded geometry, admitting certain bounds on the Schwartz kernels, forms an algebra under the composition.
	\par 
	This section is organized as follows.
	In Section \ref{sect_bish_gr}, we show that the set of operators with exponential decay of the Schwartz kernel forms an algebra on manifolds of bounded geometry.
	Then in Section \ref{sect_model_calc}, we give explicit formulas for the Schwartz kernels of the orthogonal Bergman projector and the extension operator on the pair $(\comp^n, \comp^m)$.
	We also state the composition rules for operators with related kernels.
	Finally, in Section \ref{sect_algtay_type}, we prove that the set of operators, whose Schwartz kernel admits Taylor-type expansion, forms an algebra under the composition.

	\subsection{Algebra of operators with exponential decay of the Schwartz kernel}\label{sect_bish_gr}
	In this section, we show that the set of operators with exponential decay essentially forms an algebra on manifolds of bounded geometry.
	\par 
	Let $(M, g^{TM})$ be a complete manifold with a positive bound on the injectivity radius, $r_M > 0$. We note $n := \dim M$.
	We fix an arbitrary sequence of Hermitian vector bundles $(E_p, \nabla^{E_p}, h^{E_p})$, $p \in \nat^*$, on $M$, endowed with a fixed connection.
	We fix $s > 0$ and assume that the Ricci curvature ${\rm{Ric}}_{g^{TM}}$ of $(M, g^{TM})$ satisfies the bound
	\begin{equation}\label{eq_ric_bnd}
		{\rm{Ric}}_{g^{TM}}  \geq  -(n-1)s.
	\end{equation}
	Clearly, (\ref{eq_ric_bnd}) is satisfied for some $s > 0$ once $(M, g^{TM})$ is of bounded geometry.
	\par 
	Let us fix a volume form $dv_M$ on $M$, such that the condition, analogous to (\ref{eq_vol_comp_unif}), is satisfied with respect to the metric $g^{TM}$.
	Let us fix $q \in \nat^*$, and a sequence of operators $A_p^{1}, \ldots, A_p^{q}$, $p \in \nat^*$, acting on $\ccal^{\infty}(M, E_p)$ by the convolutions with smooth kernels $A_p^{1}(x_1, x_2), \ldots, A_p^{q}(x_1, x_2) \in E_{p, x_1} \otimes  E_{p, x_2}^{*}$, $x_1, x_2 \in X$, with respect to the volume form $dv_M$.
	We assume that for any $k \in \nat$, there are $c, C_h > 0$, $h = 1, \ldots, q$, such that for any $x_1, x_2 \in X$, we have
	\begin{equation}\label{eq_a_bnd_1}
		\big| A_p^h(x_1, x_2) \big|_{\ccal^k(M \times M)} \leq C_h \cdot p^{\frac{n+k}{2}} \cdot \exp \big(- c \sqrt{p} \cdot \dist(x_1, x_2) \big),
	\end{equation}
	and there is $h$ among $1, \ldots, q$, such that even stronger bound holds
	\begin{multline}\label{eq_a_bnd_2}
		\big| A_p^h(x_1, x_2) \big|_{\ccal^k(M \times M)} 
		\\
		\leq C_h \cdot p^{\frac{n+k}{2}}  \cdot \exp \Big(- c \sqrt{p} \cdot \big(  \dist(x_1, x_2) + \dist(x_1, W) + \dist(x_2, W) \big) \Big),
	\end{multline}
	where $W$ is a closed subset in $M$, and the $\ccal^k$-norm is taken in the sense of Theorem \ref{thm_ext_exp_dc}.
	\begin{lem}\label{lem_bnd_prod_a}
		There is $C_0 > 0$, which depends only on $s, r_M, n$ and $c$ from (\ref{eq_a_bnd_1}) and (\ref{eq_a_bnd_2}), such that for any $p \in \nat^*$, $\sqrt{p} > 4 (n-1) \frac{\sqrt{s}}{c}$, the following operator is well-defined 
		\begin{equation}\label{eq_dp_defn_apl}
			D_p := A_p^{1} \circ \cdots \circ A_p^{q},
		\end{equation}
		and for $C := \prod_{h = 1}^{q} C_h$, the following bound holds
		\begin{multline}\label{eq_bnd_prod_a}
			\big| D_p(x_1, x_2) \big|_{\ccal^k(X \times X)} 
			\\
			\leq C_0^q C  p^{\frac{n+k}{2}} \cdot \exp \Big(- \frac{c}{8} \sqrt{p} \cdot \big(  \dist(x_1, x_2) + \dist(x_1, W) + \dist(x_2, W) \big) \Big).
		\end{multline}			 
	\end{lem}
	In the proof of Lemma \ref{lem_bnd_prod_a} and elsewhere, the following proposition plays a crucial role.
	\begin{prop}\label{prop_exp_bound_int}
		There is a constant $C' > 0$, which depends only on $n$, $s$, $r_M$, such that for any $x_0 \in M$, $l > 2 (n-1) \sqrt{s}$, the following bound holds
		\begin{equation}\label{eq_prop_exp_bound_int}
			\int_{M} \exp \big(-l \dist(x_0, x) \big) dv_{g^{TM}}(x) < \frac{C'}{l^{n}}.
		\end{equation}
	\end{prop}
	\begin{proof}
		The main idea of our proof is to use Bishop-Gromov inequality to bound the volumes of balls and spheres in our manifold.
		We denote by $v(n, s, r)$ the volume of a ball of radius $r$ in the space form of constant curvature $-s$, $s > 0$, cf. \cite[p. 69]{PetersRiemG}.
		From \cite[pp. 69, 269]{PetersRiemG}, we see that
		\begin{equation}\label{eq_vnr_bnd}
			v(n, s, r) = {\rm{Vol}}(S^{n-1}) \cdot \int_{0}^{r} \Big( \frac{\sinh (\sqrt{s} r)}{\sqrt{s}}  \Big)^{n-1} dr,
		\end{equation}
		where ${\rm{Vol}}(S^{n-1})$ is the volume of the unit sphere $S^{n-1}$, endowed with the standard Riemannian metric.
		From (\ref{eq_vnr_bnd}), we see that for any $s > 0$, there is $C > 0$, such that for any $r > 0$, we have
		\begin{equation}\label{eq_vnr_bnd2}
			v(n, s, r) \leq C \exp \big( (n-1) \sqrt{s} r \big).
		\end{equation}
		\par 
		Bishop-Gromov inequality \cite[Lemma 9.36]{PetersRiemG} states that the following function is non-increasing
		\begin{equation}\label{eq_bis_gr}
			\real \ni r \mapsto \frac{{\rm{Vol}}(\mathbb{B}_{x_0}^{M}(r))}{v(n, s, r)}.
		\end{equation}
		Moreover, the limit of (\ref{eq_bis_gr}), as $r \to 0$, is equal to $1$.
		In particular, for any $r \geq 0$, we have
		\begin{equation}\label{eq_bnd_vol_eas}
			{\rm{Vol}}(\mathbb{B}_{x_0}^{M}(r)) \leq v(n, s, r).
		\end{equation}
		\par 
		We now decompose the integral in (\ref{eq_prop_exp_bound_int}) into two parts: over $\mathbb{B}_{x_0}^{M}(r_M)$ and over its complement, $V$.
		For the second part, we have the following bound
		\begin{equation}\label{eq_exp_bound_int1}
			\int_{V} \exp \big(-l \dist(x_0, x) \big) dv_{g^{TM}}(x) 
			\leq
			\sum_{i = 1}^{\infty} \exp (-i l r_M) {\rm{Vol}}\Big(\mathbb{B}_{x_0}^{M}((i+1) r_M) \setminus \mathbb{B}_{x_0}^{M}(i r_M)\Big).
		\end{equation}
		However, from (\ref{eq_vnr_bnd2}) and (\ref{eq_bnd_vol_eas}), we see that 
		\begin{equation}\label{eq_exp_bound_int2}
			{\rm{Vol}}\Big(\mathbb{B}_{x_0}^{M}((i+1) r_M) \setminus \mathbb{B}_{x_0}^{M}(i r_M)\Big)
			\leq
			C \exp \big( (i + 1) (n-1) \sqrt{s} r_M \big).
		\end{equation}
		By combining (\ref{eq_exp_bound_int1}) and (\ref{eq_exp_bound_int2}), for $l > 2 (n-1) \sqrt{s}$, we easily get the bound
		\begin{equation}\label{eq_exp_bound_int22}
			\int_{V} \exp \big(-l \dist(x_0, x) \big) dv_{g^{TM}}(x) 
			\leq
			C
			\exp
			\Big(
				- \frac{l r_M}{2}
			\Big).
		\end{equation}
		\par 
		It is now only left to bound the integral over $\mathbb{B}_{x_0}^{M}(r_M)$.
		The coarea formula gives us for $r < r_M$ the following identity
		\begin{equation}\label{eq_coarea}
			{\rm{Vol}}(\mathbb{B}_{x_0}^{M}(r)) = \int_0^{r} {\rm{Vol}}(S_{x_0}^{M}(r')) dr',
		\end{equation}
		where $S_{x_0}^{M}(r)$ is the volume of the $r$-sphere around ${x_0}$.
		By differentiating the function (\ref{eq_bis_gr}), and using (\ref{eq_coarea}), we see that (\ref{eq_bis_gr}) is equivalent to the following bound
		\begin{equation}\label{eq_bis_gr_2}
			\frac{{\rm{Vol}}(S_{x_0}^{M}(r))}{{\rm{Vol}}(\mathbb{B}_{x_0}^{M}(r))} \leq \frac{t(n, s, r)}{v(n, s, r)},
		\end{equation}
		where $t(n, s, r) = \frac{d}{dr} v(n, s, r)$.
		From (\ref{eq_vnr_bnd}), (\ref{eq_bnd_vol_eas})  and (\ref{eq_bis_gr_2}), we deduce
		\begin{equation}\label{eq_exp_bound_int101}
			{\rm{Vol}}(S_{x_0}^{M}(r)) \leq {\rm{Vol}}(S^{n-1}) \cdot \Big( \frac{\sinh (\sqrt{s} r)}{\sqrt{s}}  \Big)^{n-1}.
		\end{equation}
		From (\ref{eq_exp_bound_int101}), we see that there is $C_0 > 0$, which depends only on $s$ and $r_M$, such that
		\begin{equation}\label{eq_exp_bound_int124}
			\int_{0}^{r_M} \exp (-l r) {\rm{Vol}}(S_{x_0}^{M}(r)) dr 
			\leq 
			\frac{C_0}{l^n}.
		\end{equation}
		Now, from the coarea formula as in (\ref{eq_coarea}), we have
		\begin{equation}\label{eq_exp_bound_int123}
			\int_{\mathbb{B}_{x_0}^{M}(r_M)} \exp \big(-l \dist(x_0, x) \big) dv_{g^{TM}}(x) 
			=
			\int_{0}^{r_M} \exp (-l r) {\rm{Vol}}(S_{x_0}^{M}(r)) dr.
		\end{equation}
		From (\ref{eq_exp_bound_int22}), (\ref{eq_exp_bound_int124}) and (\ref{eq_exp_bound_int123}), we conclude.
	\end{proof}		
	\begin{proof}[Proof of Lemma \ref{lem_bnd_prod_a}]
		For simplicity of the presentation, we only present the proof for $k = 0$, as the general case is treated in an analogous way.
		We trivially have the following identity
		\begin{equation}\label{eq_b_f}
			D_p(x_1, x_2) = \int_{M^{\times (q - 1)}} A_p^{1}(x_1, z_1) \cdot A_p^{2}(z_1, z_2) \cdot \ldots \cdot A_p^{q}(z_{q - 1}, x_2) dv_M(z_1) \cdot \ldots \cdot dv_M (z_{q - 1}).
		\end{equation}
		Now, the triangle inequality readily implies
		\begin{equation}\label{eq_triangle}
		\begin{aligned}
			&
			\dist(x_1, z_1) + \dist(z_1, z_2) + \cdots + \dist(z_{q - 1}, x_2) \geq \dist(x_1, x_2),
			\\
			&
			\dist(x_1, z_1) + \cdots + \dist(z_{q - 1}, x_2) + \dist(z_i, W)  \geq \dist(x_1, W).
		\end{aligned}
		\end{equation}
		Then for $C$ as in (\ref{eq_bnd_prod_a}), by (\ref{eq_a_bnd_1}), (\ref{eq_a_bnd_2}), (\ref{eq_b_f}) and (\ref{eq_triangle}), the following bound holds
		\begin{multline}\label{eq_b_bnd}
			\big| D_p(x_1, x_2) \big|
			\leq 
			C p^{\frac{n}{2}} \cdot  \exp \Big(- \frac{c}{4} \sqrt{p} \cdot \big( \dist(x_1, W) + \dist(x_1, x_2) \big) \Big)
			\cdot
			\\
			\cdot
			\int_M
			p^{\frac{n}{2}} \exp \Big(- \frac{c}{2} \sqrt{p} \cdot \dist(z_{q - 1}, x_2) \Big)
			\cdot
			\int_M
			p^{\frac{n}{2}} \exp \Big(- \frac{c}{2} \sqrt{p} \cdot \dist(z_{q - 2}, z_{q - 1}) \Big)
			\cdot
			\ldots
			\\
			\cdot
			\int_M
			p^{\frac{n}{2}} \exp \Big(- \frac{c}{2} \sqrt{p} \cdot \dist(z_1, z_2) \Big) 
			dv_M(z_1)
			\cdots
			dv_M(z_{q-1}).
		\end{multline}
		By triangle inequality, we also have
		\begin{equation}\label{eq_trng_2}
			\dist(x_1, W) + \dist(x_1, x_2)  \geq \dist(x_2, W).
		\end{equation}
		By applying Proposition \ref{prop_exp_bound_int} and (\ref{eq_vol_comp_unif}), for the integrals over $z_1, \ldots, z_{q - 1}$ from (\ref{eq_b_bnd}), and (\ref{eq_trng_2}), we get (\ref{eq_bnd_prod_a}) for $k = 0$ from (\ref{eq_b_bnd}) for $C_0 := C'$ from  Proposition \ref{prop_exp_bound_int}.
	\end{proof}
	\par 
	We fix a smooth submanifold $H \subset M$ such that $(M, H, g^{TM})$ is of bounded geometry.
	We will use the following consequence of Proposition \ref{prop_exp_bound_int} in what follows.
	\begin{cor}\label{cor_exp_bound_int}
		There are $c, C' > 0$, which depend only on $n$, $m$, $r_M$, $r_N$, $r_{\perp}$ $s$, $r_M$, and $C_0$ from (\ref{eq_bnd_curv_tm}) and (\ref{eq_bnd_a_ck}), such that for any $y_0 \in H$, $l > c$, the following bound holds
		\begin{equation}
			\int_{H} \exp \big(-l \dist_M(y_0, y) \big) dv_{g^{TH}}(y) < \frac{C'}{l^{m}}.
		\end{equation}
	\end{cor}
	\begin{rem}
		One can easily construct an example of a triple $(M, H, g^{TM})$ of bounded geometry, for which the immersion $\iota : H \to M$ is not a quasi-isometry. Hence Corollary \ref{cor_exp_bound_int} doesn't follow directly from Proposition \ref{prop_exp_bound_int}.
	\end{rem}
	\begin{proof}
		Directly from the coordinate-wise description of bounded geometry condition, cf. Section \ref{sect_bnd_geom_cf}, we see that there are $\epsilon_0, c > 0$, depending only on $R$ from (\ref{eq_r_defn_const}) and $C_0$, for which (\ref{eq_bnd_curv_tm}), (\ref{eq_bnd_a_ck}) hold, such that for any $y_0, y_1 \in H$, verifying $\dist_M(y_0, y_1) < \epsilon_0$, we have
		\begin{equation}\label{eq_dist_hm_00}
			\dist_M(y_0, y_1) > c \dist_H(y_0, y_1).
		\end{equation}
		We let $\epsilon := \frac{1}{2} \min \{ \epsilon_0, R \}$.
		Similarly to the proof of Proposition \ref{prop_exp_bound_int}, we decompose the integration into two parts: over $\mathbb{B}_{y_0}^H(\epsilon)$ and over its complement, $V$.
		From (\ref{eq_dist_hm_00}), we then conclude that there is a constant $C''$, depending only on $R$ and $C_0$, for which (\ref{eq_bnd_curv_tm}), (\ref{eq_bnd_a_ck}) hold, such that
		\begin{equation}\label{eq_cor_int_amb_1}
			\int_{\mathbb{B}_{y_0}^H(\epsilon)} \exp \big(-l \dist_M(y_0, y) \big) dv_{g^{TH}}(y) < \frac{C''}{l^{m}}.
		\end{equation}
		\par 
		Now, let us estimate the integral over $V$.
		For this, remark that for $x \in \mathbb{B}_H^M(\frac{\epsilon}{2 \max(1, c)})$ and $y_1 := \pi_0(x)$, where $\pi_0$ is as introduced before (\ref{eq_kappan}), we have
		\begin{equation}\label{eq_cor_int_amb_121}
			\dist_M(y_0, y_1) \geq \dist_M(y_0, x) - \frac{\epsilon}{2 \max(1, c)}.
		\end{equation}
		But since for any $x \in \mathbb{B}_H^M(\frac{\epsilon}{2 \max(1, c)})$, such that $y_1 \in V$, we have $\dist_M(y_0, y_1) \geq \frac{\epsilon}{\max(1, c)}$ by (\ref{eq_dist_hm_00}), we conclude by (\ref{eq_cor_int_amb_121}) that we have
		\begin{equation}
			\dist_M(y_0, y_1) \geq \frac{1}{2} \dist_M(y_0, x).
		\end{equation}
		From this, we deduce that there is a constant $C'''$, depending only on $R$ and $C_0$, for which (\ref{eq_bnd_curv_tm}), (\ref{eq_bnd_a_ck}) hold, such that
		\begin{equation}\label{eq_cor_int_amb_2}
			\int_{V} \exp \big(-l \dist_M(y_0, y) \big) dv_{g^{TH}}(y) < C''' \cdot \int_{\mathbb{B}_V^M(\frac{\epsilon}{2 \max(1, c)})} \exp \big(- \frac{l}{2} \dist_M(y_0, x) \big) dv_{g^{TM}}(x).
		\end{equation}
		We conclude by (\ref{eq_exp_bound_int22}), (\ref{eq_cor_int_amb_1}) and (\ref{eq_cor_int_amb_2}).
	\end{proof}

	\subsection{Model operators; Fock-Bargmann space and kernel calculus}\label{sect_model_calc}
		In this section, we consider the model situation of the complex vector space, for which an explicit formula for the Schwartz kernels of Bergman projectors and the extension operator can be given.
		We then use those explicit formulas to give a description for the compositions of the operators, the Schwartz kernels of which can be expressed using the above kernels.
		This section (as well as the next one) is motivated in many ways by the works of Ma-Marinescu \cite{MaMarToepl}, \cite{MaHol}, and it motivates the formulation of Theorems \ref{thm_ext_as_exp}.
		\par 
		Endow $X := \comp^n$ with the standard metric and consider a trivialized complex line bundle $L_0$ on $\comp^n$.
		We endow $L_0$ with the Hermitian metric $h^{L_0}$, given by 
		\begin{equation}\label{eq_model_metrl}
			\| 1 \|_{h^{L_0}}(Z) = \exp \Big(- \frac{\pi}{2} |Z|^2 \Big),
		\end{equation}
		where $Z$ is the natural real coordinate on $\comp^n$, and $1$ is the trivializing section of $L_0$.
		An easy verification shows that (\ref{eq_model_metrl}) implies that (\ref{eq_gtx_def}) holds in our setting.
		Recall that \cite[\S 4.1.6]{MaHol} shows that the Kodaira Laplacian $\mathscr{L}$ on $\ccal^{\infty}(X, L_{0})$, multiplied by $2$, and viewed as an operator on $\ccal^{\infty}(X)$ using the orthonormal trivialization, given by $1 \cdot \exp (\frac{\pi}{2} |Z|^2 )$, is given by
		\begin{equation}\label{eq_mathscr_l_op}
			\mathscr{L} = \sum_{i = 1}^{n} b_i b_i^{+},
		\end{equation}
		where $b_i$, $b_i^{+}$ are \textit{creation} and \textit{annihilation} operators, defined as
		\begin{equation}
			b_i = -2 \frac{\partial}{\partial z_i} + \pi \overline{z}_i, \qquad b_i^{+} = 2 \frac{\partial}{\partial  \overline{z}_i} + \pi z_i. 
		\end{equation}
		\par 
		A classical calculation, cf. \cite[Theorem 4.1.20]{MaHol}, shows that the orthonormal basis with respect to the induced $L^2$-norm of $\ker \mathscr{L}$ is given in the orthonormal trivialization above by
		\begin{equation}\label{eq_orth_basker}
			\Big( \frac{\pi^{|\beta|}}{\beta!} \Big)^{1/2} z^{\beta} \exp \Big( -\frac{\pi}{2} |Z|^2 \Big), \qquad \beta \in \nat^n.
		\end{equation}
		In particular, \cite[(4.1.84)]{MaHol}, the Bergman kernel $\mathscr{P}_{n}$ of $\comp^n$ is given by
		\begin{equation}\label{eq_berg_k_expl}
			\mathscr{P}_n(Z, Z') = \exp \Big(
				-\frac{\pi}{2} \sum_{i = 1}^{n} \big( 
					|z_i|^2 + |z'_i|^2 - 2 z_i \overline{z}'_i
				\big)
			\Big), \quad \text{for } Z, Z' \in \comp^n.
		\end{equation}
		\par 
		Also, we see that the Schwartz kernel of the orthogonal Bergman kernel, corresponding to the projection onto holomoprhic sections orthogonal to those which vanish along $\comp^m$, is given by
		\begin{equation}
			\mathscr{P}_{n, m}^{\perp}(Z, Z') 
			= 
			\sum_{\beta \in \nat^m}
			\Big( \frac{\pi^{|\beta|}}{\beta!} \Big) z^{\beta} \overline{z}'^{\beta} \exp \Big( -\frac{\pi}{2} (|Z|^2 + |Z'|^2) \Big).
		\end{equation}		
		By simplifying the above expression using (\ref{eq_berg_k_expl}), we see that $\mathscr{P}_{n, m}^{\perp}(Z, Z')$ corresponds precisely to the quantity, defined in (\ref{eq_pperp_defn_fun}).
		\par 
		Let us calculate the $L^2$-extension operator $\mathscr{E}_{n, m}$, extending each element from $(\ker \mathscr{L})|_Y$ to an element from $\ker \mathscr{L}$ with the minimal $L^2$-norm.
		From (\ref{eq_orth_basker}), we easily see that for $Z_Y \in \comp^m$, $Z_N \in \comp^{n-m}$ and $g \in (\ker \mathscr{L})|_Y$, we have
		\begin{equation}\label{eq_escr}
			(\mathscr{E}_{n, m} g) (Z_Y, Z_N) =
			g(Z_Y) 
			\exp \Big(
				-\frac{\pi}{2} |Z_N|^2
			\Big).
		\end{equation}	
		We extend $\mathscr{E}_{n, m}$ to the whole $L^2$-space by $g \mapsto (\mathscr{E}_{n, m} \circ \mathscr{P}_{m}) g$.	
		From (\ref{eq_escr}), we see that the kernel of $\mathscr{E}_{n, m}$ corresponds precisely to the quantity, defined in  (\ref{eq_ext_defn_fun}).
	\par 
	Now, a lot of calculations in this article will have something to do with compositions of operators having Schwartz kernels, given by the product of polynomials with the above kernels.
	For that reason, the following lemma will be of utmost importance in what follows.
	\begin{lem}\label{lem_comp_poly}
			For any polynomials $A_1(Z, Z'), A_2(Z, Z')$, $Z, Z' \in \real^{2n}$, there is a polynomial $A_3 := \mathcal{K}_{n, m}[A_1, A_2]$, the coefficients of which are polynomials of the coefficients of $A_1, A_2$, such that
			\begin{equation}\label{eq_lem_comp_poly_1}
				(A_1 \cdot \mathscr{P}_{n, m}^{\perp}) \circ (A_2 \cdot \mathscr{P}_{n, m}^{\perp})
				=
				A_3 \cdot \mathscr{P}_{n, m}^{\perp}.
			\end{equation}
			Moreover, $\deg A_3 \leq \deg A_1 + \deg A_2$. Also, if both polynomials $A_1$, $A_2$ are even or odd (resp. one is even, another is odd), then the polynomial $A_3$ is even (resp. odd).
			\par 
			Similarly, there is a polynomial $A'_3 := \mathcal{K'}_{n, m}[A_1, A_2]$ with the same properties as $A_3$, such that 
			\begin{equation}\label{eq_lem_comp_poly_2}
				(A_1 \cdot \mathscr{P}_n) \circ (A_2 \cdot \mathscr{P}_{n, m}^{\perp})
				=
				A'_3 \cdot \mathscr{P}_{n, m}^{\perp}.
			\end{equation}
			\par 
			Also, for any polynomials $A_1(Z, Z'), A_2(Z_Y, Z'_Y)$, where $Z, Z' \in \real^{2n}$, $Z_Y, Z'_Y \in \real^{2m}$, there is a polynomial $A''_3 := \mathcal{K''}_{n, m}[A_1, A_2]$ in $(Z, Z'_Y)$, with the same properties as $A_3$, such that
			\begin{equation}\label{eq_lem_comp_poly_3}
				(A_1 \cdot \mathscr{P}_{n, m}^{\perp}) \circ \mathscr{E}_{n, m} \circ (A_2 \cdot \mathscr{P}_m)
				=
				A''_3 \cdot \mathscr{E}_{n, m}.
			\end{equation}
			Finally, for any polynomials $A_1(Z, Z'_Y), A_2(Z_Y, Z'_Y)$, where $Z \in \real^{2n}$, $Z_Y, Z'_Y \in \real^{2m}$, there is a polynomial $A'''_3 := \mathcal{K'''}_{n, m}[A_1, A_2]$ in $(Z, Z'_Y)$, with the same properties as $A_3$, such that
			\begin{equation}\label{eq_lem_comp_poly_4}
				(A_1 \cdot \mathscr{E}_{n, m}^{\perp}) \circ (A_2 \cdot \mathscr{P}_m)
				=
				A'''_3 \cdot \mathscr{E}_{n, m}.
			\end{equation}
		\end{lem}
		\begin{proof}
			First of all, since  $\mathscr{P}_{n, n}^{\perp} = \mathscr{P}_{n}$, for $n = m$, (\ref{eq_lem_comp_poly_1}) was proved in \cite[Lemma 7.1.1, (7.1.6)]{MaHol} by the use of so-called kernel calculus.
			Let us now show that the general case of (\ref{eq_lem_comp_poly_1}) can be reduced to this special one.
			For this, remark that by (\ref{eq_pperp_defn_fun}), we have
			\begin{equation}\label{eq_compa_1}
			\begin{aligned}
				&
				\mathscr{P}_{n, m}^{\perp}(Z, Z') = \mathscr{P}_{n, m}^{\perp}(Z, Z'_Y) \cdot \exp \Big(- \frac{\pi}{2} |Z'_N|^2 \Big),
				&&
				\mathscr{P}_{n, m}^{\perp}(Z, Z'_Y) = \mathscr{P}_n(Z, Z'_Y),
				\\
				&
				\mathscr{P}_{n, m}^{\perp}(Z, Z') 
				=
				\mathscr{P}_n(Z_Y, Z')
				\cdot
				\exp \Big(- \frac{\pi}{2} |Z_N|^2 \Big),
				&&
				\mathscr{P}_{n, m}^{\perp}(Z_Y, Z') = \mathscr{P}_n(Z_Y, Z').		
			\end{aligned}				
			\end{equation}	
			We decompose the polynomials $A_1$, $A_2$ as follows
			\begin{equation}\label{eq_compa_3}
				A_1(Z, Z')
				=
				\sum_{\alpha} Z_N^{\alpha} \cdot A_1^{\alpha}(Z_Y, Z'), 
				\qquad
				A_2(Z, Z')
				=
				\sum_{\alpha'} A_2^{\alpha'}(Z, Z'_Y) Z'_N{}^{\alpha'},
			\end{equation}
			where $\alpha, \alpha' \in \nat^{2(n - m)}$, $|\alpha| \leq \deg A_1, |\alpha'| \leq \deg A_2$.
			Now, by (\ref{eq_compa_1}) and (\ref{eq_compa_3}), we have
			\begin{multline}\label{eq_compa_4}
				\Big( (A_1 \cdot \mathscr{P}_{n, m}^{\perp}) \circ (A_2 \cdot \mathscr{P}_{n, m}^{\perp}) \Big)(Z, Z')
				=
				\exp \Big(- \frac{\pi}{2} \big( |Z_N|^2 + |Z'_N|^2 \big) \Big)
				\cdot
				\\
				\cdot
				\sum_{\alpha} \sum_{\alpha'} Z_N^{\alpha} Z'_N{}^{\alpha'} \cdot 
				\Big( (A_1^{\alpha} \cdot \mathscr{P}_n) \circ (A_2^{\alpha'} \cdot \mathscr{P}_n) \Big)(Z_Y, Z'_Y).
			\end{multline}
			From (\ref{eq_compa_1}) and (\ref{eq_compa_4}), we conclude
			\begin{equation}\label{eq_knm_from_knn}
				 \mathcal{K}_{n, m}[A_1, A_2](Z, Z')
				 =
				 \sum_{\alpha} \sum_{\alpha'} Z_N^{\alpha} Z'_N{}^{\alpha'}
				 \cdot
				  \mathcal{K}_{n, n}[A_1^{\alpha}, A_2^{\alpha'}](Z_Y, Z'_Y).
			\end{equation}
			Now, (\ref{eq_lem_comp_poly_1}) follows from (\ref{eq_knm_from_knn}) and the fact that (\ref{eq_lem_comp_poly_1}) holds for $n = m$.
			\par
			Along the same lines, from (\ref{eq_compa_1}), we obtain
			\begin{equation}\label{eq_knmpr_from_knn}
				 \mathcal{K'}_{n, m}[A_1, A_2](Z, Z')
				 =
				 \sum_{\alpha'} Z'_N{}^{\alpha'}
				 \cdot
				  \mathcal{K}_{n, n}[A_1, A_2^{\alpha'}](Z, Z'_Y),
			\end{equation}
			which implies (\ref{eq_lem_comp_poly_2}) since (\ref{eq_lem_comp_poly_1}) holds for $n = m$.
			\par
			Now, by (\ref{eq_escr}), we can write 
			\begin{equation}\label{eq_enm_pmfrm}
				\mathscr{E}_{n, m}(Z, Z'_Y) =
				\exp \Big(
				-\frac{\pi}{2} |Z_N|^2
				\Big)
				\cdot
				\mathscr{P}_m(Z_Y, Z'_Y).
			\end{equation}
			From (\ref{eq_lem_comp_poly_1}), (\ref{eq_compa_1}) and (\ref{eq_enm_pmfrm}), we deduce
			\begin{equation}\label{eq_kmprpr_form}
				\mathcal{K''}_{n, m}[A_1, A_2](Z, Z'_Y)
				=
				\mathcal{K}_{n, m}
				\Big[
				A_1, \mathcal{K}_{m, m}[1, A_2]
				\Big](Z, Z'_Y).
			\end{equation}
			Now, (\ref{eq_lem_comp_poly_3}) follows from (\ref{eq_kmprpr_form}) and the fact that (\ref{eq_lem_comp_poly_1}) holds.
			\par 
			Similarly, using the notations from (\ref{eq_compa_3}), we deduce
			\begin{equation}\label{eq_kmprpr_form1231}
				\mathcal{K'''}_{n, m}[A_1, A_2](Z, Z'_Y)
				=
				\sum_{\alpha} Z_N^{\alpha}
				\cdot
				\mathcal{K}_{m, m}
				\big[
				A_1^{\alpha}, A_2
				\big](Z_Y, Z'_Y),
			\end{equation}
			which clearly implies (\ref{eq_lem_comp_poly_4}).
		\end{proof}
		\begin{sloppypar}
		\begin{rem}\label{rem_k_calculc}
			Directly from the definitions, for a polynomial $P$, we have
			\begin{equation}\label{eq_k_calc_1}
			\begin{aligned}
				&
				\mathcal{K}_{n, m}[A_1 \cdot P(Z'), A_2]
				=
				\mathcal{K}_{n, m}[A_1, P(Z) \cdot A_2],
				\\
				&
				\mathcal{K}_{n, m}[A_1, A_2 \cdot P(Z')]
				=
				\mathcal{K}_{n, m}[A_1, A_2] \cdot P(Z').
			\end{aligned}
			\end{equation}
			Using the kernel calculus from \cite[\S 7.1]{MaHol} or the explicit calculations, one can verify, cf. \cite[(7.1.10)]{MaHol}, that for $i, j \leq m$, the following holds
			\begin{equation}\label{eq_k_calc_2}
				\begin{aligned}
				&
				\mathcal{K}_{n, m}[1, z_i z_j]
				=
				 z_i z_j, 
				&&
				\mathcal{K}_{n, m}[1, z_i \overline{z}_j]
				=
				\frac{1}{\pi} \delta_{ij} + z_i \overline{z}'_j,
				\\
				&
				\mathcal{K}_{n, m}[1, \overline{z}_i \overline{z}_j]
				=
				\overline{z}'_i \overline{z}'_j,
				&&
				\\
				&
				\mathcal{K}_{n, m}[1, P_i(Z) z_i]
				=
				\mathcal{K}_{n, m}[1, P_i(Z)]
				z_i,
				&&
				\mathcal{K}_{n, m}[1, P_i(Z) \overline{z}_i]
				=
				\mathcal{K}_{n, m}[1, P_i(Z)]
				\overline{z}'_i,
				\end{aligned}
			\end{equation}
			where the polynomial $P_i(Z)$ doesn't involve the variables $z_i$ and $\overline{z}_i$.
			Finally, for any $k = 2m + 1, \ldots, 2n$, from the trivial fact $\int_{\real} Z \exp(- \pi |Z|^2) dZ = 0$, we get
			\begin{equation}\label{eq_k_calc_3}
				\mathcal{K}_{n, m} \Big[ A_1(Z_Y, Z'_Y), Z_k \cdot A_2(Z_Y, Z'_Y) \Big]
				=
				0.
			\end{equation}
			\begin{comment}
				From (\ref{eq_knmpr_from_knn}), we get
				\begin{equation}
					\mathcal{K'}_{n, m} \Big[A_1(Z, Z'), A_2(Z, Z'_Y) \Big](Z, Z')
					=
					\mathcal{K}_{n, n} \Big[A_1(Z, Z'), A_2(Z, Z'_Y) \Big](Z, Z'_Y).
				\end{equation}
				From this and (\ref{eq_k_calc_2}), we deduce, in particular that $\mathcal{K'}_{n, m}[\cdot, \cdot]$ verifies identities as in (\ref{eq_k_calc_2}) for $i, j \leq m$.
				From (\ref{eq_knmpr_from_knn}), for $i \leq m$ and $j = m+1, \ldots, n$, we also deduce
				\begin{equation}\label{eq_kpr_calc_2}
					\begin{aligned}
					&
					\mathcal{K'}_{n, m}[1, z_i z_j]
					=
					 z_i z_j, 
					&&
					\mathcal{K'}_{n, m}[1, z_i \overline{z}_j]
					=
					0,
					\\
					&
					\mathcal{K'}_{n, m}[1, \overline{z}_i \overline{z}_j]
					=
					0,
					&&
					\mathcal{K'}_{n, m}[1, \overline{z}_i z_j]
					=
					z_j \overline{z}'_i.
					\\
					&
					\mathcal{K'}_{n, m}[1,  P_j(Z)  z_j]
					=
					\mathcal{K'}_{n, m}[1,  P_j(Z)]
					z_j,
					&&
					\mathcal{K'}_{n, m}[1,  P_j(Z) \overline{z}_j]
					=
					0,
					\end{aligned}
				\end{equation}
				where we use notation as in (\ref{eq_k_calc_2}).
				Finally, again from (\ref{eq_knmpr_from_knn}), for $k, j = m+1, \ldots, n$, we get
				\begin{equation}\label{eq_kpr_calc_3}
					\mathcal{K'}_{n, m}[1, z_j z_k]
					=
					z_j z_k,
					\quad
					\mathcal{K'}_{n, m}[1, z_j \overline{z}_k]
					=
					\frac{1}{\pi} \delta_{jk},
					\quad
					\mathcal{K'}_{n, m}[1, \overline{z}_j \overline{z}_k]
					=
					0.
				\end{equation} 
			\end{comment}
		\end{rem}
		\end{sloppypar}

	\subsection{Algebra of operators with Taylor-type expansion of the Schwartz kernel}\label{sect_algtay_type}
		The main goal of this section is to prove that the set of operators, acting on the sections of a trivial vector bundle over $\comp^n$ by the convolution with Schwartz kernel admitting Taylor-type expansion and exponential decay away from $\comp^m \subset \comp^n$, forms an algebra under composition.
		We also extend this result for operators on general triples of bounded geometry.
		\par 
		Let us now state precisely our results.
		The proofs will be given in the end of this section.
		We fix $q \in \nat$, $q \geq 2$, and operators $\mathcal{G}_t$, $\mathcal{A}_t^{1}, \ldots, \mathcal{A}_t^{q}$, $t \in [0, 1]$, acting on the sections of the trivial vector bundle $\comp^{r_0} \times \comp^n$ over $\comp^n$ by the convolutions with smooth kernels $\mathcal{G}_t(Z, Z')$, $\mathcal{A}_t^{1}(Z, Z'), \ldots, \mathcal{A}_t^{q}(Z, Z') \in \enmr{\comp^{r_0}}$ with respect to the volume form $dv_{\comp^n}$ on $\comp^n$.
		We assume that there are $c_0, q_1 > 0$, such that for any $l \in \nat$, there are $C > 0$, $Q_{h, 1} \geq 0$, $h = 1, \ldots, q$, such that for any $t \in [0, 1]$, $Z, Z' \in \real^{2n}$, $\alpha, \alpha' \in \nat^{2n}$, $|\alpha| + |\alpha'| \leq l$, we have
		\begin{align}
			& \nonumber
			\bigg| \frac{\partial^{|\alpha|+|\alpha'|}}{\partial Z^{\alpha} \partial Z'{}^{\alpha'}} \mathcal{A}_t^{h}(Z, Z') \bigg|
			 \leq 
			 C \Big(1 + |Z| + |Z'| \Big)^{Q_{h, 1} + q_1 l} \exp\Big(- c_0 \big( |Z_Y - Z'_Y| + |Z_N| + |Z'_N| \big) \Big),
			 \\ \label{eq_a_bnd_1loc}
			 &
			 \bigg| \frac{\partial^{|\alpha|+|\alpha'|}}{\partial Z^{\alpha} \partial Z'{}^{\alpha'}} \mathcal{G}_t(Z, Z') \bigg|
			 \leq 
			 C \Big(1 + |Z| + |Z'| \Big)^{Q_{1, 1} + q_1 l} \exp\Big(- c_0 |Z - Z'|  \Big).
		\end{align}
		\begin{lem}\label{lem_bnd_prod_aloc}
			The operators $\mathcal{D}_t := \mathcal{A}_t^{1} \circ \cdots \circ \mathcal{A}_t^{q}$,  $\mathcal{D}'_t := \mathcal{G}_t \circ \mathcal{A}_t^{2} \circ \cdots \circ \mathcal{A}_t^{q}$ are well-defined and have smooth Schwartz kernels $\mathcal{D}_t(Z, Z')$, $\mathcal{D}'_t(Z, Z')$ with respect to $dv_{\comp^n}$. 
			Moreover, for any $l \in \nat$, there is $C > 0$, such that for any $t \in [0, 1]$, $Z, Z' \in \real^{2n}$, $\alpha, \alpha' \in \nat^{2n}$, $|\alpha| + |\alpha'| \leq l$, we have
			\begin{multline}\label{eq_bnd_prod_aloc}
				\bigg| \frac{\partial^{|\alpha|+|\alpha'|}}{\partial Z^{\alpha} \partial Z'{}^{\alpha'}} \mathcal{R}_t(Z, Z') \bigg|
				\leq C \Big(1 + |Z| + |Z'| \Big)^{Q_{1, 1} + \cdots +Q_{q, 1} + q_1 l} \cdot
				\\
				\cdot
				\exp\Big(- \frac{c_0}{8} \big( |Z_Y - Z'_Y| + |Z_N| + |Z'_N| \big) \Big),
			\end{multline}
			where $\mathcal{R}_t$ designates either $\mathcal{D}_t$ or $\mathcal{D}'_t$.
		\end{lem}
		\par 
		Now, assume, in addition to (\ref{eq_a_bnd_1loc}), that for any $r \in \nat$, $h \in 1, \ldots, q$, there are $\mathcal{J}_r^{h}(Z, Z') \in \enmr{\comp^{r_0}}$ polynomials in $Z, Z' \in  \real^{2n}$, such that $\mathcal{F}_r^{h} := \mathcal{J}_r^{h} \cdot \mathscr{P}_{n, m}^{\perp}$ (resp. $\mathcal{F}'_r := \mathcal{J}_r^{1} \cdot \mathscr{P}_n$) appear as coefficients of Taylor-type expansion for $\mathcal{A}_t^{h}$ (resp. $\mathcal{G}_t$).
		More precisely, we assume that there are $\epsilon_0, c_1, q_1, q_2 > 0$, such that for any $k, l \in \nat$, $h = 1, \ldots, q$, there are $C > 0$, $Q_{h, 2} \geq 0$, such that for any $t \in [0, 1]$, $Z, Z' \in \real^{2n}$, $|Z|, |Z'| \leq \frac{\epsilon_0}{t}$, $\alpha, \alpha' \in \nat^{2n}$, $|\alpha| + |\alpha'| \leq l$, we have
		\begin{multline}\label{eq_at_tayl_type}
			\bigg| 
			\frac{\partial^{|\alpha|+|\alpha'|}}{\partial Z^{\alpha} \partial Z'{}^{\alpha'}}
			\bigg(
					\mathcal{A}_t^{h}(Z, Z')
					-
					\sum_{r = 0}^{k}
					t^{r}						
					\mathcal{F}_{r}^h(Z, Z') 
			\bigg)
			\bigg|
			\\
			\leq
			C t^{k + 1}
			\Big(1 + |Z| + |Z'| \Big)^{Q_{h, 2} + q_1 l + q_2 k}
			\exp\Big(- c_1 \big( |Z_Y - Z'_Y| + |Z_N| + |Z'_N| \big) \Big).		
		\end{multline}
		\vspace*{-1.15cm}
		\begin{multline}\label{eq_at_tayl_t_exp22}
			\bigg| 
			\frac{\partial^{|\alpha|+|\alpha'|}}{\partial Z^{\alpha} \partial Z'{}^{\alpha'}}
			\bigg(
					\mathcal{G}_t(Z, Z')
					-
					\sum_{r = 0}^{k}
					t^{r}						
					\mathcal{F}'_{r}(Z, Z') 
			\bigg)
			\bigg|
			\\
			\leq
			C t^{k + 1}
			\Big(1 + |Z| + |Z'| \Big)^{Q_{1, 2} + q_1 l + q_2 k}
			\exp\big(- c_1  |Z - Z'|  \big).
		\end{multline} 
		Define polynomials $\mathcal{J}_{r, 0}(Z, Z') \in \enmr{\comp^r}$ (resp. $\mathcal{J}'_{r, 0}(Z, Z') \in \enmr{\comp^r}$), $r \in \nat$, in $Z, Z' \in  \real^{2n}$, as follows
			\begin{equation}\label{eq_part_calc}
			\begin{aligned}
				&
				 \mathcal{J}_{r, 0}
				 =
				 \sum_{\lambda}
				 \mathcal{K}_{n, m} \Big[ \mathcal{J}_{\lambda_1}^{1},  \mathcal{K}_{n, m} \Big[ \mathcal{J}_{\lambda_2}^{2}, \cdots  ,  \mathcal{K}_{n, m} \big[ \mathcal{J}_{\lambda_{q-1}}^{q-1}, \mathcal{J}_{\lambda_{q}}^{q}  \big] \cdots \Big],
				 \\
				 &
				 \mathcal{J}'_{r, 0}
				 =
				 \sum_{\lambda}
				 \mathcal{K'}_{n, m} \Big[ \mathcal{J}_{\lambda_1}^{1},  \mathcal{K}_{n, m} \Big[ \mathcal{J}_{\lambda_2}^{2}, \cdots  ,  \mathcal{K}_{n, m} \big[ \mathcal{J}_{\lambda_{q-1}}^{q-1}, \mathcal{J}_{\lambda_{q}}^{q} \big] \cdots \Big],
			\end{aligned}
			\end{equation}
			where $\lambda$ runs over all partitions $(\lambda_1, \cdots, \lambda_q)$ of $r$ by natural numbers $\lambda_i$, and $\mathcal{K}_{n, m}$, $\mathcal{K'}_{n, m}$ are as in Lemma \ref{lem_comp_poly}. 
			We let $\mathcal{F}_{r, 0} := \mathcal{J}_{r, 0} \cdot \mathscr{P}_{n, m}^{\perp}$  (resp. $\mathcal{F}'_{r, 0} := \mathcal{J}'_{r, 0} \cdot \mathscr{P}_{n, m}^{\perp}$).
		\begin{lem}\label{lem_at_tayl_t_exp}
			In the notations of (\ref{eq_a_bnd_1loc}), (\ref{eq_at_tayl_type}), for any $k, l \in \nat$, there is $C > 0$, such that for any $t \in ]0, 1]$, $Z, Z' \in \real^{2n}$, $|Z|, |Z'| \leq \frac{\epsilon_0}{2t}$, $\alpha, \alpha' \in \nat^{2n}$, $|\alpha| + |\alpha'| \leq l$, and for $Q := \max \{ \sum_{h = 1}^{q} Q_{h, 1}, \sum_{h = 1}^{q} Q_{h, 2} \}$, $c_2 := \min \{ c_0, c_1 \}$, the following estimate holds
		\begin{multline}\label{eq_at_tayl_t_exp}
			\bigg| 
			\frac{\partial^{|\alpha|+|\alpha'|}}{\partial Z^{\alpha} \partial Z'{}^{\alpha'}}
			\bigg(
					R_t(Z, Z')
					-
					\sum_{r = 0}^{k}
					t^{r}						
					\mathcal{G}_{r, 0}(Z, Z') 
			\bigg)
			\bigg|
			\\
			\leq
			C t^{k + 1}
			\Big(1 + |Z| + |Z'| \Big)^{Q + q_1 l + q_2 k}
			\exp\Big(- \frac{c_2}{8} \big( |Z_Y - Z'_Y| + |Z_N| + |Z'_N| \big) \Big),
		\end{multline}
		where $\mathcal{R}_t$ (resp. $\mathcal{G}_{r, 0}$) designates either $\mathcal{D}_t$ or $\mathcal{D}'_t$ (resp. either $\mathcal{F}_{r, 0}$ or $\mathcal{F}'_{r, 0}$).
		\end{lem}
		\par 
		Now, let us formulate similar results for general complex manifolds.
		More precisely, let $(X, Y, g^{TX})$ be a triple of bounded geometry. 
		Let $dv_X$ be a volume form over $X$, satisfying (\ref{eq_vol_comp_unif}).
		We fix a Hermitian line (resp. vector) bundle $(L, h^{L})$ (resp. $(F, h^{F})$) on $X$.
		\par 
		Let us fix $q \in \nat^*$, and a sequence of operators $A_p^{1}, \ldots, A_p^{q}$, $p \in \nat^*$, acting on $\ccal^{\infty}(X, L^p \otimes F)$ by the convolutions with smooth kernels $A_p^{1}(x_1, x_2), \ldots, A_p^{q}(x_1, x_2) \in (L^p \otimes F)_{x_1} \otimes  (L^p \otimes F)_{x_2}^{*}$ with respect to the volume form $dv_X$.
		We assume that $A_p^{1}, \ldots, A_p^{q}$ satisfy the assumptions (\ref{eq_a_bnd_1}) and (\ref{eq_a_bnd_2}) for $n := 2n$ and $W := Y$.
		\par 
		We fix $y_0 \in Y$ and trivialize $(L, h^{L})$ (resp. $(F, h^{F})$) and associated dual vector bundles in a neighborhood of $y_0$ using Fermi coordinates and parallel transport with respect to $\nabla^{L}$ (resp. $\nabla^{F}$) as we did before Theorem \ref{thm_ext_as_exp}.
		Assume that for any $h = 1, \ldots, q$, $r \in \nat$, there are $J_{r, h}^{\perp}(Z, Z') \in \enmr{F_{y_0}}$ polynomials in $Z, Z' \in \real^{2n}$, such that for $F_{r, h}^{\perp} := J_{r, h}^{\perp} \cdot \mathscr{P}_{n, m}^{\perp}$, the following holds.
		\par 
		There are $\epsilon_0, c_1, q_1, q_2 > 0$, $p_1 \in \nat^*$, such that for any $k, l \in \nat$, $h = 1, \ldots, q$, there are $C  > 0$, $Q_{h, 3} \geq 0$, such that for any $p \geq p_1$, $Z = (Z_Y, Z_N)$, $Z' = (Z'_Y, Z'_N)$, $Z_Y, Z'_Y \in \real^{2m}$, $Z_N, Z'_N \in \real^{2(n-m)}$, $|Z|, |Z'| \leq \epsilon_0$, $\alpha, \alpha' \in \nat^{2n}$, $|\alpha|+|\alpha'| \leq l$, the following bound holds
		\begin{multline}\label{eq_aph_bergm_like_orth}
			\bigg| 
				\frac{\partial^{|\alpha|+|\alpha'|}}{\partial Z^{\alpha} \partial Z'{}^{\alpha'}}
				\bigg(
					\frac{1}{p^n} A_p^{h} \big(\psi_{y_0}(Z), \psi_{y_0}(Z') \big)
					-
					\sum_{r = 0}^{k}
					p^{-\frac{r}{2}}						
					F_{r, h}^{\perp}(\sqrt{p} Z, \sqrt{p} Z') 
					\kappa_{X}^{-\frac{1}{2}}(Z)
					\kappa_{X}^{-\frac{1}{2}}(Z')
				\bigg)
			\bigg|
			\\
			\leq
			C p^{- \frac{k + 1 - l}{2} }
			\Big(1 + \sqrt{p}|Z| + \sqrt{p} |Z'| \Big)^{Q_{h, 3} + q_1 l + q_2 k}
			\cdot
			\\
			\cdot
			\exp\Big(- c_1 \sqrt{p} \big( |Z_Y - Z'_Y| + |Z_N| + |Z'_N| \big) \Big).
		\end{multline}
		Define polynomials $J_{r, D}^{\perp}(Z, Z') \in \enmr{F_{y_0}}$, $r \in \nat$, in $Z, Z' \in  \real^{2n}$, as follows
		\begin{equation}\label{eq_part_calc_dp}
			 J_{r, D}^{\perp}
			 =
			 \sum_{\lambda}
			 \mathcal{K}_{n, m} \Big[ J_{\lambda_1, 1}^{\perp}, \mathcal{K}_{n, m} \Big[ J_{\lambda_2, 2}^{\perp}, \cdots , \mathcal{K}_{n, m} \big[ J_{\lambda_{q-1}, q-1}^{\perp}, J_{\lambda_q, q}^{\perp}  \big] \cdots \Big],
		\end{equation}
		where $\lambda$ runs over all partitions $(\lambda_1, \cdots, \lambda_q)$ of $r$ by natural numbers $\lambda_i$, and let $F_{r, D}^{\perp} := J_{r, D}^{\perp} \cdot \mathscr{P}_{n, m}^{\perp}$.
		\begin{lem}\label{lem_apl_tayl_exp}
			In the notations of (\ref{eq_a_bnd_1}), (\ref{eq_a_bnd_2}), (\ref{eq_aph_bergm_like_orth}), for any $k, l \in \nat$, there is $C  > 0$, such that for any $p \geq p_1$, $Z = (Z_Y, Z_N)$, $Z' = (Z'_Y, Z'_N)$, $Z_Y, Z'_Y \in \real^{2m}$, $Z_N, Z'_N \in \real^{2(n-m)}$, $|Z|, |Z'| \leq \frac{\epsilon_0}{2}$, $\alpha, \alpha' \in \nat^{2n}$, $|\alpha|+|\alpha'| \leq l$, for $D_p$ as in (\ref{eq_dp_defn_apl}), $Q := \sum_{h = 1}^{q} Q_{h, 3}$, $c_2 = \min\{c_0, c_1 \}$, we have
		\begin{multline}\label{eq_dp_bergm_like_orth}
			\bigg| 
				\frac{\partial^{|\alpha|+|\alpha'|}}{\partial Z^{\alpha} \partial Z'{}^{\alpha'}}
				\bigg(
					\frac{1}{p^n} D_p \big(\psi_{y_0}(Z), \psi_{y_0}(Z') \big)
					-
					\sum_{r = 0}^{k}
					p^{-\frac{r}{2}}						
					F_{r, D}^{\perp}(\sqrt{p} Z, \sqrt{p} Z') 
					\kappa_{X}^{-\frac{1}{2}}(Z)
					\kappa_{X}^{-\frac{1}{2}}(Z')
				\bigg)
			\bigg|
			\\
			\leq
			C p^{- \frac{k + 1 - l}{2} }
			\Big(1 + \sqrt{p}|Z| + \sqrt{p} |Z'| \Big)^{Q + q_1 l + q_2 k}
			\cdot
			\\
			\cdot
			\exp\Big(- \frac{c_2}{8} \sqrt{p} \big( |Z_Y - Z'_Y| + |Z_N| + |Z'_N| \big) \Big).
		\end{multline}
		\end{lem}
		\par 
		Let $dv_Y$ be a volume form over $Y$, satisfying (\ref{eq_vol_comp_unif}).
		We fix an operator $C_p$, $p \in \nat^*$, acting on $\ccal^{\infty}(Y, \iota^*( L^p \otimes F))$ by the convolution with smooth kernel $C_p(y_1, y_2)$, $y_1, y_2 \in Y$, with respect to the volume form $dv_Y$.
	 	Assume there is $c_2 > 0$, such that for any $k \in \nat$, there is $C > 0$, such that
		\begin{equation}\label{eq_cp_bnd_1}
			\big| C_p(y_1, y_2) \big|_{\ccal^k(Y \times Y)} \leq C p^{m + \frac{k}{2}} \exp \big(- c_2 \sqrt{p} \dist(y_1, y_2) \big).
		\end{equation}
		\par 
		Assume that for a fixed $y_0 \in Y$, for any $r \in \nat$, there are $J_{r, C}(Z_Y, Z'_Y) \in \enmr{F_{y_0}}$ polynomials in $Z_Y, Z'_Y \in \real^{m}$, such that for $F_{r, C} := J_{r, C} \cdot \mathscr{P}_m$, the following condition holds.
		There are $\epsilon_0, c_3, q_1, q_2 > 0$, $p_1 \in \nat^*$,  such that for $k, l \in \nat$, there are $C > 0$, $Q_{C} \geq 0$, such that for $p \geq p_1$, $Z_Y, Z'_Y \in \real^{2m}$, $|Z_Y|, |Z'_Y| \leq \epsilon_0$, $\alpha, \alpha' \in \nat^{2m}$, $|\alpha|+|\alpha'| \leq l$, the following bound holds
		\begin{multline}\label{eq_apl_ass_tay_cp}
			\bigg| 
				\frac{\partial^{|\alpha|+|\alpha'|}}{\partial Z_Y^{\alpha} \partial Z'_Y{}^{\alpha'}}
				\bigg(
					\frac{1}{p^m} C_p\big(\psi_{y_0}(Z_Y), \psi_{y_0}(Z'_Y) \big)
					-
					\sum_{r = 0}^{k}
					p^{-\frac{r}{2}}						
					F_{r, C}(\sqrt{p} Z_Y, \sqrt{p} Z'_Y) 
					\kappa_{Y}^{-\frac{1}{2}}(Z_Y)
					\kappa_{Y}^{-\frac{1}{2}}(Z'_Y)
				\bigg)
			\bigg|
			\\
			\leq
			C p^{- \frac{k + 1 - l}{2}}			
			\Big(1 + \sqrt{p}|Z_Y| + \sqrt{p} |Z'_Y| \Big)^{Q_{C} + q_1 l + q_2 k}
			\exp\Big(- c_3 \sqrt{p} |Z_Y - Z'_Y| \Big).		
		\end{multline}
		\par 
		Recall that $\kappa_N$ and $\ext_p^{0}$ were defined in (\ref{eq_kappan}) and (\ref{eq_ext0_op}) respectively.
		We denote 
		\begin{equation}\label{eq_dp_ep_kapp}
			D_p := A_p^{1} \circ (\kappa_N^{-\frac{1}{2}} \cdot \ext_p^{0}) \circ C_p.
		\end{equation}
		Define polynomials $J_{r, D}^{E}(Z, Z') \in \enmr{F_{y_0}}$, $r \in \nat$, in $Z \in  \real^{2n}$, $Z '\in  \real^{2m}$, by
		\begin{equation}\label{eq_part_calc_dp}
			J_{r, D}^{E}
			=
			\sum_{r_0 = 0}^{r}
			\mathcal{K''}_{n, m}[J_{r_0, 1}^{\perp}, J_{r - r_0, C}],
		\end{equation}
		and let $F_{r, D}^{E} := J_{r, D}^{E} \cdot \mathscr{E}_{n, m}$.
		\begin{lem}\label{lem_ac_ext_op_tay}
			In the notations of (\ref{eq_a_bnd_1}), (\ref{eq_aph_bergm_like_orth}), (\ref{eq_cp_bnd_1}), (\ref{eq_apl_ass_tay_cp}), for any $k, l \in \nat$, there is $C  > 0$, such that for any $p \geq p_1$, $Z = (Z_Y, Z_N)$, $Z_Y, Z'_Y \in \real^{2m}$, $Z_N \in \real^{2(n - m)}$ $|Z|, |Z'_Y| \leq \frac{\epsilon_0}{2}$, $\alpha \in \nat^{2n}$, $\alpha' \in \nat^{2m}$, $|\alpha| + |\alpha'| \leq l$, $Q := Q_{1, 3} + Q_{C}$, $c_4 := \min \{ c_0, c_1, c_2, c_3 \}$, we have
			\begin{multline}\label{eq_ext_as_expdp}
				\bigg| 
					\frac{\partial^{|\alpha|+|\alpha'|}}{\partial Z^{\alpha} \partial Z'_Y{}^{\alpha'}}
					\bigg(
						\frac{1}{p^m} D_p \big(\psi_{y_0}(Z), \psi_{y_0}(Z'_Y) \big)
						-
						\sum_{r = 0}^{k}
						p^{-\frac{r}{2}}						
						F_{r, D}^{E}(\sqrt{p} Z, \sqrt{p} Z'_Y) 
						\kappa_{X}^{-\frac{1}{2}}(Z)
						\kappa_{Y}^{-\frac{1}{2}}(Z'_Y)
					\bigg)
				\bigg|
				\\
				\leq
				C p^{- \frac{k + 1 - l}{2}}
				\Big(1 + \sqrt{p}|Z| + \sqrt{p} |Z'_Y| \Big)^{Q + q_1 l + q_2 k}
				\exp\Big(- \frac{c_4}{8} \sqrt{p} \big( |Z_Y - Z'_Y| + |Z_N| \big) \Big).
			\end{multline}
		\end{lem}
		\begin{proof}[Proof of Lemma \ref{lem_bnd_prod_aloc}]
			The proof is completely analogous to the proof of Lemma \ref{lem_bnd_prod_a}, done for $W := \comp^m$. The only change needed in the proof is that instead of Proposition \ref{prop_exp_bound_int}, one has to use that for any $c > 0$, $Q \geq 0$, there is a constant $C$, such that for any $A \in \comp^{n}$, we have
			\begin{equation}\label{eq_bnd_intlocal}
				\int_{\comp^n}
				|Z|^Q 
				\exp (- c |Z - A| )
				dZ_1 \wedge \cdots \wedge dZ_{2n}
				\leq 
				C (1 + |A|^Q).
			\end{equation}
			This fact can be shown easily by the change of variables $Z \mapsto Z + A$.
		\end{proof}
		\begin{proof}[Proof of Lemma \ref{lem_at_tayl_t_exp}]
			To simplify the presentation, we restrict ourselves to the case $l = 0$, as the general case is treated in an analogous way. 
			We present first the proof that the asymptotic expansion (\ref{eq_at_tayl_t_exp}) holds for $k = 0$.
			We only treat the asymptotic expansion of $\mathcal{D}_t$, as for $\mathcal{D}'_t$ the proof is analogous.
			Analogously to (\ref{eq_b_f}), we have
			\begin{multline}\label{eq_b_floc_comp00}
				\mathcal{D}_t(Z, Z')
				=
				\int_{(\real^{2n})^{\times (q - 1)}} \mathcal{A}_t^{1}(Z, Z_1) 
				\mathcal{A}_t^{2}(Z_1, Z_2) \cdot \ldots \cdot \mathcal{A}_t^{q}(Z_{q - 1}, Z') 
				\cdot 
				\\
				\cdot
				dv_{\real^{2n}}(Z_1) \cdot \ldots \cdot dv_{\real^{2n}}(Z_{q-1})
			\end{multline}
			We decompose the integral (\ref{eq_b_floc_comp00}) into a sum of two integrals.
			The first one is over the set $Q_t := \mathbb{B}_0^{\real^{2n}}(\frac{\epsilon_0}{t})^{\times(q-1)}$, and the second one is over its complement $Q_t^{c}$.
			Let us bound the contribution in (\ref{eq_b_floc_comp00}) coming from the integral on $Q_t^{c}$.
			Similarly to (\ref{eq_triangle}), for any $Z, Z' \in \real^{2n}$, $|Z|, |Z'| \leq \frac{\epsilon_0}{2 t}$, and any $(Z_1, \ldots, Z_{s-1}) \in Q_t^{c}$, we have
			\begin{equation}
				\dist(Z, Z_1) + \dist(Z_1, Z_2) + \cdots + \dist(Z_{s - 1}, Z') \geq \frac{\epsilon_0}{2 t}.
			\end{equation}
			By (\ref{eq_a_bnd_1loc}), similarly to the proof of Lemma \ref{lem_bnd_prod_a} and (\ref{eq_bnd_intlocal}), there is $C > 0$, such that
			\begin{multline}\label{eq_b_floc_comp}
				\bigg|				
				\int_{Q_t^{c}} \mathcal{A}_t^{1}(Z, Z_1) \cdot \mathcal{A}_t^{2}(Z_1, Z_2) \cdot \ldots \cdot \mathcal{A}_t^{q}(Z_{q - 1}, Z') dv_{\comp^n}(Z_1) \cdot \ldots \cdot dv_{\comp^n}(Z_{q-1})
				\bigg|
				\\
				\leq C (1 + |Z| + |Z'|)^{\sum_{h = 1}^{q} Q_{h, 1}} \exp\Big(- \frac{c}{8} \Big( |Z_Y - Z'_Y| + |Z_N| + |Z'_N| + \frac{\epsilon_0}{2 t} \Big) \Big).
			\end{multline}
			From the fact that for any $c, \epsilon > 0$, $k \in \nat$, there is a constant $C > 0$, such that $\exp(- c\epsilon/t) < C t^k$, for any $t \in ]0, 1]$, implies that the right-hand side of, (\ref{eq_b_floc_comp}) is majorated by the right-hand side of (\ref{eq_at_tayl_type}), and hence, the contribution of the integral over $Q_t^{c}$ is negligible. 
			\par 
			Now, let $A_1$, $A_2$, $A_3$ be as in Lemma \ref{lem_comp_poly}.
			By (\ref{eq_b_floc_comp}), we see that for any $\epsilon > 0$, there is $C > 0$, which depends only on $k, A_1, A_2$, such that for $Z, Z' \in \real^{2n}$, $|Z|, |Z'| \leq \frac{\epsilon}{2 t}$, we have
			\begin{multline}\label{eq_a3_part_int}
				\bigg|
				(A_3 \cdot \mathscr{P}_{n, m}^{\perp})(Z, Z')
				-				
				\int_{|Z_1| < \frac{\epsilon}{t}} (A_1 \cdot \mathscr{P}_{n, m}^{\perp})(Z, Z_1) \cdot (A_2 \cdot \mathscr{P}_{n, m}^{\perp})(Z_1, Z')
				dv_{\comp^n}(Z_1)
				\bigg|
				\\
				\leq C (1 + |Z| + |Z'|)^{\deg A_1 + \deg A_2} \exp\Big(- \frac{c}{8} \Big( |Z_Y - Z'_Y| + |Z_N| + |Z'_N| + \frac{\epsilon}{2 t} \Big) \Big).
			\end{multline}
			Of course, an estimate as in (\ref{eq_a3_part_int}), holds for any number of polynomials.
			From this and (\ref{eq_b_floc_comp}), we deduce the first part of Lemma \ref{lem_at_tayl_t_exp} for $k = 0$.
			\par Now, let us extend the argument for general $k \in \nat$. For $k' \in \nat$, in the notations of (\ref{eq_at_tayl_type}), let us denote
			\begin{equation}
				\mathcal{B}_t^{l, k'}(Z, Z')
				:=
				\sum_{r = 0}^{k'}
				t^r						
				\mathcal{F}_r^l(Z, Z').
			\end{equation}
			Then we decompose $\mathcal{D}_t$ into a sum of elements $\mathcal{C}_t^{1} \circ \cdots \circ \mathcal{C}_t^{q}$, where each $\mathcal{C}_t^{l}$, $l = 1, \ldots, q$, is either equal to $\mathcal{A}_t^{l} - \mathcal{B}_t^{l, k'}$ or to $\mathcal{B}_t^{l, k'}$, and the sum of $k'$ for all the multiplicands adds up to $k$.
			For the term, which consists of the compositions of $\mathcal{B}_t^{l}$, we apply (\ref{eq_b_floc_comp}) and (\ref{eq_a3_part_int}) to conclude that it has an asymptotic expansion of the form (\ref{eq_at_tayl_t_exp}) up to the error term of the form as in the right-hand side of (\ref{eq_b_floc_comp}).
			From (\ref{eq_a3_part_int}) and the definition of $\mathcal{K}_{n, m}[\cdot, \cdot]$ from Lemma \ref{lem_comp_poly}, the coefficients of this asymptotic expansion are given by polynomials $\mathcal{J}_{r, 0}$, defined in (\ref{eq_part_calc}).
			\par 
			For other terms, we decompose the integral into two parts: over $Q_t$ and over $Q_t^c$. 
			To bound the contribution from $Q_t$, we use (\ref{eq_at_tayl_type}) and Lemma \ref{lem_bnd_prod_aloc}.
			To bound the contribution from $Q_t^c$, we use (\ref{eq_a_bnd_1loc}) and proceed as in (\ref{eq_b_floc_comp}).
			By combining the two bounds, we see that they contribute no more than the right-hand side of (\ref{eq_at_tayl_type}).
			This finishes the proof.
		\end{proof}
		\begin{proof}[Proof of Lemma \ref{lem_apl_tayl_exp}]
			The proof is very similar to the proof of Lemma \ref{lem_at_tayl_t_exp}, so we only highlight the main steps.
			We fix $Z, Z' \in \real^{2n}$, $|Z|, |Z'| < \frac{\epsilon_0}{2}$, and let $x = \psi_{y_0}(Z)$, $x' = \psi_{y_0}(Z')$.
			We decompose the integral in the formula (\ref{eq_b_f}), into two parts.
			The first one is over the set $Q := \mathbb{B}_{y_0}^{X}(\epsilon_0)^{\times(q-1)}$, and the second one is over its complement $Q^c$.
		Similarly to (\ref{eq_b_floc_comp}), but relying on Proposition \ref{prop_exp_bound_int}, we deduce the bound
			\begin{multline}\label{eq_smacontr}
				\bigg|				
				\int_{Q^c} A_p^{1}(x, x_1) \cdot A_p^{2}(x_1, x_2) \cdot \ldots \cdot A_p^{q}(x_{q - 1}, x') dv_X(x_1) \cdot \ldots \cdot dv_X(x_{q-1})
				\bigg|
				\\
				\leq
				C p^{n + \frac{k}{2}}  \cdot \exp \Big(- \frac{c}{8} \sqrt{p} \cdot \big(  \dist(x, x') + \dist(x, Y) + \dist(x', Y) + \frac{\epsilon_0}{2} \big) \Big),
			\end{multline}
			But since for any $c, \epsilon > 0$, $k \in \nat$, there is $C > 0$, such that $\exp(- c\epsilon \sqrt{p}) < C p^{-k}$, for any $p \in \nat^*$, the right-hand side of (\ref{eq_smacontr}) is majorated by the right-hand side of (\ref{eq_dp_bergm_like_orth}) and, hence, negligible.
			Now, to deal with the integration over $Q$, we pass to $\psi_{y_0}$-coordinates.
			The change of the variables introduces the $\kappa_X$ factor for every volume form.
			It will be canceled with the two $\kappa_X^{-1/2}$ factors, which appear in the asymptotic expansion (\ref{eq_aph_bergm_like_orth}).
			Once two factors are canceled, and we make a change of the variables $Z \mapsto \frac{Z}{\sqrt{p}}$, we reduce the problem to the estimates of the form (\ref{eq_a3_part_int}).
			The proof is now finished exactly as in the proof of Lemma \ref{lem_at_tayl_t_exp}.
		\end{proof}
		\begin{proof}[Proof of Lemma \ref{lem_ac_ext_op_tay}]
			The proof is analogous to the proof of Lemma \ref{lem_apl_tayl_exp} with only one change: instead of relying on  (\ref{eq_lem_comp_poly_1}) in the estimate (\ref{eq_a3_part_int}), one has to rely on (\ref{eq_lem_comp_poly_3}).
			The reason why the factor $\kappa_N^{-\frac{1}{2}}$ appears in (\ref{eq_dp_ep_kapp}) is due to the identity 
			\begin{equation}\label{eq_kappas_rel}
				\kappa_X(Z)
				=
				\kappa_N(\psi_{y_0}(Z))
				\cdot
				\kappa_Y(Z_Y),
			\end{equation}
			which implies that the term $\kappa_X$, appearing after the passage to $\psi_{y_0}$-coordinates, disappear with the terms $\kappa_X^{-\frac{1}{2}}$ and $\kappa_Y^{-\frac{1}{2}}$, which appear in the Taylor-type expansions of $A_p$ and $C_p$.	
		\end{proof}
		
\section{Spectral bound for the restriction operator}\label{sect_spec_bnd_res}
	The main goal of this section is to prove a spectral bound for the restriction operator.
	\begin{sloppypar}
	More precisely, we conserve the notation from Section \ref{sect_intro} and assume that the triple $(X, Y, g^{TX})$ is of bounded geometry.
	Consider the restriction operator $\res_p : H^{0, \perp}_{(2)}(X, L^p \otimes F) \to \ccal^{\infty}(Y, \iota^*(L^p \otimes F))$, defined as $f \mapsto f|_Y$.
	As we prove in Proposition \ref{prop_restr_is_l2}, the restriction to $Y$ of an element from $H^{0, \perp}_{(2)}(X, L^p \otimes F)$ lies in $H^{0}_{(2)}(Y, \iota^*(L^p \otimes F))$.
	Hence, we may view $\res_p$ as
	\begin{equation}\label{eq_res_op_l2_form}
		\res_p : H^{0, \perp}_{(2)}(X, L^p \otimes F) \to H^{0}_{(2)}(Y, \iota^*(L^p \otimes F)).
	\end{equation}
	The main result of this section goes as follows.
	\end{sloppypar}
	\begin{thm}\label{thm_ot_as_sp}
		There are $c, C > 0$, $p_1 \in \nat^*$ such that for any $p \geq p_1$, we have
		\begin{equation}\label{eq_ot_as_sp}
			c p^{\frac{n - m}{2}} \leq \big\| \res_p \big\| \leq C p^{\frac{n - m}{2}},
		\end{equation}
		where $\norm{\cdot}$ is the operator norm with respect to the $L^2$-scalar products (\ref{eq_l2_prod}) on $X$ and $Y$.
	\end{thm}
	\begin{rem}
		a) 
		For compact manifolds, a similar statement appeared in Sun \cite[Theorem 3.3]{SunLogBer} for $(F, h^F)$ trivial. 
		However, only the lower bound was discussed there, and there is a gap in the proof. See Sun \cite{SunLogBer} for the explanation of the error and its correction, relying on the current article.
		\par 
		b)
		A more precise version of (\ref{eq_ot_as_sp}), containing the asymptotics of $\| \res_p \|$, will be given in \cite{FinToeplImm}.
	\end{rem}
	This section is organized as follows.
	In Section \ref{sect_tay_exp}, we calculate the first two terms of the Taylor expansion of the $\dbar$-operator in a shrinking tubular neighborhood of $Y$.
	This will play a crucial role in our proof of the lower bound of Theorem \ref{thm_ot_as_sp} in Section \ref{sect_spec_low_bnd}.
	In Section \ref{sect_trace}, we establish the upper bound of Theorem \ref{thm_ot_as_sp}.
	Throughout the section the variables $p \in \nat^*$ and $t \in \real$ are related by 
	\begin{equation}\label{eq_t_p_rel}
		t = \frac{1}{\sqrt{p}}.
	\end{equation}

\subsection{Taylor expansion of the holomorphic differential near submanifold}\label{sect_tay_exp}
	The main goal of this section is to calculate the first two terms of the Taylor expansion of $\dbar^{L^p \otimes F}$-operator, considered in a shrinking neighborhood of $Y$ of size $\frac{1}{\sqrt{p}}$, as $p \to \infty$. 
	Our result is motivated by the Taylor expansions of the associated Dirac operator due to Bismut-Lebeau \cite[Theorem 8.18]{BisLeb91}, which corresponds to trivial $(L, h^L)$ in our setting, and the Taylor expansion due to Dai-Liu-Ma \cite[Theorem 4.6]{DaiLiuMa}, which corresponds to $Y$ equal to a point in our setting. 
	\par 
	More precisely, as in Section \ref{sect_intro}, we consider a triple $(X, Y, g^{TX})$ of bounded geometry.
	By means of the exponential map as in (\ref{eq_kappan}), we identify a neighborhood of the zero section $\mathbb{B}_{r_{\perp}}(N)$ in the normal bundle $N$, to a neighborhood $U := \mathbb{B}_Y^X(r_{\perp})$ of $Y$ in $X$.
	\par 
	Recall that the projection $\pi_0 : U \to Y$ and the identifications of $L, F$ to $\pi_0^* (L|_Y), \pi_0^* (F|_Y)$ in $\mathbb{B}_Y^X(r_{\perp})$ were defined before (\ref{eq_ext0_op}).
	We similarly identify $TX$ to $\pi_0^* (TX|_Y)$ over $\mathbb{B}_Y^X(r_{\perp})$ using the parallel transport with respect to the Levi-Civita connection $\nabla^{TX}$.
	Remark that since $g^{TX}$ is Kähler by (\ref{eq_gtx_def}), the decomposition $TX \otimes_{\real} \comp = T^{(1, 0)} X \oplus T^{(0, 1)} X$ is preserved by $\nabla^{TX}$, cf. \cite[Theorem 1.2.8]{MaHol}. In other words, the identification of $TX$ with $\pi_0^* (TX|_Y)$ induces the identifications 
	\begin{equation}\label{eq_tgt_isom}
		\tau: \pi_0^* (T^{(1, 0)} X|_Y) \to T^{(1, 0)} X|_U, \qquad \tau: \pi_0^* (T^{(0, 1)} X|_Y) \to T^{(0, 1)} X|_U.
	\end{equation}
	\par 
	We define the $1$-forms $\Gamma^{F}$, $\Gamma^{L}$ with values in $\enmr{\pi_0^* (F|_Y)}$, $\enmr{\pi_0^* (L|_Y)}$ as in (\ref{eq_gamma_l_def}) using the above isomorphisms.
	Recall also that the connection $\nabla^N$ on $N$ was introduced before (\ref{eq_sec_fund_f}).
	\begin{comment}
	As both manifolds $(X, g^{TX})$, $(Y, g^{TY})$ are Kähler, and, hence, by \cite[Theorem 1.2.8]{MaHol}, the respective Levi-Civita connections coincide with the connections induced by the Chern connections on the holomorhic tangent bundles, the connection $\nabla^N$ coincides with the Chern connection on $(N, g^N)$ by the usual properties of short exact sequences for Hermitian vector bundles, cf. \cite[Proposition 1.6.6]{KobaVB}.
	\end{comment}
	The connection $\nabla^N$ induces the splitting 
	\begin{equation}
		TN = N \oplus T^H N
	\end{equation}
	of the tangent space of the total space of $N$. Here $T^H N$ is the horizontal part of $N$ with respect to the connection $\nabla^N$. 
	If $U \in TY$, we denote by $U^H \in T^H N$ the horizontal lift of $U$ in $T^H N$.
	\par 
	For $\epsilon > 0$, we denote by $\mathbb{E}(\epsilon)$ (resp. $\mathbb{E}$) the set of smooth sections of $\pi_0^*(L^p|_Y\otimes F|_Y)$ on $\mathbb{B}_{\epsilon}(N)$ (resp. on the total space of $N$).
	We also denote by $\mathbb{E}^{(0, 1)}(\epsilon)$ (resp. $\mathbb{E}^{(0, 1)}$) the set of smooth sections of $\pi_0^*(T^{*(0, 1)}X|_{Y}) \otimes \pi_0^*(L^p|_Y\otimes F|_Y)$ on $\mathbb{B}_{\epsilon}(N)$ (resp. on the total space of $N$).
	\par 
	Clearly, the above isomorphisms allow us to see $\dbar^{L^p \otimes F}$ as an operator
	\begin{equation}\label{eq_dbar_op_interpr}
		\dbar^{L^p \otimes F} : \mathbb{E}(r_{\perp}) \to \mathbb{E}^{(0, 1)}(r_{\perp}).
	\end{equation}
	\par 
	We fix a point $y_0 \in Y$ and an orthonormal frame $(e_1, \ldots, e_{2m})$ (resp. $(e_{2m+1}, \ldots, e_{2n})$) in $(T_{y_0}Y, g^{TY})$ (resp. in $(N_{y_0}, g^{N}_{y_0})$) such that (\ref{eq_cond_jinv}) is satisfied.
	Using the exponential coordinates on $Y$ and the parallel transport of $(e_{2m+1}, \ldots, e_{2n})$ along the geodesics on $Y$, as in Fermi coordinates $\psi_{y_0}$ in (\ref{eq_defn_fermi}), we introduce complex coordinates $z_1, \ldots, z_m$ on $Y$ and linear “vertical" coordinates $z_{m + 1}, \ldots, z_n$ on $N$.
	Using those coordinates, we define the operators
	\begin{equation}\label{eq_defn_dbarhn}
		\dbar_H^{L^p \otimes F}, \mathcal{L}_N^{L^p \otimes F} : \mathbb{E} \to \mathbb{E}^{(0, 1)}, 
	\end{equation}
	by prescribing their action at a point $(y_0, Z_N)$, $Z_N \in \real^{2(n-m)}$, as follows
	\begin{equation}\label{eq_hor_norm_dop}
		\dbar_H^{L^p \otimes F} = \sum_{i = 1}^{m} d \overline{z}_i|_{y_0} \cdot \Big( \frac{\partial}{\partial \overline{z}_i} \big|_{y_0}\Big)^{H},
		\qquad 
		\mathcal{L}_N^{L^p \otimes F} = \sum_{i = m+1}^{n} d\overline{z}_i|_{y_0} \cdot \Big( \frac{\partial}{\partial \overline{z}_i} + \frac{\pi z_i}{2} \Big).
	\end{equation}
	The first differentiation in (\ref{eq_hor_norm_dop}) is well-defined because $\pi_{0*} ( \frac{\partial}{\partial \overline{z}_i} |_{y_0} )^{H} =  \frac{\partial}{\partial \overline{z}_i} |_{y_0}$ is of type $(0, 1)$, and the second derivation is well-defined because the vector bundles are trivialized along fibers of $\pi_0$.
	\par We use notation (\ref{eq_t_p_rel}), and for any $\epsilon > 0$ define the rescaling operator $F_t : \mathbb{E}(\epsilon) \to  \mathbb{E}(\frac{\epsilon}{t})$ for $f \in \mathbb{E}(\epsilon)$ as follows
	\begin{equation}\label{eq_ft_defn}
		(F_t f )(y, Z_N) := f \big( y, t Z_N \big), \qquad (y, Z_N) \in \mathbb{B}_{\frac{\epsilon}{t}}(N).
	\end{equation}
	The operator $F_t : \mathbb{E}^{(0, 1)}(\epsilon) \to  \mathbb{E}^{(0, 1)}(\frac{\epsilon}{t})$ is defined in an analogous way.
	\begin{sloppypar}
	\begin{thm}\label{thm_tay_exp_dbar}
		As $p \to \infty$, we have
		\begin{equation}\label{eq_tay_exp_dbar}
			F_t \circ \dbar^{L^p \otimes F} \circ F_t^{-1}
			=
			\frac{1}{t}
			\mathcal{L}_N^{L^p \otimes F}
			+
			\dbar_H^{L^p \otimes F}
			+
			O\big( 
			t |Z_N|^2 \nabla_N + t |Z_N| \nabla_H + t|Z_N|
			\big),
		\end{equation}
		where $O(t |Z_N|^2 \nabla_N + t |Z_N| \nabla_H + t|Z_N|)$
		is an operator of the form 
		$
			\sum_{i = 1}^{2m} a_i(t, y, Z_N) \cdot d x_i|_{y_0} \cdot ( \frac{\partial}{\partial x_i} |_{y_0})^{H}
			+
			\sum_{j = 2m+1}^{2n} b_j(t, y, Z_N) \cdot d x_j|_{y_0} \cdot \frac{\partial}{\partial x_j}
			+
			c(t, y, Z_N),
		$
		such that there is a constant $C > 0$, for which $|a_i(t, y, Z_N)| \leq C t |Z_N|^2$, $|b_j(t, y, Z_N)| \leq C t |Z_N|$, $|c(t, y, Z_N)| \leq C t|Z_N|$ holds for any $y \in Y$, $|Z_N| < r_{\perp}$, $i = 1, \ldots, m$, and $j = m + 1, \ldots, n$. 
	\end{thm}
	\end{sloppypar}
	\begin{proof}
		To simplify the notations, we denote $\frac{\tau \partial}{\partial \overline{z}_i} := \tau ( \frac{\partial}{\partial \overline{z}_i}|_{y_0})$ and $\frac{\tau \partial}{\partial Z_i} := \tau ( \frac{\partial}{\partial Z_i}|_{y_0})$.
		By (\ref{eq_tgt_isom}), the action of the operator $\dbar^{L^p \otimes F}$, viewed as in (\ref{eq_dbar_op_interpr}), at a point $(y_0, Z_N)$, $|Z_N| < r_{\perp}$, is given by
		\begin{equation}\label{eq_dbar_1}
			\dbar^{L^p \otimes F}
			=
			\sum_{i = 1}^{n} dz_i|_{y_0} \cdot \frac{\tau \partial}{\partial \overline{z}_i}
			 +
			 \sum_{i = 1}^{n} dz_i|_{y_0} \cdot \Big( p \Gamma^{L}_{(y_0, Z_N)}\big( \frac{\tau \partial}{\partial \overline{z}_i} \big) + \Gamma^{F}_{(y_0, Z_N)}\big( \frac{\tau \partial}{\partial \overline{z}_i} \big) \Big).
		\end{equation}
		\par 
		A calculation from Bismut-Lebeau \cite[p. 94-96]{BisLeb91} shows
		\begin{multline}\label{eq_dbar_2}
			F_t \circ \Big( \sum_{i = 1}^{n} dz_i|_{y_0} \cdot \frac{\tau \partial}{\partial \overline{z}_i} \Big) \circ F_t^{-1}
			=
			\frac{1}{t}
			\Big(
			\sum_{i = m+1}^{n} d \overline{z}_i|_{y_0} \cdot \frac{\partial}{\partial \overline{z}_i}
			\Big)
			+
			\dbar_H^{L^p \otimes F}
			\\
			-
			\dim Y \cdot \nu^*(y) \wedge
			+
			 O(t |Z_N|^2  \nabla_N + t |Z_N| \nabla_H),
		\end{multline}
		where the $O$-term is interpreted as in (\ref{eq_tay_exp_dbar}), and $\nu \in \ccal^{\infty}(Y, N)$ is the dual of the \textit{mean curvature} of $\iota$, defined as follows
		\begin{equation}\label{eq_mn_curv_d}
			\nu := \frac{1}{2m} \sum_{i=1}^{2m} A(e_i)e_i,
		\end{equation}
		where the sum runs over an orthonormal basis $e_1, \ldots, e_{2m}$ of $(TY, g^{TY})$ as in (\ref{eq_cond_jinv}).
		However, as the manifolds are Kähler (hence, $J$ commutes with $A$), by (\ref{eq_a_no_tors}), for $i = 1, \ldots, m$, we have
		\begin{equation}
			A(e_{2i - 1})e_{2i - 1} = A(J e_{2i})J e_{2i} = J A(J e_{2i})e_{2i} = J A(e_{2i})J e_{2i} = - A(e_{2i}) e_{2i}.
		\end{equation}
		Hence, we have
		\begin{equation}\label{eq_nu_zero}
			\nu = 0.
		\end{equation}
		\begin{comment}
		which, for compact manifolds, is a consequence of the well-known fact that every complex submanifold of a Kähler manifold is volume minimizing in its homology class, cf. \cite[III.(1.25)]{DemCompl}.
		\end{comment}
		\par 
		Remark that in \cite{BisLeb91}, no assumption of bounded geometry was made, and (\ref{eq_dbar_2}) was stated locally in $Y$.
		However, in the proof of Bismut-Lebeau, the $O$-term comes from the expansion of the Christoffel symbols and the connection forms, and hence it can be bounded uniformly in $Y$ by the results of Section \ref{sect_bnd_geom_cf}.
		\par Now, from Lemma \ref{lem_comp_phi_par} and (\ref{eq_psi_phi_vect_f_com}), we see that for $i = 1, \ldots, 2m$, $j = 2m+1, \ldots, 2n$, at the point $(y_0, Z_N)$, we have
		\begin{equation}\label{eq_fractau_fracpsi}
		\begin{aligned}
			&
			\frac{\tau \partial}{\partial  Z_j}
			=
			\frac{\partial \psi}{\partial  Z_j}
			+
			O(|Z_N|^2),
			\\
			&
			\frac{\tau \partial}{\partial  Z_i}
			=
			\frac{\partial \psi}{\partial  Z_i}
			-
			\sum_{l = 1}^{2m}
 			\frac{\partial \psi}{\partial Z_l}
 			g^{TM}_{y_0}
 			\big(
 				A(
 					e_i
 				)Z_N
 				,
 				e_l
 			\big)
 			+
			O(|Z_N|^2).
		\end{aligned}
		\end{equation}
		From Lemma \ref{lem_gamma_exp} and (\ref{eq_fractau_fracpsi}), we deduce 
		\begin{equation}\label{eq_gamma_calc_tau}
		\begin{aligned}
			&
			\Gamma^{F}_{(y_0, Z_N)}\big( \frac{\tau \partial}{\partial \overline{z}_i} \big) = O(|Z_N|), 
			&&
			\Gamma^{F}_{(y_0, Z_N)}\big( \frac{\tau \partial}{\partial \overline{z}_j} \big) = O(|Z_N|), 
			\\
			&
			\Gamma^{L}_{(y_0, Z_N)}\big( \frac{\tau \partial}{\partial \overline{z}_i} \big) = O(|Z_N|^3), 
			&& 
			\Gamma^{L}_{(y_0, Z_N)}\big( \frac{\tau \partial}{\partial \overline{z}_j} \big) = \frac{1}{2}
			R^L_{y_0} \big(Z_N, \frac{\partial}{\partial \overline{z}_j}\big) + O(|Z_N|^3).
		\end{aligned}
		\end{equation} 
		Now, from (\ref{eq_dbar_1}),  (\ref{eq_dbar_2}), (\ref{eq_nu_zero}) and (\ref{eq_gamma_calc_tau}), we get
		\begin{multline}\label{eq_dbar_3}
			F_t \circ \dbar^{L^p \otimes F} \circ F_t^{-1}
			=
			\frac{1}{t}
				\sum_{i = m+1}^{n} d \overline{z}_i \cdot \Big( \frac{\partial}{\partial \overline{z}_i}
				+
				\frac{1}{2}
				R^L_{y_0} \big(Z_N, \frac{\partial}{\partial \overline{z}_i}\big)
			\Big)
			+
			\dbar_H^{L^p \otimes F}
			\\
			+
			 O(t |Z_N|^2  \nabla_N + t |Z_N| \nabla_H + t|Z_N|).
		\end{multline}
		Finally, by (\ref{eq_omega}), we have $R^L_{y_0} (Z_N, \frac{\partial}{\partial \overline{z}_i} ) = \pi z_i$. 
		By this and (\ref{eq_dbar_3}), we conclude.
	\end{proof}

\subsection{Extension theorem, a proof of the lower bound in Theorem \ref{thm_ot_as_sp}}\label{sect_spec_low_bnd}
	In this section, we establish the lower bound in Theorem \ref{thm_ot_as_sp}.
	For this, in the notations from Section \ref{sect_tay_exp},  we show that the following refined version of Ohsawa-Takegoshi extension theorem holds.
	\begin{thm}\label{thm_ot_weak}
		There are $C > 0$, $p_1 \in \nat^*$, such that for any $p \geq p_1$ and $g \in H^{0}_{(2)}(Y, \iota^*(L^p \otimes F))$, there is $f \in H^{0}_{(2)}(X, L^p \otimes F)$, such that $f|_Y = g$ and 
		\begin{equation}\label{eq_ot_weak}
			\norm{f}_{L^2(X)} \leq \frac{C}{p^{\frac{n - m}{2}}} \norm{g}_{L^2(Y)}.
		\end{equation}
	\end{thm}
	\begin{proof}[Proof of the lower bound in Theorem \ref{thm_ot_as_sp}.]
		Let $C$ be as in (\ref{eq_ot_weak}). As $\ext_p$ is defined as the minimal extension with respect to the $L^2$-norm, for any $g \in H^{0}_{(2)}(Y, \iota^*(L^p \otimes F))$, by Theorem \ref{thm_ot_weak}, we have
		\begin{equation}
			\big\| \ext_p g \big\|_{L^2(X)} \leq \frac{C}{p^{\frac{n - m}{2}}} \norm{g}_{L^2(Y)}.
		\end{equation}
		This means exactly that the lower bound in Theorem \ref{thm_ot_as_sp} holds for $c := C^{-1}$.
	\end{proof}
	\par 
	By the above, to settle the main goal of this section, we need to establish Theorem  \ref{thm_ot_weak}.
	The main idea of the proof of Theorem \ref{thm_ot_weak} is to pass through the general framework of the proof of Ohsawa-Takegoshi extension theorem, cf. \cite[\S 11]{DemBookAnMet}. We choose a smooth extension of $g$ over $X$, and then obtain the holomorphic extension by modifying the smooth one using a solution of $\dbar$-equation with singular weight, which forces the solution to annihilate along $Y$.
	\par The novelty here is that instead of choosing an arbitrary smooth extension, as it is done in \cite[\S 11]{DemBookAnMet}, we choose a specific one, given by the operator (\ref{eq_ext0_op}). 
	The incentive for doing so comes from the fact that we would like to get some estimates, related to this extension, which would be uniform in $p$, and this doesn't seem doable for an arbitrary choice of the extension.
	Choosing a specific extension allows us to use Theorem \ref{thm_tay_exp_dbar} to prove the uniform versions of some $L^2$-estimates, which are indispensable in the proof.
	For those estimates and after, we need the following proposition.
	\begin{prop}\label{prop_der_bound}
		For any $k \in \nat$, there are $C > 0$, $p_1 \in \nat^*$, such that for any $p \geq p_1$ and $f \in H^{0}_{(2)}(X, L^p \otimes F)$, we have
		\begin{equation}\label{eq_der_bound}
			\big\| \nabla^k f \big\|_{L^2(X)} \leq C p^{\frac{k}{2}} \big\| f \big\|_{L^2(X)},
		\end{equation}
		where $\nabla$ is the covariant derivative with respect to the induced Chern and Levi-Civita connections.
	\end{prop}
	\par For the proof of Proposition \ref{prop_der_bound} and in many places later, the following result will be crucial.
	\par 
	\begin{thm}[{Ma-Marinescu \cite[Theorem 1]{MaMarOffDiag}}]\label{thm_bk_off_diag}
		For any $k \in \nat$, there are $c, C > 0$, $p_1 \in \nat^*$ such that for $p \geq p_1$, we have
		\begin{equation}\label{eq_bk_off_diag}
			\Big|  B_p^X(x_1, x_2) \Big|_{\ccal^k(X \times X)} \leq C p^{n + \frac{k}{2}} \cdot \exp \big(- c \sqrt{p} \cdot \dist(x_1, x_2) \big),
		\end{equation}
		where $\ccal^k$-norm here is interpreted as in Theorem \ref{thm_ext_exp_dc}.
	\end{thm}
	\begin{proof}[Proof of Proposition \ref{prop_der_bound}]
	For compact manifolds, the statement is very classical.
	For manifolds of bounded geometry, the proof can be obtained by a slight modification of the argument in \cite[Lemma 2]{MaMarOffDiag}.
	We present here an alternative proof, based on Theorem \ref{thm_bk_off_diag}.
		First of all, by Proposition \ref{prop_exp_bound_int} and Theorem \ref{thm_bk_off_diag}, there are $C > 0$, $p_1 \in \nat^*$ such that for any $x_0 \in X$, for $p \geq p_1$, we have
		\begin{equation}\label{eq_bound_int_berg}
			\int_{X} \big| (\nabla^k B_p^X)(x_0, x)  \big|  dv_X(x) \leq C p^{\frac{k}{2}}, 
			\qquad 
			\int_{X} \big| (\nabla^k B_p^X)(x, x_0)  \big|  dv_X(x) \leq C p^{\frac{k}{2}}.
		\end{equation}
		Also, since $f = B_p^X f$, $\nabla^k f$ admits the integral representation
		\begin{equation}\label{eq_res_int_repr}
			\nabla^k f(x_1) = \int_{X} (\nabla^k B_p^X)(x_1, x_2) \cdot f(x_2)  dv_X(x_2).
		\end{equation}
		We deduce (\ref{eq_der_bound}) directly from (\ref{eq_bound_int_berg}), (\ref{eq_res_int_repr}) and Young's inequality for integral operators, cf. \cite[Theorem 0.3.1]{SoggBook} applied for $p, q = 2$, $r = 1$ in the notations of \cite{SoggBook}.
	\end{proof}
	
	\begin{proof}[Proof of Theorem \ref{thm_ot_weak}]
		Recall that the operator $\ext_p^{0}$ was defined in (\ref{eq_ext0_op}). 
		We would like to verify that for any $g \in H^{0}_{(2)}(Y, \iota^*(L^p \otimes F))$, $y_0 \in Y$, we have
		\begin{equation}\label{eq_dbar_y_zero}
			\alpha(y_0) = 0, \quad \text{where} \quad \alpha := \dbar^{L^p \otimes F} (\ext_p^{0} g).
		\end{equation}
		\par 
		Indeed, let us work in a neighborhood $V := \mathbb{B}_Y^X(\frac{r_{\perp}}{4})$ of $Y$ in $X$. 
		Recall that $t \in \real_+$ and $F_t$ were defined in (\ref{eq_ft_defn}).
		Then in the notations of (\ref{eq_ext0_op}), on $V$, we have
		\begin{equation}\label{eq_gtild_d}
			\ext_p^{0} g = F_t^{-1} \tilde{g}, \qquad \tilde{g}(y, Z_N) = g(y)  \exp \Big(- \frac{\pi}{2} |Z_N|^2 \Big).
		\end{equation}
		Recall that $\dbar_H^{L^p \otimes F}, \mathcal{L}_N^{L^p \otimes F}$ were defined in (\ref{eq_hor_norm_dop}).
		A trivial calculation shows that on $V$, we have
		\begin{equation}\label{eq_gtild_n}
			\mathcal{L}_N^{L^p \otimes F} \tilde{g} = 0.
		\end{equation}
		Also, since $\nabla^N$ preserves $g^N$, in the notations of (\ref{eq_hor_norm_dop}), similarly to \cite[(8.97)]{BisLeb91}, we have
		\begin{equation}\label{eq_gtild_h}
			\Big( \Big( \frac{\partial}{\partial \overline{z}_i} \big|_{y_0}\Big)^{H} \tilde{g}\Big) (y_0, Z_N)
			=
			\Big( \frac{\partial}{\partial \overline{z}_i} g \Big)(y_0)  \exp \Big(- \frac{\pi}{2} |Z_N|^2 \Big).
		\end{equation}
		As a consequence of (\ref{eq_gtild_h}) and the fact that $g$ is holomorphic, we obtain 
		\begin{equation}\label{eq_gtild_h2}
			\dbar_H^{L^p \otimes F}  \tilde{g} = 0.
		\end{equation}
		From (\ref{eq_gtild_d}), (\ref{eq_gtild_n}), (\ref{eq_gtild_h2}) and the fact that all the residue terms in Theorem \ref{thm_tay_exp_dbar} contain $|Z_N|$, we deduce (\ref{eq_dbar_y_zero}).
		\par 
		Now, using the $L^2$-estimates, let us construct a holomorphic perturbation of $\ext_p^{0} g$, satisfying the assumptions of Theorem \ref{thm_ot_weak}.
		Recall that $\delta_Y : X \setminus Y \to \real$, $\alpha_Y : X \to \real$,  were defined in (\ref{eq_delta_defn_y}).
		For $\epsilon > 0$, let us now define the weight $\delta_p : X \setminus Y \to \real$ as follows
		\begin{equation}\label{eq_delta_p_defn_wght}
			\delta_p := 2(n - m)\delta_Y - \epsilon p \alpha_Y.
		\end{equation}
		By taking $\epsilon$ small, by Lemma \ref{thm_plurisub}, we see that there exists $p_1 \in \nat^*$, such that for any $p \geq p_1$, over $X$, the following inequality holds in the distributional sense
		\begin{equation}\label{eq_pomg_deltap}
			p \omega + \frac{\imun}{2 \pi}  \partial \dbar \delta_p >  \frac{p}{2}\omega.
		\end{equation} 
		Let us fix $\epsilon$ small enough, such that for any $|Z_N| < r_{\perp}$, we have
		\begin{equation}\label{eq_zn_diff_phi}
			\frac{\pi}{2}  |Z_N|^2 - \epsilon \alpha_Y(y, Z_N) \geq \frac{\pi}{4} |Z_N|^{2}.
		\end{equation}
		\par 
		We will now prove that there are $C_1 > 0$, $p_1 \in \nat^*$, such that for any $p \geq p_1$, $g \in H^{0}_{(2)}(Y, \iota^*(L^p \otimes F))$ and $\alpha$ defined in (\ref{eq_dbar_y_zero}), we have
		\begin{equation}\label{eq_int_alph}
			\int_{X} | \alpha |^2 e^{-\delta_p} dv_X \leq C_1 \| g \|_{L^2(Y)}^{2}.
		\end{equation}
		\par 
		Since $\alpha$ has support over $\mathbb{B}_Y^X(r_{\perp})$, we have
		\begin{equation}
			\int_{X \setminus \mathbb{B}_Y^X(\frac{r_{\perp}}{4})} | \alpha |^2 e^{-\delta_p} dv_X 
			=
			\int_{\mathbb{B}_Y^X(r_{\perp}) \setminus \mathbb{B}_Y^X(\frac{r_{\perp}}{4})} | \alpha |^2 e^{-\delta_p} dv_X 
		\end{equation}
		From this, by (\ref{eq_ext0_op}) and (\ref{eq_zn_diff_phi}), there are $c, C_2 > 0$, $p_1 \in \nat^*$, such that for any $p \geq p_1$, we have
		\begin{equation}\label{eq_int_alph222}
			\int_{X \setminus \mathbb{B}_Y^X(\frac{r_{\perp}}{4})} | \alpha |^2 e^{-\delta_p} dv_X \leq C_2 \exp(-cp) \Big( \| g \|^2_{L^2(Y)} + \| \nabla g \|^2_{L^2(Y)}	 \Big).
		\end{equation}
		Now, as $\alpha$ has support in $\mathbb{B}_Y^X(\frac{r_{\perp}}{2})$, it is enough to work in $(y, Z_N)$, $y \in Y$, $Z_N \in N_y$ coordinates.
		To estimate the integral over $\mathbb{B}_Y^X(\frac{r_{\perp}}{4})$, we use (\ref{eq_gtild_d}) and make the change of variables by $F_t$ to get
		\begin{multline}\label{eq_int_rescalling}
			\int_{\mathbb{B}_Y^X(\frac{r_{\perp}}{4})} | \alpha |^2 e^{-\delta_p} dv_Y \wedge dv_N
			\\
			=
			\int_{\mathbb{B}_{\frac{r_{\perp}}{4t}}(N)} \Big| \big(F_t \circ \dbar^{L^p \otimes F} \circ F_t^{-1} \big) \tilde{g} \Big|^2(y, tZ_N) 
			 \frac{e^{\epsilon p \alpha_Y(y, Z_N)}}{|Z_N|^{2(n - m)}} dv_Y \wedge dv_N.
		\end{multline}
		By  (\ref{eq_gtild_n}), (\ref{eq_gtild_h}), (\ref{eq_gtild_h2}), (\ref{eq_zn_diff_phi}), (\ref{eq_int_rescalling}), and the fact that for $j = 1, 2$,
		\begin{equation}
			\int_{\real^{2(n-m)}} \frac{|Z_N|^j \exp (- \frac{\pi}{4} |Z_N|^2) dv_{\real^{2(n - m)}}(Z_N)}{|Z_N|^{2(n - m)}} < + \infty,
		\end{equation}		 
		we conclude that there are $C_3 > 0$, $p_1 \in \nat$, such that for any $p \geq p_1$, $g \in H^{0}_{(2)}(Y, \iota^*(L^p \otimes F))$ and $\alpha$ as in (\ref{eq_dbar_y_zero}), we have
		\begin{equation}\label{eq_int_alph2}
			\int_{\mathbb{B}_Y^X(\frac{r_{\perp}}{4})} | \alpha |^2 e^{-\delta_p} dv_X \leq C_3 \cdot t^2 \cdot \Big( \| g \|^2_{L^2(Y)} + \| \nabla g \|^2_{L^2(Y)}	 \Big).
		\end{equation}
		From Proposition \ref{prop_der_bound} and (\ref{eq_int_alph2}), we deduce (\ref{eq_int_alph}).
		\par 
		From (\ref{eq_pomg_deltap}), the fact that we may resolve the $\dbar$-equation, see \cite[Théorème 5.1]{Dem82}, the trivial fact that $\dbar^{L^p \otimes F} \alpha = 0$ and (\ref{eq_int_alph}), we see that there are $C > 0$, $p_1 \in \nat^*$, such that for any $p \geq p_1$, $g \in H^{0}_{(2)}(Y, \iota^*(L^p \otimes F))$, there is $f_0 \in L^2(X, L^p \otimes F)$, such that
		\begin{equation}\label{eq_dbar_sol}
			\dbar^{L^p \otimes F} f_0 = \alpha, \qquad \int_{X \setminus Y} | f_0 |^2 e^{-\delta_p} dv_X \leq  \frac{C}{p} \| g \|_{L^2(Y)}^{2}.
		\end{equation}
		Let us prove that $f := \ext_p^{0} g - f_0$ verifies the assumptions of Theorem \ref{thm_ot_weak}.
		\par 
		From the $L^2$-condition and the standard regularity results, \cite[Lemme 6.9]{Dem82}, $f$ is a smooth holomorphic section.
		In particular, $f_0$ is smooth as well.
		However, since $\exp(-2(n - m)\delta_Y)$ is not integrable, the $L^2$-bound (\ref{eq_dbar_sol}) implies that $f_0|_Y = 0$.
		Hence, we conclude that $f|_Y = g$.
		It is only left to verify that $f$ satisfies the needed $L^2$-bound (\ref{eq_ot_weak}).
		\par 
		From the Gauss integral calculation, and the fact that our triple is of bounded geometry, there is $C_4 > 1$, such that we have
		\begin{equation}\label{eq_exp0_bnd_norm}
			\frac{C_4^{-1}}{p^{\frac{n - m}{2}}} \norm{g}_{L^2(Y)} \leq \big \| \ext_p^{0} g \big \| _{L^2(X)} \leq \frac{C_4}{p^{\frac{n - m}{2}}} \norm{g}_{L^2(Y)}.
		\end{equation}
		Let us now prove that there is $C_5 > 0$ such that for any $p \geq p_1$, $g \in H^{0}_{(2)}(Y, \iota^*(L^p \otimes F))$, the following bound holds
		\begin{equation}\label{eq_l2_est_fin}
			\int_{X} | f_0 |^2 e^{-\delta_p} dv_X  \geq C_5 p^{n-m} \int_{X} | f_0 |^2 dv_X.
		\end{equation}
		This will be clearly enough for our needs, as from the $L^2$-bound in (\ref{eq_dbar_sol}), (\ref{eq_exp0_bnd_norm}) and (\ref{eq_l2_est_fin}), we would deduce the $L^2$-bound (\ref{eq_ot_weak}).
		\par 
		First of all, since $\alpha_Y \geq \min\{  \frac{1}{2} (\frac{r_{\perp}}{4})^2, \frac{1}{2} \}$ on $X \setminus \mathbb{B}_Y^X(\frac{r_{\perp}}{4})$, there are $c, C_6 > 0$, such that 
		\begin{equation}\label{eq_l2_est_fin0}
			\int_{X \setminus \mathbb{B}_Y^X(\frac{r_{\perp}}{4})} | f_0 |^2 e^{-\delta_p} dv_X 
			\geq
			C_6 \exp(\epsilon c p) \int_{X \setminus \mathbb{B}_Y^X(\frac{r_{\perp}}{4})} | f_0 |^2 dv_X 
		\end{equation}
		\par 
		It is now only left to give the lower bound for the integrand on the left-hand side of (\ref{eq_l2_est_fin}), where the integration is done over $\mathbb{B}_Y^X(\frac{r_{\perp}}{4})$.
		But remark that from (\ref{eq_delta_defn_y}) and (\ref{eq_delta_p_defn_wght}), over $\mathbb{B}_Y^X(\frac{r_{\perp}}{4})$, there is $C_7 > 0$, such that for any $p \in \nat^*$, we have $e^{-\delta_p} \geq C_7 p^{n - m}$.
		From this, we deduce 
		\begin{equation}\label{eq_l2_est_fin132}
			\int_{\mathbb{B}_Y^X(\frac{r_{\perp}}{4})} | f_0 |^2 e^{-\delta_p} dv_X  \geq C_7 p^{n-m} \int_{\mathbb{B}_Y^X(\frac{r_{\perp}}{4})} | f_0 |^2 dv_X.
		\end{equation}
		From (\ref{eq_l2_est_fin0}) and (\ref{eq_l2_est_fin132}), we obtain (\ref{eq_l2_est_fin}).
	\end{proof}
	\begin{rem}
		Our proof shows that there is $\ext_p^{1} :  H^{0}_{(2)}(Y, \iota^*(L^p \otimes F)) \to H^{0}_{(2)}(X, L^p \otimes F)$, verifying  $(\ext_p^{1} g) |_Y = g$ for $g \in  H^{0}_{(2)}(Y, \iota^*(L^p \otimes F))$, and such that (\ref{eq_ext_as}) holds for $\ext_p^{1}$ instead of $\ext_p^{0}$.
	\end{rem}
	
\subsection{Asymptotic trace theorem, a proof of the upper bound in Theorem \ref{thm_ot_as_sp}}\label{sect_trace}
	The main goal of this section is to show that the restriction to $Y$ of an element from $H^{0, \perp}_{(2)}(X, L^p \otimes F)$ lies in $H^{0}_{(2)}(Y, \iota^*(L^p \otimes F))$ and to give a proof of the upper bound in Theorem \ref{thm_ot_as_sp}. 
	The main idea of our proof is to study the Schwartz kernel of the restriction operator and to use the exponential bound on the Bergman kernel.
	\par 
	Let us explain more precisely how to relate the Schwartz kernel $\res_p(y, x)$, $y \in Y$, $x \in X$, of the restriction operator, evaluated with respect to the volume form $dv_X$, with the Bergman kernel.
	In fact, directly from the identity $\res_Y \circ B_p^X = \res_p$, we obtain that 
	\begin{equation}\label{eq_relat_res_oper_bergm_kern}
		\res_p(y, x) =  B_p^X(y, x).
	\end{equation}
	\begin{prop}\label{prop_restr_is_l2}
		There is $p_1 \in \nat^*$, such that the restriction to $Y$ of any element from $H^{0, \perp}_{(2)}(X, L^p \otimes F)$ lies in $H^{0}_{(2)}(Y, \iota^*(L^p \otimes F))$.
	\end{prop}
	\begin{proof}
		By Proposition \ref{prop_exp_bound_int} and Theorem \ref{thm_bk_off_diag}, there are $C > 0$, $p_1 \in \nat^*$ such that for any $x_0 \in X$, for $p \geq p_1$, we have
		\begin{equation}\label{eq_bound_int_berg21}
			\int_{Y} \big| B_p^X(y, x_0)  \big|  dv_Y(y) \leq C p^{n - m}, 
		\end{equation}
		We deduce directly from (\ref{eq_bound_int_berg}), (\ref{eq_relat_res_oper_bergm_kern}), (\ref{eq_bound_int_berg21}) and Young's inequality for integral operators, cf. \cite[Theorem 0.3.1]{SoggBook} applied for $p, q = 2$, $r = 1$ in the notations of \cite{SoggBook}, that Proposition \ref{prop_restr_is_l2} holds, and, moreover, there are $C > 0$ and $p_1 \in \nat^*$, such that for $p \geq p_1$, $\| \res_p \| \leq C p^{n-m}$.
	\end{proof}
	\begin{proof}[Proof of the upper bound in Theorem \ref{thm_ot_as_sp}.]
		The proof is essentially identical to the proof of Proposition \ref{prop_restr_is_l2} with only one change: instead of looking at $\res_p$, we consider the operator $\res_p \circ \res_p^*$.
		 In fact, directly from (\ref{eq_relat_res_oper_bergm_kern}) and the fact that $B_p^X \circ B_p^X = B_p^X$, we conclude that the Schwartz kernel $(\res_p \circ \res_p^*)(y, y')$, $y, y' \in Y$, of $\res_p \circ \res_p^*$, evaluated with respect to the volume form $dv_Y$, satisfies
		 \begin{equation}
		 	(\res_p \circ \res_p^*)(y, y') = B_p^X(y, y').
		 \end{equation}
		 The proof now proceeds in the same way as in Proposition \ref{prop_restr_is_l2} with an additional remark $\| \res_p \circ \res_p^* \| = \| \res_p \|^2$.
	\end{proof}
	
\section{Asymptotic expansions of the two kernels}\label{sect_lgbk_prfs}
	The main goal of this section is to study the asymptotics of the Schwartz kernels of the orthogonal Bergman projector and the extension operator.
	This section is organized as follows. 
	In Section \ref{sect_exp_dec_log}, we establish Theorem \ref{thm_logbk_exp_dc}.
	We also show that despite the global nature of the orthogonal Bergman kernel, the asymptotic expansion of it depends only locally on the geometry of the problem.
	In Section \ref{sect_ct_asymp}, we prove that after certain reparametrization given by a homothety in Fermi coordinates, Bergman kernel admits a complete asymptotic expansion. 	
	In Section \ref{sect_as_exp_logbk}, we establish Theorem \ref{thm_berg_perp_off_diag}.
	Finally, in Section \ref{sect_exp_dec_ot}, we prove all the other results announced in Section \ref{sect_intro}.
	We use the notation (\ref{eq_t_p_rel}) throughout the section.
	
	\subsection{Exponential decay and localization of the orthogonal Bergman kernel}\label{sect_exp_dec_log}
		The first main goal of this section is to prove that the orthogonal Bergman kernel has exponential off-diagonal decay, i.e. to establish Theorem \ref{thm_logbk_exp_dc}.
		The second main goal is to establish that the asymptotic expansion of the orthogonal Bergman kernel is localized.
		\par 
		The main difficulty here is that the projection $B_p^{\perp}$ is not the spectral projection associated to the Laplacian. 
		Hence, the methods, developed by Dai-Liu-Ma in \cite{DaiLiuMa}, cf. \cite[Proposition 4.1.5]{MaHol}, based on the finite propagation speed property for the wave equation, cannot be applied.
		\par 
		One can remedy this by extending a slightly more technical approach through the $L^2$-estimates with varying weights, similar to what has been done by Lindholm \cite[Proposition 9]{Lindholm}.
		We follow here an alternative approach.
		The main idea of our proof of Theorem \ref{thm_logbk_exp_dc} is to construct the “approximate projection" $B_p^{\perp, a}$ onto the $H^{0, \perp}_{(2)}(X, L^p \otimes F)$ (the precise meaning of this is given after (\ref{eq_bpa_sp_gap})), such that the Schwartz kernel of it verifies the estimate analogous to (\ref{eq_logbk_exp_dc}). 
		Then by means of the spectral theory, essentially relying on Theorem \ref{thm_ot_as_sp}, we relate $B_p^{\perp, a}$ and $B_p^{\perp}$, and show that the exponential estimate of the Schwartz kernel persist through this relation. An exponential estimate of the Bergman kernel, see Theorem \ref{thm_bk_off_diag}, plays a crucial role in our approach.
		The advantage of this approach over the one of Lindholm is that it introduces some of the techniques, which will later play a crucial role in our proof of Theorem \ref{thm_berg_perp_off_diag}, where $L^2$-estimates, it seems, do not apply.
		\par 
		We use the notations and assumptions from Theorem \ref{thm_logbk_exp_dc}.
		The next proposition is crucial for our further estimates.
		\begin{prop}[{Ma-Marinescu \cite[Lemma 2]{MaMarOffDiag}}]\label{prop_der_bound2}
			For any $k \in \nat$, there is $C > 0$, such that for any $p \in \nat^*$, $f \in H^{0}_{(2)}(X, L^p \otimes F)$ and $x \in X$, we have
			\begin{equation}
				\big| \nabla^k f (x) \big| \leq C p^{\frac{n + k}{2}} \| f \|_{L^2(X)}.
			\end{equation}
		\end{prop}
		\par 
		Let us first explain the construction of the approximate projection.
		Define $A_p := \ext_p^{0} \circ \res_Y \circ B_p^{X}$, where $\ext_p^{0}$ is as in (\ref{eq_ext0_op}).
		Then define $B_p^{\perp, a}$ as follows
		\begin{equation}\label{eq_bpperp_a_ap_defn}
			B_p^{\perp, a} : L^2(X, L^p \otimes F) \to L^2(X, L^p \otimes F), \qquad B_p^{\perp, a} := A_p^{*} \circ A_p.
		\end{equation}
		\par 
		Clearly, for any $f \in L^2(X, L^p \otimes F)$, we have $\scal{B_p^{\perp, a} f}{f}_{L^2} = \scal{A_p f}{A_p f}_{L^2}$. Hence in the notations of (\ref{eq_h00_defn}), we have
		\begin{equation}\label{eq_ker_bpa}
			\ker B_p^{\perp, a} = H^{0, 0}_{(2)}(X, L^p \otimes F) \oplus \ker B_p^{X}.
		\end{equation}
		By Theorem \ref{thm_ot_as_sp} and (\ref{eq_exp0_bnd_norm}), the operators $B_p^{\perp, a}$, $p \in \nat$, posses a \textit{uniform spectral gap}, i.e. there are $a, b > 0$, $p_1 \in \nat^*$, such that for any $p \geq p_1$, we have
		\begin{equation}\label{eq_bpa_sp_gap}
			{\rm{Spec}}(B_p^{\perp, a}) \subset \{0\} \cup [a, b].
		\end{equation}
		\par 
		The properties (\ref{eq_ker_bpa}) and (\ref{eq_bpa_sp_gap}) justify the name “approximate projection", as $B_p^{\perp}$ is the only self-adjoint operator on $H^0_{(2)}(X, L^p \otimes F)$, satisfying (\ref{eq_ker_bpa}) and (\ref{eq_bpa_sp_gap}) for $a, b = 1$.
		The property (\ref{eq_bpa_sp_gap}) is what we call by “spectral relation" between $B_p^{\perp, a}$ and $B_p^{\perp}$.
		\begin{proof}[Proof of Theorem \ref{thm_logbk_exp_dc}]
			Clearly, by Lemma \ref{lem_bnd_prod_a}, Theorem \ref{thm_bk_off_diag} and (\ref{eq_ext0_op}), we see that there are $c > 0$, $p_1 \in \nat^*$, such that for any $k \in \nat$, there is $C_1 > 0$, such that for any $p \geq p_1$ and $x_1, x_2 \in X$,
			\begin{equation}\label{eq_bpa_bnd_2}
				\big| B_p^{\perp, a}(x_1, x_2) \big|_{\ccal^k(X \times X)} \leq C_1 p^{n + \frac{k}{2}} \exp \Big(- c \sqrt{p} \cdot \big(  \dist(x_1, x_2) + \dist(x_1, Y) + \dist(x_2, Y) \big) \Big).
			\end{equation}
			\par 
			We will now show that the estimate (\ref{eq_logbk_exp_dc}) follows formally from (\ref{eq_ker_bpa}), (\ref{eq_bpa_sp_gap}) and (\ref{eq_bpa_bnd_2}).
			To do so, from (\ref{eq_ker_bpa}) and (\ref{eq_bpa_sp_gap}), we see that for $\epsilon := 1 - \frac{a}{2b}$, and for any $r \in \nat^*$, the following bound holds
			\begin{equation}\label{eq_spec_est_bpa}
				\Big\| 
					\Big(
						B_p^X - \frac{B_p^{\perp, a}}{2b} 
					\Big)^{r}
					-
					B_p^X
					+
					B_p^{\perp}
				\Big\|
				\leq
				\epsilon^r.
			\end{equation}
			\begin{sloppypar}
			Since the operator under the norm of (\ref{eq_spec_est_bpa}) vanishes on the orthogonal complement of $H^0_{(2)}(X, L^p \otimes F)$ and takes values inside of $H^0_{(2)}(X, L^p \otimes F)$, Proposition \ref{prop_der_bound2} and (\ref{eq_spec_est_bpa}) then imply that for any $k \in \nat$, there are $C_2 > 0$, $l \in \nat$, such that for any $p \geq p_1$, $r \in \nat^*$, we have
			\begin{equation}\label{eq_spec_est_bpa_2}
				\Big|
				\Big(
				\Big(
						B_p^X - \frac{B_p^{\perp, a}}{2b} 
					\Big)^{r}
					-
					B_p^X
					+
					B_p^{\perp}
				\Big)
				(x_1, x_2)
				\Big|_{\ccal^k(X \times X)}
				\leq
				C_2 p^{n + k/2} \epsilon^r.
			\end{equation}
			\end{sloppypar}
			\par 
			Now, using the trivial identity $(B_p^X)^{r} = B_p^X$, we expand $(B_p^X - \frac{B_p^{\perp, a}}{2b})^{r} - B_p^X$ into $2^r - 1$ summands, such that each summand contains $B_p^{\perp, a}$ as a multiple. By Theorem \ref{thm_bk_off_diag} and (\ref{eq_bpa_bnd_2}), we may apply Lemma \ref{lem_bnd_prod_a} for $W := Y$, to bound each of those summands. 
			Hence, there are $C_3 > 0$, $p_1 \in \nat^*$, such that for any $x_1, x_2 \in X$, $p \geq p_1$, $r \in \nat$, we have
			\begin{multline}\label{eq_spec_est_bpa_3}
				\Big|
				\Big(
				\Big(
						B_p^X - \frac{B_p^{\perp, a}}{2b} 
					\Big)^{r}
					-
					B_p^X
				\Big)
				(x_1, x_2)
				\Big|_{\ccal^k(X \times X)}
				\\
				\leq 
				C_3^r p^{n + k/2} \cdot \exp \Big(- \frac{c}{8} \sqrt{p} \cdot \big(  \dist(x_1, x_2) + \dist(x_1, Y) + \dist(x_2, Y) \big) \Big).
			\end{multline}
			Now, we fix $x_1, x_2 \in X$, and let
			\begin{equation}\label{eq_r_choice}
				r := 
				\bigg \lceil 
					\frac{c}{16 \log (\max(C_3, 2))} \sqrt{p} \cdot \Big( \dist(x_1, x_2) + \dist(x_1, Y) + \dist(x_2, Y) \Big)
				\bigg \rceil,
			\end{equation}
			where $\lceil \cdot \rceil$ is the ceil function.
			Then in the right-hand side of (\ref{eq_spec_est_bpa_3}), for this choice of $r$, the power of $C_3$ becomes negligible with respect to the last multiplicand.
			Hence, (\ref{eq_logbk_exp_dc}) holds by (\ref{eq_spec_est_bpa_2}), (\ref{eq_spec_est_bpa_3}), (\ref{eq_r_choice}) and the inequality $\epsilon < 1$.
		\end{proof}
		\par 
		Now, in the second part of this section, we show that despite the global nature of the orthogonal Bergman kernel, the asymptotic expansion of it depends only on the local geometry of the problem.
		\par
		\begin{sloppypar}		 
		More precisely, we fix $X, Y, (L, h^L), (F, h^F), dv_X, dv_Y$ as in Section \ref{sect_intro}.
		We denote by $X', Y', (L', h^{L'}), (F', h^{F'}), dv_{X'}, dv_{Y'}$ some other manifold, submanifold, etc.
		We define $g^{TX}$, $g^{TX'}$ as in (\ref{eq_gtx_def}).
		We fix $y_0 \in Y$, $y'_0 \in Y'$ and assume that there is a biholomorphism between $U := \mathbb{B}_{y_0}^{X}(r_0)$ and $U' := \mathbb{B}_{y'_0}^{X'}(r_0)$ for $r_0 < r_X, r_Y$, such that it induces a biholomorphism of $Y \cap U$ to $Y' \cap U'$ and sends $dv_X, dv_Y$ to $dv_{X'}, dv_{Y'}$ respectively.
		We also assume that the biholomorphism can be extended to holomorphic isometries between $(L, h^L), (F, h^F)$ and $(L', h^{L'}), (F', h^{F'})$.
		In particular, it is a local isometry between $(X, g^{TX}, y_0)$ and $(X', g^{TX'}, y'_0)$.
		\par 
		We denote by $B_p^X{}'$, $B'_p{}^{\perp}$ the Bergman projector and the orthogonal Bergman projector associated with $X', (L', h^{L'}), (F', h^{F'}), dv_{X'}$.
		For $x_1, x_2 \in X$, we denote by $B_p^X{}'(x_1, x_2)$, $B'_p{}^{\perp}(x_1, x_2)$ the Schwartz kernels of those operators, evaluated with respect to $dv_{X'}$.
		\begin{thm}\label{thm_local_logbk}
			There is $p_1 \in \nat^*$, such that for any $r_0 > 0$ as above, there is $c > 0$, such that for any $k \in \nat$, there is $C > 0$, such that for $V := \mathbb{B}_{y_0}^{X}(r_0/2)$, $x, x' \in V$, and  any $p \geq p_1$, we have
			\begin{equation}\label{eq_local_logbk}
				\big|  (B_p^{\perp} - B'_p{}^{\perp})(x, x')  \big|_{\ccal^k(V \times V)} \leq C \exp (- c \sqrt{p} ).
			\end{equation}
		\end{thm}
		\end{sloppypar}
		\par 
		The proof of Theorem \ref{thm_local_logbk} is based on the following proposition.
		\begin{prop}\label{prop_bk_local}
			There is $p_1 \in \nat^*$, such that for any $r_0 > 0$ as above, there is $c > 0$, such that for any $k \in \nat$, there is $C > 0$, such that for $V := \mathbb{B}_{y_0}^{X}(r_0/2)$, $x, x' \in V$, and any $p \geq p_1$, we have
			\begin{equation}\label{eq_bk_local}
				\big|  (B_p^X - B_p^X{}')(x, x') \big|_{\ccal^k(V \times V)} \leq C \exp (- c \sqrt{p} ).
			\end{equation}
		\end{prop}
		\begin{sloppypar}
		\begin{proof}
			Our proof here is an easy modification of the proof of \cite[Theorem 1]{MaMarOffDiag}, so we only briefly explain the main steps.
			First, we denote by $\laplcomp_p$ (resp. $\laplcomp'_p$) the Kodaira Laplacians on $\ccal^{\infty}(X, L^p \otimes F)$ (resp. $\ccal^{\infty}(X', L'{}^p \otimes F')$).
			We would like to prove that there are constants $a, c_1, c_2 > 0$ such that for any $k \in \nat$, $u_0 > 0$, there is $C > 0$, such that for any $u \geq u_0$, $p \in \nat^*$, $x_1, x_2 \in V := \mathbb{B}_{y_0}^{X}(r_0/2)$, the following estimates hold
			\begin{multline}\label{eq_hk_est_12}
				\Big| \Big( 
				\exp \Big( - \frac{u}{p} \laplcomp_p \Big) - 
				\exp \Big( - \frac{u}{p} \laplcomp'_p \Big)
				\Big)
				 (x_1, x_2) \Big|_{\ccal^k(V \times V)}
				\\
				 \leq C p^{n + k/2} \cdot  \exp \Big(c_1 u - \frac{ap}{u} \cdot \big( \dist(x_1, x_2)  + 1 \big) \Big).
			\end{multline}
			\vspace*{-0.8cm}
			\begin{multline}\label{eq_hk_est_1}
				\Big| \Big( 
				\laplcomp_p \exp \Big( - \frac{u}{p} \laplcomp_p \Big) - 
				\laplcomp'_p \exp \Big( - \frac{u}{p} \laplcomp'_p \Big)
				\Big)
				 (x_1, x_2) \Big|_{\ccal^k(V \times V)}
				\\
				  \leq C p^{n + 1 + k/2} \cdot  \exp \Big(- c_2 u - \frac{ap}{u} \cdot \big( \dist(x_1, x_2)  + 1 \big) \Big).
			\end{multline}
			The proof is a direct modification of the proof of \cite[(3.1), (3.2)]{MaMarOffDiag}, which states in the same notations that 
			\begin{align}
				\label{eq_hk_est_22}
				&
				\Big| 
				\exp \Big( - \frac{u}{p} \laplcomp_p \Big)
				 (x_1, x_2) \Big|_{\ccal^k(X \times X)} \leq
				 C p^{n + k/2}
				\cdot  \exp \Big(c_1 u - \frac{ap}{u} \cdot \dist(x_1, x_2) \Big),
				\\
				\label{eq_hk_est_2}
				&
				\Big| 
				\laplcomp_p \exp \Big( - \frac{u}{p} \laplcomp_p \Big)
				 (x_1, x_2) \Big|_{\ccal^k(X \times X)} \leq
				 C p^{n + 1 + k/2}
				\cdot  \exp \Big(- c_2 u - \frac{ap}{u} \cdot \dist(x_1, x_2) \Big).
			\end{align}
			\par 
			More precisely, remark first that, (\ref{eq_hk_est_12}), (\ref{eq_hk_est_1}) are consequences of (\ref{eq_hk_est_22}) and (\ref{eq_hk_est_2}) for $ \dist(x_1, x_2) > \frac{r_0}{2}$. 
			Now,  in the case $ \dist(x_1, x_2) < \frac{r_0}{2}$, the proof of (\ref{eq_hk_est_12}), (\ref{eq_hk_est_1}) is the same as the proof of (\ref{eq_hk_est_22}), (\ref{eq_hk_est_2}) in \cite[proof of Theorem 4]{MaMarOffDiag} with only one change. In the notations of \cite[proof of Theorem 4]{MaMarOffDiag}, for $h$ instead of $\sqrt{p} \dist(x_1, x_2)/\epsilon$, one should take $\sqrt{p} \frac{r_0}{2 \epsilon}$.
			Then, again in the notations of \cite[proof of Theorem 4]{MaMarOffDiag}, due to finite propagation speed of solutions of hyperbolic equations, cf. \cite[(3.9)]{MaMarOffDiag}, and the fact that the Laplacians $\laplcomp_p$ and $\laplcomp'_p$ coincide in $U$, the difference $(H_{u, h} (\frac{1}{\sqrt{p}} \laplcomp_p) - H_{u, h} (\frac{1}{\sqrt{p}} \laplcomp'_p) )(x_1, \cdot)$ vanishes. 
			Once (\ref{eq_hk_est_12}) and (\ref{eq_hk_est_1}) are established, the proof now proceeds in exactly the same way as in \cite[proof of Theorem 4]{MaMarOffDiag} with one final modification: the estimates of the quantities associated to $\laplcomp_p$ should be replaced by the estimates of the difference of the same quantities associated to $\laplcomp_p$ and $\laplcomp'_p$.
			\par
			Now, (\ref{eq_bk_local}) follows from (\ref{eq_hk_est_12}), (\ref{eq_hk_est_1}) and the following identity, cf. \cite[(3.14)]{MaMarOffDiag},
			\begin{multline}
				B_p^X
				-
				B_p^X{}'
				=
				\Big(
				\exp \Big( - \frac{u}{p} \laplcomp_p \Big)
				-
				\exp \Big( - \frac{u}{p} \laplcomp'_p \Big)
				\Big)
				\\	
				+			
				\int_u^{+ \infty}
				\Big(
				\frac{1}{p} \laplcomp'_p \exp \Big( - \frac{u_1}{p} \laplcomp'_p \Big) 
				-
				\frac{1}{p} \laplcomp_p \exp \Big( - \frac{u_1}{p} \laplcomp_p \Big) 
				\Big) du_1,
			\end{multline}
			by using exactly the same estimates as in \cite[proof of Theorem 1]{MaMarOffDiag}.
		\end{proof}
		\begin{proof}[Proof of Theorem \ref{thm_local_logbk}]
			The proof is an easy modification of the proof of Theorem \ref{thm_logbk_exp_dc}.
			The only essential difference is that instead of working with the approximate projection, we work with the difference of the two operators associated to different geometries.
			We use the notation from the proof of Theorem \ref{thm_logbk_exp_dc}.
			We denote by $B'_p{}^{\perp, a}$ the operator, constructed as $B_p^{\perp, a}$ in (\ref{eq_bpperp_a_ap_defn}), but for  $X', Y', (L', h^{L'}), (F', h^{F'}), dv_{X'}, dv_{Y'}$.
			We fix a smooth cut-off function $\rho_0 : U \to [0, 1]$ (resp.  $\rho_1$), equal to $1$ on $V$ (resp. on $\mathbb{B}_{y_0}^{X}(\frac{3}{4} r_0)$), and $0$ on $\partial \mathbb{B}_{y_0}^{X}(\frac{2}{3} r_0)$ (resp. on $\partial \mathbb{B}_{y_0}^{X}(r_0)$).
			Since $U$ and $U'$ are identified by a fixed diffeomorphism, we may regard $\rho_0$, $\rho_1$ as functions, defined on $X$ or $X'$ by extending them by zero.
			For $r \in \nat$, we consider the difference
			\begin{equation}\label{eq_bp_bppr_1}
				D_p := 
				\rho_0
				\Big(				
				\Big(
						B_p^X - \frac{B_{p}^{\perp, a}}{2b} 
					\Big)^{r}
					-
					B_p^X
				\Big)
				\rho_0
				-
				\rho_0
				\Big(
				\Big(
						B_p^X{}'  - \frac{B'_{p}{}^{\perp, a}}{2b} 
					\Big)^{r}
					-
					B_p^X{}' 
				\Big)
				\rho_0.
			\end{equation}
			Once the brackets in (\ref{eq_bp_bppr_1}) are opened, one can replace each multiplicand $A$ in the resulting expression by the sum $\rho_1 A \rho_1 + (1 - \rho_1) A \rho_1 + \rho_1 A (1 - \rho_1) + (1 - \rho_1) A (1 - \rho_1)$.
			We denote by $R_p$ the sum of the terms, which contain at least one $(1 - \rho_1)$-term.
			Clearly, by Lemma \ref{lem_bnd_prod_a} applied for $W := \partial \mathbb{B}_{y_0}^{X}(\frac{5}{7} r_0)$, there is $p_1 \in \nat^*$, $x, x' \in V$ such that  for $p \geq p_1$, for any $r_0 > 0$, there are $c_1, C_1 > 0$, such that
			\begin{equation}\label{eq_spec_est_bpaprime_3222}
				\big|
				R_p(x, x')
				\big|_{\ccal^k(V \times V)}
				\leq 
				C_1^r p^{n + k/2} \cdot \exp (
				-c_1 \sqrt{p} 
				).
			\end{equation}
			\par 
			As ${\rm{supp}} \, \rho_1 \subset U$, using the diffeomorphism between $U$ and $U'$, we may interpret all the operators $\rho_1 B_{p}^{\perp, a} \rho_1$, $\rho_1 (B'_{p}{}^{\perp, a}) \rho_1$, $\rho_1 B_p^X \rho_1$, $\rho_1 B_p^X{}' \rho_1$ as operators, acting over the same space, $X$.
			Let us now study the terms, where only $\rho_1$ appear.
			For those terms, it is easy to see that one can rearrange the summands so that the terms in (\ref{eq_bp_bppr_1}) with $\rho_1$ can be expressed as a sum of $(2^r- 1)2^r$ elements, each of which would contain as a multiplicand either $\rho_1 (B_{p}^{\perp, a} - B'_{p}{}^{\perp, a}) \rho_1$, or $\rho_1 (B_p^X - B_p^X{}') \rho_1$ and one of $\rho_1 (B_{p}^{\perp, a} ) \rho_1$ or $\rho_1 (B'_{p}{}^{\perp, a}) \rho_1$.
			Then by Lemma \ref{lem_bnd_prod_a}  applied for $W := \partial U$, Theorem \ref{thm_bk_off_diag}, Proposition \ref{prop_bk_local} and (\ref{eq_bpa_bnd_2}), we see that there are $c_2, C_2 > 0$, such that for any $p \geq p_1$, $x, x' \in V$, we get the bound
			\begin{equation}\label{eq_spec_est_bpaprime_3}
				\big|
				(D_p - R_p)(x, x')
				\big|_{\ccal^k(V \times V)}
				\leq 
				C_2^r p^{n + k/2} \cdot 
				\exp 
				(
					- c_2 \sqrt{p} 
				).
			\end{equation}
			We assume for simplicity that $C_2 > C_1$ and $c_2 < c_1$.
			By summing up (\ref{eq_spec_est_bpaprime_3222}) and (\ref{eq_spec_est_bpaprime_3}), we finally deduce that for any $p \geq p_1$, $x, x' \in V$, we have
			\begin{equation}\label{eq_spec_est_bpaprime_355}
				\big|
				D_p(x, x')
				\big|_{\ccal^k(V \times V)}
				\leq 
				2 C_2^r p^{n + k/2} \cdot
				\exp 
				(
					- c_2 \sqrt{p} 
				).
			\end{equation}
			\par 
			Now, by  taking a sum of (\ref{eq_spec_est_bpa_2}) and the analogous estimate for $X', Y', (L', h^{L'}), (F', h^{F'}), dv_{X'}, dv_{Y'}$, we get that for any $k \in \nat$, there is $C_3 > 0$, such that for any $p \geq p_1$, $r \in \nat^*$, $x, x' \in V$, we have
			\begin{equation}\label{eq_spec_est_bpaprime_2}
				\big|
					(
					D_p
					+
					B_{p}^{\perp}
					-
					B'_{p}{}^{\perp}
					)(x, x')
				\big|_{\ccal^k(V \times V)}
				\leq
				C_3 p^{n + k/2} \epsilon^r.
			\end{equation}
			We now adjust $r$ as follows
			\begin{equation}\label{eq_r_choice2}
				r := \Big \lceil \frac{c}{4 \log (\max(C_2, 2))} \sqrt{p} \Big \rceil.
			\end{equation}
			Then the contribution of $C_2^r$ in (\ref{eq_spec_est_bpaprime_355}) gets eliminated by the last multiplicand in the right-hand side of  (\ref{eq_spec_est_bpaprime_355}).
			The proof is now finished by (\ref{eq_spec_est_bpaprime_355}) and (\ref{eq_spec_est_bpaprime_2}).
		\end{proof}
		\end{sloppypar}
		
\subsection{Bergman kernel asymptotics in Fermi coordinates}\label{sect_ct_asymp}
	The main goal of this section is to prove that after a reparametrization given by a homothety with factor $\sqrt{p}$ in Fermi coordinates, Bergman kernel admits a complete asymptotic expansion in powers of $\sqrt{p}$, as $p \to \infty$. 	
	The proof relies on the analogous result of Dai-Liu-Ma \cite{DaiLiuMa}, stated in geodesic coordinates, and the calculations from Sections \ref{sect_coord_syst}, \ref{sect_par_transport}.
	\par
	We use notations from Section \ref{sect_intro} and assume that $(X, g^{TX})$ is of bounded geometry.
	Recall that $A \in \ccal^{\infty}(Y, T^*Y \otimes \enmr{TX|_Y})$, $R > 0$, $B \in \ccal^{\infty}(Y, {\rm{Sym}}^2(T^*X|_Y) \otimes TX|_Y)$,  and $\nu \in \ccal^{\infty}(Y, N)$ were defined in (\ref{eq_sec_fund_f}), (\ref{eq_r_defn_const}), (\ref{eq_b_defn}) and (\ref{eq_mn_curv_d}) respectively.
	\par 
	We fix a point $y_0 \in Y$ and an orthonormal frame $(e_1, \ldots, e_{2m})$ (resp. $(e_{2m+1}, \ldots, e_{2n})$) in $(T_{y_0}Y, g^{TY})$ (resp. in $(N_{y_0}, g^{N}_{y_0})$) as in (\ref{eq_cond_jinv}).
	Recall that Fermi coordinates, $\psi_{y_0}$, were defined in (\ref{eq_defn_fermi}).
	Recall that the function $\kappa_X$ in a neighborhood of $y_0$, was defined in (\ref{eq_defn_kappaxy}).
	We fix an orthonormal frame $(f_1, \ldots, f_r)$ of $(F_{y_0}, h^{F}_{y_0})$ and define the orthonormal frame $(\tilde{f}_1, \ldots, \tilde{f}_r)$ of $(F, h^F)$ in a neighborhood of $y_0$, as in Section \ref{sect_bnd_geom_cf}.
	Similarly, we trivialize $(L, h^L)$.
	The choice of those frames and the associated dual frames allows us to interpret $B_p^X(x_1, x_2)$ as an element of $\enmr{F_{y_0}}$ for $x_1, x_2 \in X$ in a neighborhood of $y_0$.
	Recall that $\mathscr{P}_{n}$ was defined in (\ref{eq_berg_k_expl}).
	\begin{prop}\label{prop_berg_off_diag}
			For any $r \in \nat$, there are $J_r(Z, Z') \in \enmr{F_{y_0}}$ polynomials in $Z, Z' \in \real^{2n}$, satisfying the same assumptions as polynomials from Theorem \ref{thm_berg_perp_off_diag}, such that for $F_r := J_r \cdot \mathscr{P}_{n}$, the following holds.
			There are $\epsilon, c > 0$, such that for any $k, l, l' \in \nat$, there is $C > 0$, such that for any $y_0 \in Y$, $p \in \nat^*$, $Z, Z' \in \real^{2n}$, $|Z|, |Z'| \leq \epsilon$, $\alpha, \alpha' \in \nat^{2n}$, $|\alpha|+|\alpha'| \leq l$, $Q^3_{k, l, l'} := 3 (n + k + l' + 2) + l$, the following bound holds
			\begin{multline}\label{eq_berg_off_diag}
				\bigg| 
					\frac{\partial^{|\alpha|+|\alpha'|}}{\partial Z^{\alpha} \partial Z'{}^{\alpha'}}
					\bigg(
						\frac{1}{p^n} B_p^X\big(\psi_{y_0}(Z), \psi_{y_0}(Z') \big)
						-
						\sum_{r = 0}^{k}
						p^{-\frac{r}{2}}						
						F_r(\sqrt{p} Z, \sqrt{p} Z') 
						\kappa_X^{-\frac{1}{2}}(Z)
						\kappa_X^{-\frac{1}{2}}(Z')
					\bigg)
				\bigg|_{\ccal^{l'}(Y)}
				\\
				\leq
				C p^{-(k + 1 - l) / 2}
				\Big(1 + \sqrt{p}|Z| + \sqrt{p} |Z'| \Big)^{Q^3_{k, l, l'}}
				\exp\Big(- c \sqrt{p} |Z - Z'| \Big),
			\end{multline}
			where the $\ccal^{l'}$-norm is taken with respect to $y_0$.
			Also, the following identity holds
			\begin{equation}\label{eq_jo_expl_form}
				J_0(Z, Z') = {\rm{Id}}_{F_{y_0}}.
			\end{equation}
			Moreover, under the assumptions (\ref{eq_comp_vol_omeg}), we have
			\begin{equation}\label{eq_j1_expl_form}
					J_1(Z, Z') = {\rm{Id}}_{F_{y_0}} \cdot \pi 
					\Big(
			 g^{TX}_{y_0} \big(z_N, A(\overline{z}_Y - \overline{z}'_Y) (\overline{z}_Y - \overline{z}'_Y) \big)
			 +
			 g^{TX}_{y_0} \big(\overline{z}'_N, A(z_Y - z'_Y) (z_Y - z'_Y) \big)			 
			 \Big).
			\end{equation}
		\end{prop}
		\par 
		Before describing the proof of Proposition \ref{prop_berg_off_diag}, let us recall the relevant asymptotic expansion of Dai-Liu-Ma.
		Recall first that for $x_0 \in X$, the coordinates $\phi_{x_0}$ in a neighborhood of $x_0$ were defined in (\ref{eq_phi_defn}).
		Let us define the function $\kappa'_{X}$ in a neighborhood of $x_0$ by the following formula
		\begin{equation}
			(\phi_{x_0}^* dv_X)(Z) 
			=
			\kappa'_{X}(Z)
			d Z_1 \wedge \cdots \wedge d Z_{2n}.
		\end{equation}
		We use the orthonormal frame $(\tilde{f}'_1, \ldots, \tilde{f}'_r)$ of $(F, h^F)$, defined in (\ref{eq_frame_tilde}) in a neighborhood of $y_0$.
		Similarly, we trivialize $(L, h^L)$.
		The choice of those frames as well as the associated dual ones allows us to interpret the Schwartz kernel of $B_p^X$ as an element of $\enmr{F_{y_0}}$ for $x_1, x_2 \in X$ in a neighborhood of $y_0$. We denote this element here by $B_p^{X \phi}(x_1, x_2)$ to distinguish it from $B_p^X(x_1, x_2)$ previously defined.
		\begin{thm}[{Dai-Liu-Ma \cite[Theorem 4.1.18]{DaiLiuMa}, cf. \cite[Theorems 4.2.9 and Problem 6.1]{MaHol} and \cite[Theorem 4.3]{KordMaMarin} }]\label{thm_berg_dailiuma}
			For any $r \in \nat$, $x_0 \in X$, there are $J^{\phi}_r(Z, Z') \in \enmr{F_{x_0}}$ polynomials in $Z, Z' \in \real^{2n}$, satisfying the same assumptions as polynomials from Theorem \ref{thm_berg_perp_off_diag}, such that for $F^{\phi}_r := J^{\phi}_r \cdot \mathscr{P}_n$, the following holds.
			There are $\epsilon, c > 0$, such that for any $k, l, l' \in \nat$, there exists $C > 0$, such that for any $x_0 \in X$, $p \in \nat^*$, $Z, Z' \in \real^{2n}$, $|Z|, |Z'| \leq \epsilon$, $\alpha, \alpha' \in \nat^{2n}$, $|\alpha|+|\alpha'| \leq l$, and $Q^4_{k, l, l'} := 2 (n + k + l' + 2) + l$, we have
			\begin{multline}\label{eq_berg_dailiuma}
				\bigg| 
					\frac{\partial^{|\alpha|+|\alpha'|}}{\partial Z^{\alpha} \partial Z'{}^{\alpha'}}
					\bigg(
						\frac{1}{p^n} B_p^{X \phi}\big(\phi_{x_0}(Z), \phi_{x_0}(Z') \big)
						-
						\sum_{r = 0}^{k}
						p^{-\frac{r}{2}}						
						F^{\phi}_r(\sqrt{p} Z, \sqrt{p} Z') 
						\kappa'_{X}{}^{-\frac{1}{2}}(Z)
						\kappa'_{X}{}^{-\frac{1}{2}}(Z')
					\bigg)
				\bigg|_{\ccal^{l'}(X)}
				\\
				\leq
				C p^{-(k + 1 - l) / 2}
				\Big(1 + \sqrt{p}|Z| + \sqrt{p} |Z'| \Big)^{Q^4_{k, l, l'}}
				\exp\Big(- c \sqrt{p} |Z - Z'| \Big),
			\end{multline}
			where the $\ccal^{l'}$-norm is taken with respect to $y_0$.
			Also, the following identity holds
			\begin{equation}\label{eq_jopr_calc}
				J^{\phi}_0(Z, Z') = {\rm{Id}}_{F_{y_0}}.
			\end{equation}
			Moreover, under the assumptions (\ref{eq_comp_vol_omeg}), we have
			\begin{equation}\label{eq_j1pr_calc}
				J^{\phi}_1(Z, Z') = 0.
			\end{equation}
		\end{thm}
		\begin{proof}[Proof of Proposition \ref{prop_berg_off_diag}]
			Recall that the diffeomorphism $h$ was defined in (\ref{eq_h_defn_tr_m}), and the functions $\xi_L$, $\xi_F$ were defined in (\ref{eq_frame_tilde}).
			Directly from the definitions, we obtain the following relation between the Schwartz kernels
			\begin{multline}\label{eq_rel_kernels}
				B_p^X\big(\psi(Z), \psi(Z') \big)
				=
				\exp \big(-p \xi_L^{*} - \xi_F^{*} \big)(\psi(Z'))
				\cdot
				\\
				B_p^{X \phi}\big(\phi(h(Z)), \phi(h(Z')) \big)
				\cdot
				\exp \big(-p \xi_L - \xi_F \big)(\psi(Z)).
			\end{multline}
			From (\ref{eq_rel_kernels}), we see that to establish Proposition \ref{prop_berg_off_diag}, it is necessary to study the Taylor expansions of each term in (\ref{eq_berg_dailiuma}) for $Z := h(Z)$, $Z' := h(Z')$.
			\par 
			From Theorem \ref{thm_berg_dailiuma}, (\ref{eq_rel_kernels}), the fact that from Proposition \ref{prop_diff_exp}, $h(Z) = Z + O(|Z|^2)$, the fact that from Proposition \ref{prop_phi_fun_exp}, $\xi_F(\psi(Z)) = O(|Z|)$ and  $\xi_L(\psi(Z)) = O(|Z|^3)$, the fact that the above $O$-terms are uniform by the coordinate description of bounded geometry, and the fact that the coefficients of the higher order Taylor expansions can be expressed in terms of $R^{TX}$, $A$, $R^F$, as described in Propositions \ref{prop_diff_exp}, \ref{prop_phi_fun_exp}, we deduce (\ref{eq_berg_off_diag}) and (\ref{eq_jo_expl_form}).
			\par 
			To establish (\ref{eq_j1_expl_form}), let us place ourselves in coordinates as in (\ref{eq_berg_k_expl}).
			By (\ref{eq_z_ovz_id}), we see that
			\begin{equation}
				2 \sum_{i = 1}^{n} z_i \overline{z}'_i
				=
				g^{TX}_{y_0}
				\Big(
				({\rm{Id}} - \imun J) Z,
				({\rm{Id}} + \imun J) Z'
				\Big).
			\end{equation}
			From this, Proposition \ref{prop_diff_exp} and (\ref{eq_z_ovz_id}), we get
			\begin{equation}\label{eq_sqr_newcrrd}
				\begin{aligned}
					&
					|h(Z)|^2
					=
					|Z|^2
					+
					2 g^{TX}_{y_0}(B(Z), Z)
					+
					O(|Z|^4), 
					\\
					&
					2
					\sum_{i = 1}^{n}
					(z_i \overline{z}'_i)(h(Z), h(Z'))
					=
					2
					\sum_{i = 1}^{n}
					z_i \overline{z}'_i
					+
					4 g^{TX}_{y_0} \big( B(Z), \overline{z}' \big)
					+
					4 g^{TX}_{y_0} \big( z, B(Z') \big) 
					+
					O(|Z|^4).
				\end{aligned}
			\end{equation}
			From (\ref{eq_berg_k_expl}) and (\ref{eq_sqr_newcrrd}), we conclude that
			\begin{equation}\label{eq_pn_chng_crd}
			\begin{aligned}
				& \mathscr{P}_{n}(\sqrt{p} h(Z), \sqrt{p} h(Z'))
				=
				\mathscr{P}_{n}(\sqrt{p} Z, \sqrt{p} Z')
				\cdot
				\\
				& 
				\quad   \qquad
				\cdot
				\bigg(
				1 
				-
				\frac{\pi}{\sqrt{p}}
				\Big(
					g^{TX}_{y_0}(B(\sqrt{p} Z), \sqrt{p} Z)
					+
					g^{TX}_{y_0}(B(\sqrt{p} Z'), \sqrt{p} Z')
					\\
					& 
					\qquad \qquad  \quad
					-
					2 g^{TX}_{y_0} \big( B(\sqrt{p} Z), \sqrt{p} \overline{z}' \big)
					-
					2 g^{TX}_{y_0} \big( \sqrt{p} z, B(\sqrt{p} Z') \big) 
				\Big)
				+
				O \Big(\frac{1}{p} |\sqrt{p} Z|^4\Big)
				\bigg).
			\end{aligned}
			\end{equation}
			Now, from Proposition \ref{prop_phi_fun_exp} and (\ref{eq_gtx_def}), we get
			\begin{equation}\label{eq_t_exp_xi}
				\exp \big(-p \xi_L - \xi_F \big)(\psi(Z))
				=
				1
				-
				\frac{\imun \pi}{3 \sqrt{p}}
				g^{TX}_{y_0} \big( J \sqrt{p} Z, B(\sqrt{p} Z) \big)
				+
				O\Big(|Z|^2
				+
				\frac{1}{p} |\sqrt{p} Z|^4
				\Big).
			\end{equation}
			By Lemma \ref{lem_scal_prod} and (\ref{eq_nu_zero}), applied for $M := X$, $H := Y$ and $T := {\rm{Id}}$, we get
			\begin{equation}\label{eq_kappa_t}
				\kappa_X^{1/2}(Z)
				=
				1 
				+
				O(|Z|^2),
			\end{equation}
			where $\nu \in \ccal^{\infty}(Y, N)$ was defined (\ref{eq_mn_curv_d}).
			We now apply (\ref{eq_kappa_t}) for $X := X$ and $Y := \{y_0\}$, to get
			\begin{equation}\label{eq_kappapr_t}
				\kappa'_X{}^{1/2}(Z)
				=
				1 
				+
				O(|Z|^2).
			\end{equation}
			From (\ref{eq_berg_dailiuma}), (\ref{eq_rel_kernels}), (\ref{eq_pn_chng_crd}), (\ref{eq_t_exp_xi}), (\ref{eq_kappa_t}), (\ref{eq_kappapr_t}) and an easy calculation, we deduce
			\begin{equation}\label{eq_j1_expl_form1111}
				\begin{aligned}
					&
					J_1(Z, Z') = - \pi {\rm{Id}}_{F_{y_0}} \cdot \bigg(
					g^{TX}_{y_0}(B(Z), Z)
					+
					g^{TX}_{y_0}(B(Z'), Z')
					\\
					& 
					\quad \qquad 
					-
					2
					\Big(
					g^{TX}_{y_0} \big( B(Z), \overline{z}' \big)
					+
					g^{TX}_{y_0} \big( z, B(Z') \big) 
					\Big)
					\\
					& 
					\quad \qquad 
					+
					\frac{\imun}{3}
					\Big(
						g^{TX}_{y_0} \big( J Z, B(Z) \big)
						-
						g^{TX}_{y_0} \big( J Z', B(Z') \big)
					\Big)
				 \bigg).
				\end{aligned}
			\end{equation}
			By using the fact that $A$ commutes with $J$ and takes its values in skew-adjoint matrices, we get
			\begin{equation}\label{eq_calc_lbkk_0}
			\begin{aligned}
				&
				g^{TX}_{y_0}(Z, B(Z))
				=
				-\frac{1}{2} 
				\Big(
				g^{TX}_{y_0}(z_N, A(Z_Y) \overline{z}_Y)
				+
				g^{TX}_{y_0}(\overline{z}_N, A(Z_Y) z_Y)
				\Big),
				\\
				&
				g^{TX}_{y_0}(JZ, B(Z))
				=
				\frac{3 \imun}{2} 
				\Big(
				g^{TX}_{y_0}(z_N, A(Z_Y) \overline{z}_Y)
				-
				g^{TX}_{y_0}(\overline{z}_N, A(Z_Y) z_Y)
				\Big).
			\end{aligned}
			\end{equation}
			Now, we use (\ref{eq_a_no_tors}) and the fact that $A$ commutes with $J$ to deduce
			\begin{equation}\label{eq_calc_lbkk_1}
				A(Z_Y) z_Y = A(z_Y) z_Y, \qquad A(Z_Y) \overline{z}_Y = A(\overline{z}_Y) \overline{z}_Y.
			\end{equation}
			From (\ref{eq_a_no_tors}), (\ref{eq_j1_expl_form1111}), (\ref{eq_calc_lbkk_0}) and (\ref{eq_calc_lbkk_1}), we deduce (\ref{eq_j1_expl_form}).
		\end{proof}
	
	\subsection{Orthogonal Bergman kernel asymptotics, a proof of Theorem \ref{thm_berg_perp_off_diag}}\label{sect_as_exp_logbk}
		The main goal of this section is to prove Theorem \ref{thm_berg_perp_off_diag}. 
		For this, we first use the localization property of the orthogonal Bergman kernel, established in Theorem \ref{thm_local_logbk}, to reduce our problem to a special case $X' := \comp^{n}$, $Y' := \comp^{m}$ and trivial $L', F'$ (but endowed with non trivial metrics).
		Then on the pair $X', Y'$ we make a homothety and show that the renormalization of the relevant operators converges to the perturbations of the model operators from Section \ref{sect_model_calc}.
		\par 
		Let us first describe precisely the construction of $X', Y', L', F'$ and the metrics on them.
		As we rely later on the complex structure of $X', Y', L', F'$ in an essential way, the construction through geodesic coordinates, as in \cite{MaHol}, wouldn't suffice for our needs.
		We rely here in our non-compact setting on the results of Section \ref{sect_stein_atl}, and we propose an approach, which in the compact case essentially coincides Dinh-Ma-Nguyen \cite[after (2.23)]{MaDinhPac}.
		\par 
		The idea is to use holomorphic coordinates and holomoprhic frames of the vector bundles around the fixed point $y_0 \in Y$ and construct the associated objects on the neighborhood of $0 \in \comp^n$ using those trivializations.
		Then we extend those objects by the partition of unity to $\comp^n$. 
		A special care has to be taken to preserve the positivity of the line bundle.
		As we would like to compare the original setting with this localized one in a uniform way, according to Theorem \ref{thm_local_logbk}, it is fundamental to choose holomorphic coordinates and frames on geodesic balls of uniform size around $y_0$.
		At this point, the results from Section \ref{sect_stein_atl} will play a crucial role.
		\par
		We fix $y_0 \in Y$.
		Let $r_c, r_c^1 > 0$ and $\chi : \mathbb{B}_{y_0}^X(r_c) \to \comp^n$, $U := \mathbb{B}_{y_0}^X(r_c)$, $V := \Im \chi$, as in Theorem \ref{lem_hol_coord_exst} and Remark \ref{rem_hol_coord_exst}.
		We put $X' := \comp^n$, $Y' := \comp^m$.
		\par 
		Let holomorphic frame $\sigma$ of $L$ over $\mathbb{B}_{y_0}^X(r_0)$, $r_0 > 0$, be constructed as in Lemma \ref{cor_hl_fr_bnd_g}. 
		Assume, for simplicity, $4 r_c < r_0$ (if this is not the case, put $r_c := \frac{r_0}{8}$).
		Define the function $\theta(Z)$ over $V$ by 
		\begin{equation}
			\exp(- 2 \theta(Z)) 
			:=
			h^L(\sigma, \sigma).
		\end{equation}
		From Lemma \ref{lem_ell_reg_sob_bnd_g} and (\ref{eq_f_un_bnd_geom}), we see that there is $C > 0$, which depends only on $r_c$ and $C_{n + 6}$ from (\ref{eq_bnd_curv_tm}) and (\ref{eq_bnd_re_ck}), such that
		\begin{equation}\label{eq_theta_c3_bnd}
			\big\| 
				\theta 
			\big\|_{\ccal^2(V)}
			\leq 
			C.
		\end{equation}
		Denote by $\theta_{0}^{[1]}$ and $\theta_{0}^{[2]}$ the first and the second order Taylor expansions of $\theta_{0} := \theta \circ \chi^{-1}$ at $0$.
		For $r_c^1 > \epsilon > 0$, we now define $\theta_{\epsilon} : \comp^{n} \to \real$ as follows
		\begin{equation}\label{eq_tht_eps}
			\theta_{\epsilon}(Z) := 
			\rho \Big(\frac{|Z|}{\epsilon} \Big) \theta_0(Z)
			+
			\Big( 
				1 - \rho \Big(\frac{|Z|}{\epsilon} \Big) 
			\Big)
			\cdot
			\Big(
				\theta_0(0) + \theta_0^{[1]}(Z) + \theta_0^{[2]}(Z)
			\Big),
		\end{equation}		
		where $\rho$ is a bump function as in (\ref{defn_rho_fun}).
		Let $h^{L'}_{\epsilon}$ be the metric on $L' := X' \times \comp$ defined by 
		\begin{equation}
			h^{L'}_{\epsilon}(1, 1)(Z) := \exp(- 2 \theta_{\epsilon}(Z)).
		\end{equation}
		Let $R^{L'}_{\epsilon}$ be the curvature of the Chern connection $\nabla^{L'}_{\epsilon}$ on $(L', h^{L'}_{\epsilon})$.
		By (\ref{eq_gtx_def}), (\ref{eq_theta_c3_bnd}), it is easy to see that there is a constant $\epsilon_0 > 0$, such that for any $\epsilon < \epsilon_0$, we have
		\begin{equation}\label{eq_rl_pos}
			\inf \Big\{
				\imun R^{L'}_{\epsilon, Z}(u, Ju) / |u|^2 : u \in T_{Z}X', Z \in X'
			\Big\}
			\geq 
			\pi,
		\end{equation}
		where $|u|$ is the norm of $u \in T_{Z}X'$, calculated by a (trivial) identification of $T_{Z}X'$ with $T_{0}X'$.
		Moreover, $\epsilon_0$ depends only on $r_c$ and $C_{n + 6}$ from (\ref{eq_bnd_curv_tm}) and (\ref{eq_bnd_re_ck}).
		From now on, we fix such $\epsilon_0$ and remove it from all subsequent subscripts.
		We assume, for simplicity, that $4r_c^1  < \epsilon_0$.
		From (\ref{eq_rl_pos}), we see, in particular, that $(L', h^{L'})$ is positive. 
		We denote by $g^{TX'}$ the metric on $X'$, defined through  $(L', h^{L'})$ as in (\ref{eq_gtx_def}).
		Of course, by (\ref{eq_tht_eps}), $\chi$ is then a local holomorphic isometry, defined over $\mathbb{B}_{y_0}^{X}(\frac{r_c}{2})$, between $(X, g^{TX})$ and $(X', g^{TX'})$.
		\par 		
		Now, let $r_0 > 0$ and a holomorphic frame $(f_1, \ldots, f_r)$ of $F$ over $\mathbb{B}_{y_0}^{X}(r_0)$ be as in Lemma \ref{cor_hl_fr_bnd_g}.
		Assume, for simplicity, that $4 r_c^1 < r_0$.
		Define the function $h^{F}_{i j}(x) \in \comp$, $x \in U$, as follows $h^{F}_{i j} := h^{F}(f_i, f_j)$.
		Define the metric $h^{F'}$ on $F' := X' \times \comp^{r}$ by 
		\begin{equation}
			h^{F'}(1_i, 1_j)(Z) :=  h^{F}_{i j} \Big(\rho \Big(\frac{|Z|}{r_c^1} \Big) \cdot Z \Big).
		\end{equation}
		where $1_l$, $l \in 1, \ldots, r$ is the constant vector in $F'$, given by $(0, \ldots, 0, 1, 0, \ldots 0)$, where $1$ is put in the $l$-th place.
		The pair $(F', h^{F'})$ is a Hermitian vector bundle on $X'$.
		\par 
		Clearly, the triple $(X', Y', g^{TX'})$ and the Hermitian vector bundles $(L', h^{L'})$,  $(F', h^{F'})$ are of bounded geometry.
		Moreover, the corresponding constants $C_k$ from (\ref{eq_bnd_curv_tm}), (\ref{eq_bnd_a_ck}), (\ref{eq_bnd_re_ck}), associated to   $(X', Y', g^{TX'})$, $(L', h^{L'})$,  $(F', h^{F'})$, can be bounded in terms of  the corresponding constants $C_{k + n + 6}$ from (\ref{eq_bnd_curv_tm}), (\ref{eq_bnd_a_ck}), (\ref{eq_bnd_re_ck}), associated to   $(X, Y, g^{TX})$, $(L, h^{L})$,  $(F, h^{F})$.
		\par 
		We denote by $g^{TX'}_{0}$, $h^{F'}_{0}$ the restrictions of $g^{TX'}$, $h^{F'}$ to $0 \in X'$.
		We interpret $g^{TX'}_{0}$ as metric on $X'$ by using the standard trivialization of $TX'$ coming from the linear structure. 
		Similarly, we interpret $h^{F'}_{0}$ as Hermitian metric on $F'$ over $X'$ by using the trivialization of $F'$. 
		\par 
		As the expansion from Theorem \ref{thm_berg_perp_off_diag} is stated in Fermi coordinates, we need an analogue of those coordinates on $X'$.
		For technical reasons, which will become clear after (\ref{eq_defn_bt_cal}), we need a global diffeomorphism of $X'$, which coincides with Fermi coordinates in a small neighborhood of $0$.
		For $0 < \epsilon < \min\{r_c, R\}$, where $R > 0$ is as in (\ref{eq_r_defn_const}), we define $\psi_{\epsilon}(Z)$ for $Z \in \comp^n$ as follows
		\begin{equation}\label{eq_phieps}
			\psi_{\epsilon}(Z) = \psi(Z)  \rho \Big(\frac{|Z|}{\epsilon} \Big)  + Z 
			\Big( 
				1 - \rho \Big(\frac{|Z|}{\epsilon} \Big) 
			\Big),
		\end{equation}
		where $\psi(Z) \in X'$ is the Fermi coordinates, defined as in (\ref{eq_defn_fermi}), but for the triple $(X', Y', g^{TX'})$.
		From bounded geometry condition, we see that there is $\epsilon_2 > 0$ such that for $\epsilon < \epsilon_2$, the derivative of $\psi_{\epsilon}(Z)$ is invertible for all $Z \in \comp^n$, and $|D \psi_{\epsilon} - {\rm{Id}}| < \frac{1}{2}$.
		Moreover, $\epsilon_2$ depends only on $r_X, r_Y, r_{\perp}$ and $C_0$ from (\ref{eq_bnd_curv_tm}) and (\ref{eq_bnd_a_ck}).
		For simplicity, we assume that $4 r_c^1 < \epsilon_2$.
		We fix such $\epsilon_2$ and denote $\psi_{\epsilon_2}$ by $\psi_{0}$ from now on.
		As $\psi_0(Z)$ obviously satisfies  $\psi_0(Z) \to \infty$, as $|Z| \to \infty$, by the invertibility of $D \psi_{0}$ and Hadamards global inverse function theorem, cf. \cite[Theorem 6.2.4]{ImplFunThBook}, $\psi_0$ is a diffeomorphism.
		Clearly, as Fermi coordinates preserve $Y'$, by (\ref{eq_phieps}), $\psi_0|_{\comp^m}$ is a diffeomorphism on $\comp^m$.
		\par 
		We define the volume form $dv_{X'}$ on $X'$ as follows
		\begin{equation}
			dv_{X'} := \rho \Big(\frac{|Z|}{r_c} \Big) (\chi^{-1})^* dv_{X} + 
			\Big( 
				1 - \rho \Big(\frac{|Z|}{r_c} \Big)
			\Big) dZ_1 \wedge \cdots \wedge dZ_{2n}.
		\end{equation}
		Similarly  to (\ref{eq_defn_kappaxy}), we define the  function $\kappa_{X'} : X' \to \real$ as follows
		\begin{equation}\label{eq_defn_kappprime}
			(\psi_0^* dv_{X'}) (Z)
			=
			\kappa_{X'}(Z)
			d Z_1 \wedge \cdots \wedge d Z_{2n}.
		\end{equation}
		Clearly, since $\psi_0(Z) = Z$, as $|Z| \to \infty$, we have $\kappa_{X'}(Z) = 1$, as $|Z| \to \infty$.
		\par 
		Now, let us fix $e \in L'_0$ and $f_1, \ldots, f_r \in F'_0$, and consider the orthonormal frames $\tilde{e}$ and $\tilde{f}_1, \ldots, \tilde{f}_r$, constructed as in Theorem \ref{thm_ext_as_exp}, for $L', F'$, but instead of $\psi$, using $\psi_0$.
		As $\psi_0$ is globally defined, those frames become also globally defined.
		\par 
		Let us now consider the Bergman projector $B_p^X{}'$ (resp. the orthogonal Bergman projector $B^{\perp}_p{}'$) associated to $X', L', F'$.
		This is a self-adjoint operator, acting on the $L^2$-space $L^2(dv_{X'}, h^{L'{}^{\otimes p} \otimes F'})$.
		The above orthonormal frames allow us to see $B_p^X{}'$ (resp. $B^{\perp}_p{}'$) as a self-adjoint operator, acting on the $L^2$-space $L^2(dv_{X'}, h^{F'}_{0})$.
		We denote by $B_p^X{}'(x_1, x_2)$, $B'_p{}^{\perp}(x_1, x_2)$ the Schwartz kernels of those operators with respect to $dv_{X'}$.
		The following theorem essentially shows that it is enough to establish our main result of this section for $X', L', F'$ instead of $X, L, F$.
		\begin{lem}\label{lem_local_logbk_pr}
			There are $c > 0$, $p_1 \in \nat^*$, such that for any $k \in \nat$, there is $C > 0$, such that for any $y_0 \in Y$, $p \geq p_1$, $x_1, x_2 \in \mathbb{B}_{y_0}^{X}(\frac{r_c}{4})$, the following estimate holds
			\begin{equation}\label{eq_local_logbk_pr}
				\Big|  B_p^{\perp}(x_1, x_2) - B'_p{}^{\perp}(x_1, x_2) \Big|_{\ccal^k(X \times X)} \leq C \exp (- c \sqrt{p} ),
			\end{equation}
			where we implicitly identified $x_1, x_2$ to points in $X'$ using $\chi$.
		\end{lem}
		\begin{proof}
			It follows from Theorem \ref{thm_local_logbk} and the fact that $\chi$ is a holomorphic diffeomorhism, which extends to isometries between $(L, h^L)$, $(F, h^F)$ and  $(L', h^{L'})$, $(F', h^{F'})$ over $\mathbb{B}_{y_0}^X(\frac{r_c}{2})$.
		\end{proof}
		\par 
		The advantage of passing from $X, L, F$ to $X', L', F'$ is twofold.
		First, since all the vector bundles are now trivialized, the operators $B_p^X{}'$ and $B^{\perp}_p{}'$ might be considered as operators, acting on the same space (independent of $p$).
		Second, as $X'$ is equal to $\comp^n$, we can use the homothety on $X'$ in our analysis.
		We will use both of those features in what follows.
		\par 
		We define $t > 0$ as in (\ref{eq_t_p_rel}) and $S_t : \ccal^{\infty}(U, F_{y_0}) \to \ccal^{\infty}(\mathbb{B}_{y_0}^X(t R), F_{y_0})$ as follows
		\begin{equation}\label{eq_oper_st_defn}
			S_t f(Z) := f \Big(\frac{ Z}{t} \Big).
		\end{equation}
		Clearly, for any $f, f' \in L^2(g^{TX'}_{0}, h^{F'}_{0})$, we have
		\begin{equation}\label{eq_st_prop}
			\scal{S_t f}{S_t f'}_{L^2(g^{TX'}_{0})} = t^{2n} \scal{f}{f'}_{L^2(g^{TX'}_{0})}.
		\end{equation}
		We also consider a map $U : L^2(g^{TX'}_{0}, h^{F'}_{0}) \to L^2(dv_{X'}, h^{F'}_{0})$, defined as follows
		\begin{equation}
			(Uf)(Z) = \kappa_{X'}^{-1/2}(\psi_0^{-1}(Z)) \cdot f(\psi_0^{-1}(Z)).
		\end{equation}
		An easy verification using (\ref{eq_defn_kappprime}) shows that $U$ is well-defined and it is an isometry, i.e.
		\begin{equation}\label{eq_u_prop}
			\scal{U f}{U f'}_{L^2(dv_{X'})} = \scal{f}{f'}_{L^2(g^{TX'}_{0})}.
		\end{equation}
		\par 
		We now consider another operators $\mathcal{B}_t$, $\mathcal{B}_t^{\perp}$, acting on $L^2(g^{TX'}_{0}, h^{F'}_{0})$ as follows
		\begin{equation}\label{eq_defn_bt_cal}
		\begin{aligned}
			& \mathcal{B}_t := S_t^{-1} \circ U^{-1} \circ B_p^X{}' \circ U \circ S_t,
			\\
			& \mathcal{B}_t^{\perp} := S_t^{-1} \circ U^{-1} \circ B_p^{\perp}{}' \circ U \circ S_t.
		\end{aligned}
		\end{equation}
		From (\ref{eq_st_prop}) and (\ref{eq_u_prop}), we see that $\mathcal{B}_t$, $\mathcal{B}_t^{\perp}$ are self-adjoint.
		\par 
		We denote by $\mathcal{B}_t(Z, Z')$, $\mathcal{B}_t^{\perp}(Z, Z')$ the Schwartz kernels of $\mathcal{B}_t$, $\mathcal{B}_t^{\perp}$ with respect to the volume form $dv_{g^{TX'}_{0}}$ on $X'$. 
		An easy calculation shows that they are related to the Schwartz kernels $B_p^X{}'(Z, Z')$, $B_p^{\perp}{}'(Z, Z')$  of $B_p^X{}'$, $B^{\perp}_p{}'$, evaluated with respect to $dv_{X'}$ as follows
		\begin{equation}\label{eq_bt_bp_rel_schw}
		\begin{aligned}
			&
			\mathcal{B}_t(Z, Z')
			=
			t^{2n} B_p^X{}'\big(\psi_0(tZ), \psi_0(tZ') \big)
			\kappa_{X'}^{\frac{1}{2}}(tZ)
			\kappa_{X'}^{\frac{1}{2}}(tZ'),
			\\
			&
			\mathcal{B}_t^{\perp}(Z, Z')
			=
			t^{2n} B_p^{\perp}{}'\big(\psi_0(tZ), \psi_0(tZ') \big)
			\kappa_{X'}^{\frac{1}{2}}(tZ)
			\kappa_{X'}^{\frac{1}{2}}(tZ'),
		\end{aligned}
		\end{equation}
		\begin{lem}\label{prop_der_bound_bt}
			There is $\epsilon > 0$, such that for any $l \in \nat$, there exists $C > 0$, such that for any $Z \in \real^{2n}$, $|Z| \leq \frac{\epsilon}{t}$, $p \in \nat^*$, $f \in \Im(\mathcal{B}_t)$, $\alpha \in \nat^{2n}$, $|\alpha| \leq l$, we have
			\begin{equation}\label{eq_der_bound_bt}
				\Big|
				\frac{\partial^{|\alpha|}}{\partial Z^{\alpha}} f (Z)
				\Big|
				 \leq C \big\| f \big\|_{L^2(g^{TX'}_{0})}.
			\end{equation}
		\end{lem}
		\begin{proof}
			Clearly, by (\ref{eq_defn_bt_cal}), we have
			\begin{equation}
				\Im \mathcal{B}_t
				\subset
				(S_t^{-1} \circ U^{-1}) \big( H^{0}_{(2)}(X', L'{}^p \otimes F') \big).
			\end{equation}
			We conclude by this, Proposition \ref{prop_berg_off_diag} and (\ref{eq_bt_bp_rel_schw}).
		\end{proof}
		\par 
		We will now introduce the operator $\mathcal{E}^0$, sending the sections of $F'$ over $Y'$ to the sections of $F'$ over $X'$ by the following formula
		\begin{equation}
			(\mathcal{E}^0 f) (Z_Y, Z_N) = 	f(Z_Y) 
			\exp \Big(
				-\frac{\pi}{2} |Z_N|^2
			\Big).
		\end{equation}
		Denote by $\mathcal{A}_t$ the operator, acting on $L^2(g^{TX'}_{0}, h^{F'}_{0})$, as follows
		\begin{equation}\label{eq_a_cal_defn}
			\mathcal{A}_t := \mathcal{E}^0 \circ {\rm{Res}}_{Y'} \circ \mathcal{B}_t,
		\end{equation}
		where ${\rm{Res}}_{Y'}$ is the restriction operator (which is well-defined in the considered composition as $\mathcal{B}_t$ has smooth Schwartz kernel).
		We define the operator $\mathcal{C}_t$ on $L^2(g^{TX'}_{0}, h^{F'}_{0})$, as follows
		\begin{equation}\label{eq_ccal_def_at}
			\mathcal{C}_t := \mathcal{A}_t^{*} \circ \mathcal{A}_t,
		\end{equation}
		where $\mathcal{A}_t^{*}$ is the adjoint of $\mathcal{A}_t$.
		The operator $\mathcal{C}_t$ will be our main tool in the proof of Theorem \ref{thm_berg_perp_off_diag}.
		Remark its similarity with (\ref{eq_bpperp_a_ap_defn}).
		Let us study some of its properties.
		\begin{lem}\label{lem_ct_spec_prop}
			The operator $\mathcal{C}_t$ is self-adjoint, and there are constants $a, b > 0$, $p_1 \in \nat$, such that
			\begin{equation}
				\spec (\mathcal{C}_t) \subset \{0\} \cup [a, b],
			\end{equation}
			for any $p \geq p_1$.
			Moreover, we have the following
			\begin{equation}
				(\ker \mathcal{C}_t)^{\perp}
				=
				(S_t^{-1} \circ U) \big( H^{0, \perp}_{(2)}(X', L'{}^p \otimes F') \big).
			\end{equation}
		\end{lem}
		\begin{proof}
			The proof is identical to (\ref{eq_ker_bpa}) and (\ref{eq_bpa_sp_gap}).
		\end{proof}
		\begin{lem}\label{lem_ct_bnd_schhw}
			The Schwartz kernel $\mathcal{C}_t(Z, Z')$ of $\mathcal{C}_t$ with respect to $dv_{\comp^n}$ satisfies the following bound.
			There are $c > 0$, $p_1 \in \nat^*$, such that for any $k, l, l' \in \nat$, there is $C > 0$, such that for any $p \geq p_1$, $Z, Z' \in \real^{2n}$, $\alpha, \alpha' \in \nat^{2n}$, $|\alpha|+|\alpha'| \leq l$, we have
			\begin{equation}\label{eq_mathcalc_bouund}
				\Big| 
					\frac{\partial^{|\alpha|+|\alpha'|}}{\partial Z^{\alpha} \partial Z'{}^{\alpha'}}
					\mathcal{C}_t(Z, Z')
				\Big|_{\ccal^{l'}(Y)}
				\leq
				C
				\exp\Big(- c \big( |Z_Y - Z'_Y| + |Z_N| + |Z'_N| \big) \Big),
			\end{equation}
			where the $\ccal^{l'}$-norm is calculated with respect to $y_0$, which was fixed in the beginning of the section.
		\end{lem}
		\begin{proof}
			From Theorem \ref{thm_bk_off_diag}, (\ref{eq_bt_bp_rel_schw}) and (\ref{eq_a_cal_defn}), we conclude that a bound like  (\ref{eq_mathcalc_bouund}) holds for $\mathcal{A}_t(Z, Z')$ instead of $\mathcal{C}_t(Z, Z')$.
			It is uniform in $\ccal^{l'}$-norm with respect to the choice of $y_0 \in Y$ because the construction of $X'$, $Y'$, $(L', h^{L'})$, etc. depends smoothly on $y_0$ and the Bergman kernel expansion depends also smoothly on $y_0$. 
			We now conclude by Lemma \ref{lem_bnd_prod_a} and (\ref{eq_ccal_def_at}).
		\end{proof}
		\begin{lem}\label{lem_ct_oper_prop}
			For any $r \in \nat$, there are $\mathcal{J}_r^{'}(Z, Z') \in \enmr{F'_{0}}$ polynomials in $Z, Z' \in \real^{2n}$ with the same properties as in Theorem \ref{thm_ext_as_exp}, such that for $\mathcal{F}_r^{'} := \mathcal{J}_r^{'} \cdot \mathscr{P}_{n, m}^{\perp}$, the following holds.
			\par 
			There are $\epsilon, c > 0$, $p_1 \in \nat^*$, such that for any $k, l, l' \in \nat$, there is $C > 0$, such that for $p \geq p_1$, $Z, Z' \in T_{y_0}X$, $|Z|, |Z'| \leq \frac{\epsilon}{t} $, $\alpha, \alpha' \in \nat^{2n}$, $|\alpha|+|\alpha'| \leq l$, we have
			\begin{multline}\label{eq_ct_tayl_type}
				\bigg| 
					\frac{\partial^{|\alpha|+|\alpha'|}}{\partial Z^{\alpha} \partial Z'{}^{\alpha'}}
					\bigg(
						\mathcal{C}_t(Z, Z')
						-
						\sum_{r = 0}^{k}
						t^{r}						
						\mathcal{F}_r^{'}(Z, Z') 
					\bigg)
				\bigg|_{\ccal^{l'}(Y)}
				\leq
				C t^{k + 1 - m}
				\cdot
				\\
				\cdot
				\Big(1 + |Z| + |Z'| \Big)^{Q^5_{k, l, l'}}
				\exp\Big(- c \big( |Z_Y - Z'_Y| + |Z_N| + |Z'_N| \big) \Big),	
			\end{multline}
			where $Q^5_{k, l, l'} := 3(2n + k + l' + 4) + l$.
			Also, the following identity holds
			\begin{equation}\label{eq_jopr0_form}
				\mathcal{J}_0^{'}(Z, Z') = {\rm{Id}}_{F_{y_0}}.
			\end{equation}
			Moreover, under the assumptions (\ref{eq_comp_vol_omeg}), we have
			\begin{equation}\label{eq_jopr1_form}
					\mathcal{J}_1^{'}(Z, Z') = {\rm{Id}}_{F_{y_0}} \cdot \pi \cdot 
					\Big(
						 g^{TX}_{y_0} \big(z_N, A(\overline{z}_Y - \overline{z}'_Y) (\overline{z}_Y - \overline{z}'_Y) \big)
						 +
						 g^{TX}_{y_0} \big(\overline{z}'_N, A(z_Y - z'_Y) (z_Y - z'_Y) \big)			 
			 		\Big).
			\end{equation}
		\end{lem}
		\begin{rem}\label{eq_rem_ab_1}
			From (\ref{eq_mathcalc_bouund}), (\ref{eq_ct_tayl_type}), (\ref{eq_jopr0_form}) and the fact that $\mathscr{P}_{n, m}^{\perp}$ is the orthogonal projector by the results of Section \ref{sect_model_calc}, we see that for $a, b$ from Lemma \ref{lem_ct_spec_prop}, we have $a \leq 1 \leq b$.
		\end{rem}
		\begin{proof}
			Recall that polynomials $J_r(Z, Z')$ and $\epsilon, c > 0$ were defined in Proposition \ref{prop_berg_off_diag}.
			From Proposition \ref{prop_berg_off_diag}, (\ref{eq_compa_1}), (\ref{eq_bt_bp_rel_schw}) and (\ref{eq_a_cal_defn}), we conclude that for 
			\begin{equation}\label{eq_jro_jry}
				\mathcal{J}_{r, 0}(Z, Z') := J_r(Z_Y, Z'),
			\end{equation}
			we can define $\mathcal{F}_{r, 0} := \mathcal{J}_{r, 0} \cdot \mathscr{P}_{n, m}^{\perp}$, so that there is $C > 0$, such that for any $|Z|, |Z'| < \frac{\epsilon}{t}$, $p \in \nat^*$, we have
			\begin{multline}\label{eq_at_tayl_type111}
				\bigg| 
					\frac{\partial^{|\alpha|+|\alpha'|}}{\partial Z^{\alpha} \partial Z'{}^{\alpha'}}
					\bigg(
						\mathcal{A}_t(Z, Z')
						-
						\sum_{r = 0}^{k}
						t^{r}						
						\mathcal{F}_{r, 0}(Z, Z') 
					\bigg)
				\bigg|_{\ccal^{l'}(Y)}
				\\
				\leq
				C t^{k + 1 - m}
				\cdot
				\Big(1 + |Z| + |Z'| \Big)^{Q^3_{k, l, l'}}
				\exp\Big(- c \big( |Z_Y - Z'_Y| + |Z_N| + |Z'_N| \big) \Big).		
			\end{multline}
			From Lemma \ref{lem_at_tayl_t_exp}, (\ref{eq_ccal_def_at}) and (\ref{eq_at_tayl_type111}), we conclude that (\ref{eq_ct_tayl_type}) holds for $c := \frac{c}{8}$, $\epsilon := \frac{\epsilon}{4}$ and 
			\begin{equation}\label{eq_jopr01_form_k}
				\mathcal{J}_r^{'}(Z, Z')
				=
				\sum_{r' = 0}^{r}
				\mathcal{K}_{n,m}
				[\overline{\mathcal{J}_{r_0, 0}(Z', Z) }
				,
				\mathcal{J}_{r - r_0, 0}(Z, Z') 
				],
			\end{equation}
			where we borrowed the notation from Lemma \ref{lem_comp_poly}.
			From (\ref{eq_jo_expl_form}), (\ref{eq_jro_jry}) and (\ref{eq_jopr01_form_k}), we deduce (\ref{eq_jopr0_form}).
			From (\ref{eq_k_calc_2}) and (\ref{eq_j1_expl_form}), we deduce 
			\begin{equation}\label{eq_aux_knmj0010}
				\mathcal{K}_{n,m}
				[\overline{\mathcal{J}_{0, 0}(Z', Z) }
				,
				\mathcal{J}_{1, 0}(Z, Z') 
				]
				=
				{\rm{Id}}_{F_{y_0}} \cdot \pi  \cdot
						 g^{TX}_{y_0} \big(\overline{z}'_N, A(z_Y - z'_Y) (z_Y - z'_Y) \big).
			\end{equation}
			Now, remark that the summands from the equation (\ref{eq_jopr01_form_k}), for $r = 1$, are the adjoints of each other.
			From this and (\ref{eq_aux_knmj0010}), we deduce (\ref{eq_jopr1_form}).
		\end{proof}
		\begin{proof}[Proof of Theorem \ref{thm_berg_perp_off_diag}]
			Let us show that Theorem \ref{thm_berg_perp_off_diag} follows from Lemmas \ref{lem_local_logbk_pr}, \ref{prop_der_bound_bt}, \ref{lem_ct_spec_prop},  \ref{lem_ct_bnd_schhw}, \ref{lem_ct_oper_prop}.
			More precisely, from Lemma \ref{lem_local_logbk_pr} and (\ref{eq_bt_bp_rel_schw}), we see that it is enough to establish that there are polynomials  $J_r^{\perp}$, $r \in \nat$ as in Theorem \ref{thm_berg_perp_off_diag}, and $\epsilon, c > 0$, $p_1 \in \nat^*$, such that for $F_r^{\perp}$ as in Theorem \ref{thm_berg_perp_off_diag} and any $k, l, l' \in \nat$, there exists $C > 0$, such that for any $p \geq p_1$, $Z, Z' \in \real^{2n}$, $|Z|, |Z'| \leq \frac{\epsilon}{t}$, $\alpha, \alpha' \in \nat^{2n}$, $|\alpha|+|\alpha'| \leq l$, the following bound holds
			\begin{multline}\label{eq_bt_perp_offdg0}
				\bigg| 
					\frac{\partial^{|\alpha|+|\alpha'|}}{\partial Z^{\alpha} \partial Z'{}^{\alpha'}}
					\bigg(
						\mathcal{B}_t^{\perp}(Z, Z')
						-
						\sum_{r = 0}^{k}
						t^{r}					
						F_r^{\perp}(Z, Z') 
					\bigg)
				\bigg|_{\ccal^{l'}(Y)}
				\\
				\leq
				C t^{k + 1}
				\Big(1 + |Z| + |Z'| \Big)^{Q^2_{k, l, l'}}
				\exp\Big(- c \big( |Z_Y - Z'_Y| + |Z_N| + |Z'_N| \big) \Big).		
			\end{multline}
			We will establish now that there are polynomials  $J_r^{\perp}$, $r \in \nat$, as described above, such that for some constant $C_1 > 0$, the following bound holds
			\begin{equation}\label{eq_bt_perp_offdg}
				\bigg| 
					\frac{\partial^{|\alpha|+|\alpha'|}}{\partial Z^{\alpha} \partial Z'{}^{\alpha'}}
					\bigg(
						\mathcal{B}_t^{\perp}(Z, Z')
						-
						\sum_{r = 0}^{2k + 1}
						t^{r}					
						F_r^{\perp}(Z, Z') 
					\bigg)
				\bigg|_{\ccal^{l'}(Y)}
				\leq
				C_1 t^{2k + 2}
				\Big(1 + |Z| + |Z'| \Big)^{Q''_{2k + 1, l, l'}},
			\end{equation}
			where $Q''_{k, l, l'} := 3(4(n+2)(k+1) + l') + l$.
			Before this, let us show that this will imply (\ref{eq_bt_perp_offdg0}). 
			Remark that from Theorem \ref{thm_logbk_exp_dc}, the bound on the degrees of $J_r^{\perp}$, $r \in \nat$, and the fact that the polynomials $J_r^{\perp}$ depend smoothly on $y_0$, we see that there are $c, C_2 > 0$, $p_1 \in \nat$, such that for $p \geq p_1$, $Z, Z' \in \real^{2n}$, $\alpha, \alpha' \in \nat^{2n}$, $|\alpha|+|\alpha'| \leq l$, the following bound holds
			\begin{multline}\label{eq_exp_btperp_not}
				\bigg| 
					\frac{\partial^{|\alpha|+|\alpha'|}}{\partial Z^{\alpha} \partial Z'{}^{\alpha'}}
					\bigg(
						\mathcal{B}_t^{\perp}(Z, Z')
						-
						\sum_{r = 0}^{2k + 1}
						t^{r}					
						F_r^{\perp}(Z, Z') 
					\bigg)
				\bigg|_{\ccal^{l'}(Y)}
				\\
				\leq
				C_2
				\Big(1 + |Z| + |Z'| \Big)^{6k + l + l' + 6}
				\exp\Big(- c \big( |Z_Y - Z'_Y| + |Z_N| + |Z'_N| \big) \Big).		
			\end{multline}
			From (\ref{eq_bt_perp_offdg}), (\ref{eq_exp_btperp_not}) and Cauchy inequality, we deduce (\ref{eq_bt_perp_offdg0}).
			Hence, it is left to establish (\ref{eq_bt_perp_offdg}), on which we concentrate from now on.
			\par 
			From (\ref{eq_defn_bt_cal}), we see that $\mathcal{B}_t^{\perp}$, $\mathcal{B}_t$ are the only self-adjoint operators on $L^2(g^{TX'}_{0}, h^{F'}_{0})$ such that
			\begin{equation}\label{eq_spec_bt_ct0}
			\begin{aligned}
				&
				\spec{\mathcal{B}_t^{\perp}} \subset \{ 0, 1\}, 
				&&
				(\ker \mathcal{B}_t^{\perp})^{\perp}
				=
				(S_t^{-1} \circ U) \big( H^{0, \perp}_{(2)}(X', L'{}^p \otimes F') \big),
				\\
				&
				\spec{\mathcal{B}_t} \subset \{ 0, 1\},  
				&&(\ker \mathcal{B}_t)^{\perp}
				=
				(S_t^{-1} \circ U) \big( H^{0}_{(2)}(X', L'{}^p \otimes F') \big),
			\end{aligned}
			\end{equation}
			From this and Lemma \ref{lem_ct_spec_prop}, we see that for $a, b$ as in Lemma \ref{lem_ct_spec_prop}, we have
			\begin{equation}\label{eq_spec_bt_ct}
				\mathcal{B}_t^{\perp}
				=
				\int_{\Omega} \frac{1}{\lambda - \mathcal{C}_t} d \lambda,
			\end{equation}
			where $\Omega \subset \comp$ is a circle not containing $\{0\}$ inside of it, but containing $\{a, b\}$. By Remark \ref{eq_rem_ab_1}, we see that $\Omega$ contains $1$ inside of it.
			\par 
			Let us now show the existence of polynomials $J_r^{\perp}(Z, Z') \in \enmr{F_{y_0}}$ as in Theorem \ref{thm_berg_perp_off_diag} and provide an algorithmic way of constructing them.
			We denote by $\mathcal{C}_0$ the operator, acting on $L^2(g^{TX'}_{0}, h^{F'}_{0})$ by the convolution with smooth kernel $\mathscr{P}_{n, m}^{\perp}(Z, Z') \cdot {\rm{Id}}_{F'_0}$.
			By the results of Section \ref{sect_model_calc}, the operator $\mathcal{C}_0$ is an orthogonal projection, hence we have 
			\begin{equation}\label{eq_res_f1}
				\frac{1}{\lambda - \mathcal{C}_0}
				=
				\frac{1- \mathcal{C}_0}{\lambda}
				+
				\frac{\mathcal{C}_0}{\lambda - 1}.
			\end{equation}
			Now, let us apply the resolvent formula
			\begin{equation}\label{eq_res_f}
				\frac{1}{\lambda - \mathcal{C}_t} - \frac{1}{\lambda - \mathcal{C}_0} 
				=
				\frac{1}{\lambda - \mathcal{C}_t} 
				\big(
					\mathcal{C}_t - \mathcal{C}_0
				\big)
				\frac{1}{\lambda - \mathcal{C}_0}.
			\end{equation}
			By using (\ref{eq_res_f1}) and (\ref{eq_res_f}) inductively, we see that for a given $k \in \nat$, we can represent
			\begin{equation}\label{eq_res_f2}
				\frac{1}{\lambda - \mathcal{C}_t}
				=
				\frac{1}{\lambda - \mathcal{C}_t} A_k
				+
				B_k,
			\end{equation}
			where the operators $A_k$ (resp. $B_k$) are the linear combinations with coefficients given by some universal rational functions (in $\lambda$) of the operators of the form
			\begin{equation}\label{eq_res_f3}
				B
				\big(
					\mathcal{C}_t - \mathcal{C}_0
				\big)^{k_1}
				\mathcal{C}_0
				\cdots
				\mathcal{C}_0
				\big(
					\mathcal{C}_t - \mathcal{C}_0
				\big)^{k_l}
				B',
			\end{equation}
			where $l \in \nat$, $k_1, \ldots, k_l \in \nat^*$, $k_1 + \cdots + k_l = k$, (resp. $k_1 + \cdots + k_l \leq k$) and $B$, $B'$ are either the identity operators or $\mathcal{C}_0$.
			Now, from Lemma \ref{lem_ct_spec_prop} and (\ref{eq_spec_bt_ct0}), we deduce that
			\begin{equation}\label{eq_res_f0}
				\frac{1}{\lambda - \mathcal{C}_t}
				=
				\mathcal{B}_t \frac{1}{\lambda - \mathcal{C}_t}
				+
				\frac{1- \mathcal{B}_t}{\lambda}.
			\end{equation}
			We now rewrite (\ref{eq_res_f2}) using (\ref{eq_res_f0}) as follows
			\begin{equation}\label{eq_res_f343}
				\frac{1}{\lambda - \mathcal{C}_t}
				=
				\mathcal{B}_t \frac{1}{\lambda - \mathcal{C}_t} A_k
				+
				\frac{1- \mathcal{B}_t}{\lambda} A_k
				+
				B_k.
			\end{equation}
			\par 
			Now, by Lemmas \ref{lem_bnd_prod_aloc}, \ref{lem_ct_oper_prop}, each of the terms in (\ref{eq_res_f3}) admit Taylor-type expansions as in (\ref{eq_ct_tayl_type}) for $Z, Z' \in \comp^n$, $|Z|, |Z'| < \frac{\epsilon}{2t}$, where $\epsilon$ is as in Lemma \ref{lem_ct_oper_prop}.
			From this and Lemma \ref{lem_at_tayl_t_exp}, we deduce that there are polynomials, having the same properties as described in Theorem \ref{thm_berg_perp_off_diag}, which we denote by $J_r^{\perp}(Z, Z') \in \enmr{F_{y_0}}$, $r \leq k$, such that for $F_r^{\perp}$, defined as in Theorem \ref{thm_berg_perp_off_diag}, there are $\epsilon, c > 0$, $p_1 \in \nat^*$, such that for any $l, l' \in \nat$, there is $C_3 > 0$, such that for $p \geq p_1$, $Z, Z' \in T_{y_0}X$, $|Z|, |Z'| \leq \frac{\epsilon}{2t}$, $\alpha, \alpha' \in \nat^{2n}$, $|\alpha|+|\alpha'| \leq l$, the following bound holds
			\begin{multline}\label{eq_bt_perp_offdg_fin00}
				\bigg| 
					\frac{\partial^{|\alpha|+|\alpha'|}}{\partial Z^{\alpha} \partial Z'{}^{\alpha'}}
					\bigg(
						\Big(
						\int_{\Omega} B_k d \lambda
						\Big)(Z, Z')
						-
						\sum_{r = 0}^{k}
						t^{r}					
						F_r^{\perp}(Z, Z') 
					\bigg)
				\bigg|_{\ccal^{l'}(Y)}
				\\
				\qquad \qquad \qquad \qquad \qquad  \qquad \qquad \qquad  \qquad \qquad 
				\leq
				C_3 t^{k + 1}
				\Big(1 + |Z| + |Z'| \Big)^{Q''_{k, l, l'}}.
			\end{multline}
			Also, since by Lemma \ref{lem_ct_oper_prop}, the first term of the Taylor-type expansion of $\mathcal{C}_t - \mathcal{C}_0$ vanishes, and each term in $A_k$ has exactly $k$ multiplicands  $\mathcal{C}_t - \mathcal{C}_0$, by Lemma \ref{lem_bnd_prod_aloc}, we conclude that there is $c > 0$, $p_1 \in \nat^*$, such that for any $k, l, l' \in \nat$, there is $C_4 > 0$, such that for $p \geq p_1$, $Z, Z' \in T_{y_0}X$, $|Z|, |Z'| \leq \frac{\epsilon}{2t}$, $\alpha, \alpha' \in \nat^{2n}$, $|\alpha|+|\alpha'| \leq l$, the following bound holds
			\begin{equation}\label{eq_bt_perp_offdg_fin002}
				\bigg| 
					\frac{\partial^{|\alpha|+|\alpha'|}}{\partial Z^{\alpha} \partial Z'{}^{\alpha'}}
					\bigg(
						\Big(\int_{\Omega}  \frac{1- \mathcal{B}_t}{\lambda} A_k  d \lambda \Big)
						(Z, Z')
					\bigg)
				\bigg|_{\ccal^{l'}(Y)}
				\leq
				C_4 t^k
				\Big(1 + |Z| + |Z'| \Big)^{Q''_{k, l, l'}}.
			\end{equation}
			\par 
			We will now show that the first summand on the right-hand side of (\ref{eq_res_f343}) is bounded by the term in the right-hand side of (\ref{eq_bt_perp_offdg}).
			For this, let us fix $Z' \in \comp^n$, $|Z'| < \frac{\epsilon}{2t}$.
			From Lemma \ref{lem_bnd_prod_aloc}, similarly to (\ref{eq_bt_perp_offdg_fin002}), we deduce that there is $C_5 > 0$, such that we have
			\begin{equation}\label{eq_bnd_prod_aloc111}
				\Big| \frac{\partial^{|\alpha'|}}{\partial Z'{}^{\alpha}} A_k(Z, Z') \Big|  \leq C_5 t^k (1 + |Z| + |Z'|)^{Q''_{k, l, l'}} \exp\Big(- \frac{c}{8} \big( |Z_Y - Z'_Y| + |Z_N| + |Z'_N| \big) \Big).
			\end{equation}
			From (\ref{eq_bnd_intlocal}) and  (\ref{eq_bnd_prod_aloc111}), we readily deduce that there is $C_6 > 0$, such that
			\begin{equation}\label{eq_bnd_prod_aloc222}
				\Big\| \frac{\partial^{|\alpha'|}}{\partial Z'{}^{\alpha'}} A_k(\cdot, Z') \Big\|_{L^2(g^{TX'}_{0})}
				\leq 
				C_6 t^k (1 + |Z'|)^{Q''_{k, l, l'}}.		
			\end{equation}
			However, by the choice of $\Omega$, there exists $c > 0$, such that for any $\lambda \in \Omega$, we have
			\begin{equation}\label{eq_spec_bnd_rsl}
				\spec \Big( \mathcal{B}_t \frac{1}{\lambda - \mathcal{C}_t} \Big) \subset \mathbb{B}_0^{\comp}(c).
			\end{equation}						
			From (\ref{eq_bnd_prod_aloc222}) and (\ref{eq_spec_bnd_rsl}), we deduce that there is $C_7 > 0$, such that
			\begin{equation}\label{eq_spec_bnd_rsl221}
				\Big\| \frac{\partial^{|\alpha'|}}{\partial Z'{}^{\alpha'}} \Big( \mathcal{B}_t \frac{1}{\lambda - \mathcal{C}_t} A_k \Big) (\cdot, Z') \Big\|_{L^2(g^{TX'}_{0})}
				\leq 
				C_7 t^k (1 + |Z'|)^{Q''_{k, l, l'}}.		
			\end{equation}
			From Lemma \ref{prop_der_bound_bt} and (\ref{eq_spec_bnd_rsl221}), we deduce that there is $C_8 > 0$, such that for any $Z, Z' \in \comp^n$, $|Z|, |Z'| < \frac{\epsilon}{2t}$, we have
			\begin{equation}\label{eq_spec_bnd_rsl2213}
				\Big|
				\frac{\partial^{|\alpha| + |\alpha'|}}{\partial Z{}^{\alpha} \partial Z'{}^{\alpha'}} \Big( \mathcal{B}_t \frac{1}{\lambda - \mathcal{C}_t} A_k \Big) (Z, Z')
				\Big|
				\leq
				C_8 t^k (1 + |Z'|)^{Q''_{k, l, l'}}.		
			\end{equation}
			From (\ref{eq_spec_bt_ct}), (\ref{eq_res_f2}), (\ref{eq_res_f343}), (\ref{eq_bt_perp_offdg_fin00}), (\ref{eq_bt_perp_offdg_fin002}) and (\ref{eq_spec_bnd_rsl2213}), we deduce (\ref{eq_bt_perp_offdg}).
			Also, from (\ref{eq_bt_perp_offdg_fin00}), we deduce the general algorithm for the construction of the polynomials $J_r^{\perp}$.
			\par 
			Now it only left to prove (\ref{eq_jopep_0}) and (\ref{eq_jopep_1}).
			For this, let us find explicit formula for $B_1$, as the polynomials $J_0^{\perp}$, $J_1^{\perp}$, can be then read off from (\ref{eq_bt_perp_offdg_fin00}).
			For this, we apply once (\ref{eq_res_f0}) and then twice (\ref{eq_res_f}) to get			
			\begin{equation}\label{eq_res_fb1}
				\frac{1}{\lambda - \mathcal{C}_t} =  
				\frac{1}{\lambda - \mathcal{C}_0} 
				+
				\frac{1}{\lambda - \mathcal{C}_0} 
				\big(
					\mathcal{C}_t - \mathcal{C}_0
				\big)
				\frac{1}{\lambda - \mathcal{C}_0}
				+
				\frac{1}{\lambda - \mathcal{C}_t} 
				\big(
					\mathcal{C}_t - \mathcal{C}_0
				\big)
				\frac{1}{\lambda - \mathcal{C}_0}
				\big(
					\mathcal{C}_t - \mathcal{C}_0
				\big)
				\frac{1}{\lambda - \mathcal{C}_0}.
			\end{equation}
			From (\ref{eq_res_f1}), (\ref{eq_res_f2}) and (\ref{eq_res_fb1}), we deduce that
			\begin{equation}\label{eq_b1_nonint_form}
				B_1 
				=
				\Big(
					\frac{1- \mathcal{C}_0}{\lambda}
					+
					\frac{\mathcal{C}_0}{\lambda - 1}
				\Big)
				+
				\Big(
					\frac{1- \mathcal{C}_0}{\lambda}
					+
					\frac{\mathcal{C}_0}{\lambda - 1}
				\Big)
				\big(
					\mathcal{C}_t - \mathcal{C}_0
				\big)
				\Big(
					\frac{1- \mathcal{C}_0}{\lambda}
					+
					\frac{\mathcal{C}_0}{\lambda - 1}
				\Big).
			\end{equation}
			Now, by Remark \ref{eq_rem_ab_1} and the choice of $\Omega$, we have the identities
			\begin{equation}\label{eq_ev_int_dlam}
			\begin{aligned}
				&
				\int_{\Omega} \frac{1}{\lambda} d\lambda = 0, 
				&&
				\int_{\Omega} \frac{1}{\lambda - 1} d\lambda = 1, 
				&&&
				\int_{\Omega} \frac{1}{\lambda^2} d\lambda = 0, 
				\\
				&
				\int_{\Omega} \frac{1}{(\lambda - 1)^2} d\lambda = 0, 
				&&
				\int_{\Omega} \frac{1}{\lambda (\lambda - 1)} d\lambda = 1.
				&&&
			\end{aligned}
			\end{equation}
			From (\ref{eq_b1_nonint_form}) and (\ref{eq_ev_int_dlam}), we deduce
			\begin{equation}\label{eq_b1_int_form}
				\int_{\Omega}
				B_1
				d \lambda 
				=
				\mathcal{C}_0
				+
				(1- \mathcal{C}_0)
				(
					\mathcal{C}_t - \mathcal{C}_0
				)
				\mathcal{C}_0
				+
				\mathcal{C}_0
				(
					\mathcal{C}_t - \mathcal{C}_0
				)
				(1- \mathcal{C}_0).
			\end{equation}
			Hence, by using the fact that $\mathcal{C}_0^{2} = \mathcal{C}_0$, the fact that  by Lemma \ref{lem_ct_oper_prop}, the first term of the Taylor-type expansion of $\mathcal{C}_t - \mathcal{C}_0$ vanishes, and (\ref{eq_b1_int_form}), we deduce (\ref{eq_jopep_0}).
			\par Now, let us establish (\ref{eq_jopep_1}).
			By using the notations from Proposition \ref{prop_berg_off_diag}, Remark \ref{rem_k_calculc} and (\ref{eq_jopr1_form}), we deduce from (\ref{eq_bt_perp_offdg_fin00}) and (\ref{eq_b1_int_form}) the following identity
			\begin{equation}\label{eq_j1_perp_kkpr_expr}
				J_1^{\perp}(Z, Z')
				=
				\mathcal{K}_{n, m}
				\big[
					\mathcal{J}_1^{'},
					\mathcal{J}_0^{'}
				\big]
				-
				2
				\mathcal{K}_{n, m}
				\Big[
					\mathcal{J}_0^{'}, 
					\mathcal{K}_{n, m}
					\big[
						\mathcal{J}_1^{'},
						\mathcal{J}_0^{'}
					\big]
				\Big]
				+
				\mathcal{K}_{n, m}
				\big[
					\mathcal{J}_0^{'}, 
						\mathcal{J}_1^{'}
				\big].
			\end{equation}
			We will now calculate each term in (\ref{eq_j1_perp_kkpr_expr}).
			From (\ref{eq_k_calc_2}), (\ref{eq_k_calc_3}) and (\ref{eq_jopr1_form}), we deduce
			\begin{equation}\label{eq_calc_lbkk_8}
					\mathcal{K}_{n, m}
					\big[
						\mathcal{J}_1^{'},
						\mathcal{J}_0^{'}
					\big](Z, Z') = {\rm{Id}}_{F_{y_0}} \cdot \pi 
					 g \big(z_N, A(\overline{z}_Y - \overline{z}'_Y) (\overline{z}_Y - \overline{z}'_Y) \big).
			\end{equation}
			From (\ref{eq_k_calc_3}) and (\ref{eq_calc_lbkk_8}), we deduce
			\begin{equation}\label{eq_calc_lbkk_11}
					\mathcal{K}_{n, m}
					\Big[
					\mathcal{J}_0^{'}, 
					\mathcal{K}_{n, m}
					\big[
						\mathcal{J}_1^{'},
						\mathcal{J}_0^{'}
					\big]
					\Big]
					(Z, Z') = 0.
			\end{equation}
			Now, (\ref{eq_jopep_1}) follows directly from  (\ref{eq_j1_perp_kkpr_expr}), (\ref{eq_calc_lbkk_8}) and (\ref{eq_calc_lbkk_11}).
		\end{proof}

	\subsection{Asymptotics of the extension, proofs of Theorems \ref{thm_high_term_ext}, \ref{thm_ext_exp_dc}, \ref{thm_ext_as_exp}, \ref{thm_opt_linf_bnd}}\label{sect_exp_dec_ot}
	The main goal of this section is to prove Theorem \ref{thm_high_term_ext}, Corollary \ref{cor_ot_higher} and Theorems \ref{thm_ext_exp_dc}, \ref{thm_ext_as_exp}, \ref{thm_opt_linf_bnd}.
	Theorem \ref{thm_high_term_ext}, Corollary \ref{cor_ot_higher} and Theorem \ref{thm_opt_linf_bnd} will follow almost directly from Theorems \ref{thm_ext_exp_dc}, \ref{thm_ext_as_exp}.
	The main idea for the proof of Theorems \ref{thm_ext_exp_dc}, \ref{thm_ext_as_exp}, is to find an algebraic expression for $\ext_p$ in terms of $B_p^{\perp}$, $B_p^Y$ and $\res_Y$ and then to get the needed results by transferring the statements from Theorems \ref{thm_logbk_exp_dc}, \ref{thm_berg_perp_off_diag}.
	\par 
	We conserve the notations from Section \ref{sect_intro}.
	Define the operators $G_p : L^2(Y, \iota^*( L^p \otimes F)) \to L^2(Y, \iota^*( L^p \otimes F))$ and $I_p : L^2(Y, \iota^*( L^p \otimes F)) \to L^2(X, L^p \otimes F)$ as follows
	\begin{equation}\label{eq_gp_defn}
		G_p := 
		\res_Y \circ I_p - B_p^{Y},
		\qquad
		I_p := B_p^{\perp} \circ \ext_p^{0} \circ B_p^{Y}.
	\end{equation}
	The formula for $\ext_p$ in terms of $B_p^{\perp}$, $B_p^Y$ and $\res_Y$ is based on the infinite summation with the operator $G_p$.
	To get a grip on this infinite sum, we need to study the properties of $G_p$.
	\begin{lem}\label{lem_gp_dec}
		There are $c > 0$, $p_1 \in \nat^*$, such that for any $k \in \nat$, there is $C > 0$, such that for any $p \geq p_1$, $y_1, y_2 \in Y$, the following estimate holds
		\begin{equation}\label{eq_gp_dec}
			\Big|  G_p(y_1, y_2) \Big|_{\ccal^k(Y \times Y)} \leq C p^{m + \frac{k - 1}{2}} \cdot \exp \big(- c \sqrt{p} \cdot \dist_X(y_1, y_2) \big),
		\end{equation}
		where the pointwise $\ccal^{k}$-norm is interpreted as in Theorem \ref{thm_ext_exp_dc}.
		\par 
		Also, for any $r \in \nat$, $y_0 \in Y$, there are $J_{r, G}(Z_Y, Z'_Y) \in \enmr{F_{y_0}}$ polynomials in $Z_Y, Z'_Y \in \real^{2m}$, satisfying the same properties as in Theorem \ref{thm_ext_as_exp}, such that for $F_{r, G} := J_{r, G} \cdot \mathscr{P}_m$, the following holds.
		There are $\epsilon, c > 0$, $p_1 \in \nat^*$,  such that for any $k, l, l' \in \nat$, there is $C > 0$, such that for any $y_0 \in Y$, $p \geq p_1$, $Z_Y, Z'_Y \in \real^{2m}$, $|Z_Y|, |Z'_Y| \leq \epsilon$, $\alpha, \alpha' \in \nat^{2m}$, $|\alpha|+|\alpha'| \leq l$, we have
		\begin{multline}\label{eq_apl_ass_tay_gp}
			\bigg| 
				\frac{\partial^{|\alpha|+|\alpha'|}}{\partial Z_Y^{\alpha} \partial Z'_Y{}^{\alpha'}}
				\bigg(
					\frac{1}{p^m} G_p\big(\psi_{y_0}(Z_Y), \psi_{y_0}(Z'_Y) \big)
					-
					\sum_{r = 0}^{k}
					p^{-\frac{r}{2}}						
					F_{r, G}(\sqrt{p} Z_Y, \sqrt{p} Z'_Y) 
					\kappa_{Y}^{-\frac{1}{2}}(Z_Y)
					\kappa_{Y}^{-\frac{1}{2}}(Z'_Y)
				\bigg)
			\bigg|_{\ccal^{l'}(Y)}
			\\
			\leq
			C p^{- \frac{k + 1 - l}{2}}			
			\Big(1 + \sqrt{p}|Z_Y| + \sqrt{p} |Z'_Y| \Big)^{2Q^2_{k, l, l'}}
			\exp\Big(- c \sqrt{p} |Z_Y - Z'_Y| \Big),
		\end{multline}
		where $Q^2_{k, l, l'} > 0$ is defined in Theorem \ref{thm_berg_perp_off_diag}.
		Moreover, we have
		\begin{equation}\label{eq_j01g_nul0}
			J_{0, G} = 0.
		\end{equation}
		And under assumption (\ref{eq_comp_vol_omeg}), we even have
		\begin{equation}\label{eq_j01g_nul}
			J_{1, G} = 0,
		\end{equation}
		so that $p^{m + \frac{k - 1}{2}}$ in (\ref{eq_gp_dec}) in this case can be replaced by $p^{m - 1 + \frac{k}{2}}$.
	\end{lem}
	\begin{proof}
		First of all, by Lemma \ref{lem_bnd_prod_a} and Theorem \ref{thm_bk_off_diag}, there are $c > 0$, $p_1 \in \nat^*$, such that for any $k \in \nat$, there is $C > 0$, such that for $p \geq p_1$, $y_1, y_2 \in Y$, the following estimate holds
		\begin{equation}\label{eq_gp_dec1}
			\Big|  G_p(y_1, y_2) \Big|_{\ccal^k(Y \times Y)} \leq C p^{n + \frac{k}{2}} \cdot \exp \big(- c \sqrt{p} \cdot \dist_X(y_1, y_2) \big).
		\end{equation}
		Now (\ref{eq_gp_dec1}) implies (\ref{eq_gp_dec}) for $y_1, y_2 \in Y$ verifying $\dist_X(y_1, y_2) > R$, where $R$ is defined in (\ref{eq_r_defn_const}) (and for a different choice of $c, C$). Hence, it is enough to establish (\ref{eq_apl_ass_tay_gp}), (\ref{eq_j01g_nul0}) to get (\ref{eq_gp_dec}) for all $y_1, y_2 \in Y$.
		Let us concentrate on this now.
		\par 
		Let $J'_{r, Y}(Z_Y, Z'_Y) \in \enmr{F_{y_0}}$ be the polynomials in $Z_Y, Z'_Y \in \real^{2m}$, $r \in \nat$, given by Theorem \ref{thm_berg_dailiuma}, applied for $X := Y$.
		Recall that polynomials $J_r^{\perp}(Z, Z') \in \enmr{F_{y_0}}$, were defined in Theorem \ref{thm_berg_perp_off_diag}.
		We rewrite $I_p$ in the following equivalent form
		\begin{equation}
			I_p := \big( B_p^{\perp} \cdot \kappa_N^{\frac{1}{2}} \big) \circ \big( \kappa_N^{- \frac{1}{2}} \cdot \ext_p^{0} \big) \circ B_p^{Y}.
		\end{equation}
		For $k \in \nat$, let us now write the Taylor expansions of $\kappa_N^{\frac{1}{2}}$, $\kappa_N^{- \frac{1}{2}}$ in a neighborhood of $y_0$ as follows 
		\begin{equation}
			\kappa_N^{\frac{1}{2}}(Z) = \sum_{i = 0}^{k} \kappa_{N, [i]}^{\frac{1}{2}}(Z) + O(|Z|^{k + 1}),
			\qquad 
			\kappa_N^{-\frac{1}{2}}(Z) = \sum_{i = 0}^{k} \kappa_{N, [i]}^{-\frac{1}{2}}(Z) + O(|Z|^{k + 1}),
		\end{equation}
		where $\kappa_{N, [i]}^{\frac{1}{2}}(Z)$, $\kappa_{N, [i]}^{-\frac{1}{2}}(Z)$ are homogeneous polynomials of degree $i$.
		We denote now
		\begin{equation}
			J_{r}^{\perp, \kappa}
			:=
			\sum_{r_0 = 0}^{r} J_{r_0}^{\perp}(Z, Z') \cdot \kappa_{N, [r - r_0]}^{\frac{1}{2}}(Z').
		\end{equation}
		By Theorems \ref{thm_logbk_exp_dc}, \ref{thm_berg_perp_off_diag}, Lemma \ref{lem_ac_ext_op_tay} and Theorems \ref{thm_bk_off_diag}, \ref{thm_berg_dailiuma}, we conclude that for 
		\begin{equation}\label{eq_jogo_form}
			J_{r, I}^{E}(Z, Z'_Y) 
			=
			\sum_{r_0 = 0}^{r}
			\mathcal{K''}_{n, m}[J_{r_0}^{\perp, \kappa}, J'_{r - r_0, Y}]
		\end{equation}
		 and for $F_{r, I}^{E} := J_{r, I}^{E} \cdot \mathscr{E}_{n, m}$, the following holds.
		There are $\epsilon, c > 0$, $p_1 \in \nat^*$, such that for any $k, l, l' \in \nat$, there is $C  > 0$, such that for any $p \geq p_1$, $y_0 \in Y$, $Z = (Z_Y, Z_N)$, $Z_Y, Z'_Y \in \real^{2m}$, $Z_N \in \real^{2(n - m)}$, $|Z|, |Z'_Y| \leq \epsilon$, $\alpha \in \nat^{2n}$, $\alpha' \in \nat^{2m}$, $|\alpha| + |\alpha'| \leq l$, we have
		\begin{multline}\label{eq_ext_as_exp_I_op}
			\bigg| 
				\frac{\partial^{|\alpha|+|\alpha'|}}{\partial Z^{\alpha} \partial Z'_Y{}^{\alpha'}}
				\bigg(
					\frac{1}{p^m}  I_p \big(\psi_{y_0}(Z), \psi_{y_0}(Z'_Y) \big)
					-
					\sum_{r = 0}^{k}
					p^{-\frac{r}{2}}						
					F_{r, I}^{E}(\sqrt{p} Z, \sqrt{p} Z'_Y) 
					\kappa_{X}^{-\frac{1}{2}}(Z)
					\kappa_{Y}^{-\frac{1}{2}}(Z'_Y)
				\bigg)
			\bigg|_{\ccal^{l'}(Y)}
			\\
			\leq
			C p^{- \frac{k + 1 - l}{2}}
			\Big(1 + \sqrt{p}|Z| + \sqrt{p} |Z'_Y| \Big)^{2Q^2_{k, l, l'}}
			\exp\Big(- c \sqrt{p} \big( |Z_Y - Z'_Y| + |Z_N| \big) \Big).
		\end{multline} 
		From Theorem \ref{thm_berg_dailiuma}, (\ref{eq_gp_defn}) and (\ref{eq_ext_as_exp_I_op}), we conclude that (\ref{eq_apl_ass_tay_gp}) holds for
		\begin{equation}\label{eq_jrg_calc}
			J_{r, G}(Z_Y, Z'_Y) 
			:= 
			\sum_{r_0 = 0}^{r}
			\kappa_{N, [r_0]}^{-\frac{1}{2}}(Z_Y)
			\cdot
			J^E_{r - r_0, I}(Z_Y, Z'_Y) 
			-
			J'_{r, Y}(Z_Y, Z'_Y).
		\end{equation}
		From (\ref{eq_jopep_0}), (\ref{eq_jopr_calc}) and  (\ref{eq_jogo_form}), we conclude
		\begin{equation}\label{eq_jogo_form1}
			J_{0, I}^{E}
			=
			{\rm{Id}}_{F_{y_0}}
			\cdot
			\kappa_{N}^{\frac{1}{2}}(y_0).
		\end{equation}
		From (\ref{eq_jopr_calc}), (\ref{eq_jrg_calc}) and (\ref{eq_jogo_form1}), we obtain (\ref{eq_j01g_nul0}).
		\par
		Now, we assume (\ref{eq_comp_vol_omeg}).
		From (\ref{eq_jopep_1}), (\ref{eq_kmprpr_form}), (\ref{eq_k_calc_1}), (\ref{eq_k_calc_3}), (\ref{eq_nu_zero}), (\ref{eq_j1pr_calc}), (\ref{eq_kappa_t}), and (\ref{eq_jogo_form}), we get
		\begin{equation}\label{eq_jogo_form2}
			J_{1, I}^{E}(Z, Z'_Y)
			=
			 {\rm{Id}}_{F_{y_0}} \cdot \pi 
					 g^{TX}_{y_0} \big(z_N, A(\overline{z}_Y - \overline{z}'_Y) (\overline{z}_Y - \overline{z}'_Y) \big).
		\end{equation}
		We obtain (\ref{eq_j01g_nul}) from (\ref{eq_jrg_calc}) and (\ref{eq_jogo_form2}).
	\end{proof}
	\begin{lem}\label{lem_norm_bnd}
		There are $C > 0$, $p_1 \in \nat^*$, such that for any $p \geq p_1$, we have
		\begin{equation}
			\norm{G_p} \leq \frac{C}{\sqrt{p}}.
		\end{equation}
		Moreover, under assumption (\ref{eq_comp_vol_omeg}), in the above inequality, we can replace $\sqrt{p}$ by $p$.
	\end{lem}
	\begin{proof}
		First of all, by Corollary \ref{cor_exp_bound_int} and Lemma \ref{lem_gp_dec}, there are $C > 0$, $p_1 \in \nat^*$, such that for any $y_0 \in Y$, for $p \geq p_1$, we have
		\begin{equation}\label{eq_bound_int_berg_gp}
			\int_{Y} \big| G_p(y_0, y)  \big|  dv_Y(y) \leq \frac{C}{\sqrt{p}}, \qquad \int_{Y} \big| G_p(y, y_0)  \big|  dv_Y(y) \leq \frac{C}{\sqrt{p}}.
		\end{equation}
		The result now follows from (\ref{eq_bound_int_berg_gp}) and Young's inequality for integral operators, cf. \cite[Theorem 0.3.1]{SoggBook} applied for $p, q = 2$, $r = 1$ in the notations of \cite{SoggBook}.
		The second part is proved using (\ref{eq_j01g_nul}) in exactly the same way, one only has to rely on the comment after (\ref{eq_j01g_nul}).
	\end{proof}
	\begin{sloppypar}
	Lemma \ref{lem_norm_bnd} now implies that there is $p_1 \in \nat^*$, such that for $p \geq p_1$, the infinite sum 
	\begin{equation}
			T_p := \sum_{i = 1}^{\infty} (-1)^i G^i_p.
	\end{equation}
	is well-defined as an operator on $L^2(Y, \iota^*( L^p \otimes F))$.
	\end{sloppypar}
	\begin{lem}\label{lem_ext_op_inf_sum}
		The following identity holds
		\begin{equation}\label{eq_ext_op_inf_sum}
			\ext_p
			=
			I_p
			+
			I_p
			\circ
			T_p
			.
		\end{equation}
	\end{lem}
	\begin{proof}
		Since $B_p^{\perp}$ has values in $H^{0, \perp}_{(2)}(X, L^p \otimes F)$, and $G_p$ vanishes on the kernel of $B_p^{Y}$ by (\ref{eq_gp_defn}), it is enough to establish that
	\begin{equation}
		\res_Y
		\circ
		\big(I_p
		+
		I_p
		\circ
		T_p
		\big)
		=
		B_p^{Y}.
	\end{equation}
	This identity follows from the observation that for any $i \in \nat^*$, we have
	\begin{equation}
		\res_Y
		\circ
		B_p^{\perp}
		\circ
		\ext_p^{0} 
		\circ
		G_p^{i}
		=
		\res_Y
		\circ
		I_p
		\circ
		G_p^{i}
		,
		\qquad
		\res_Y
		\circ
		B_p^{\perp}
		\circ
		\ext_p^{0} 
		\circ
		G_p^{i}
		=
		G_p^{i+1}
		+
		G_p^{i},
	\end{equation}
	which, on its turn, follows by an application of the right composition with $G_p^i$ in (\ref{eq_gp_defn}),
	and the identity $B_p^{Y} \circ G_p = G_p$, following from (\ref{eq_res_op_l2_form}).
	\end{proof}
	\begin{proof}[Proof of Theorem \ref{thm_ext_exp_dc}]
		By Lemmas \ref{lem_bnd_prod_a}, \ref{lem_gp_dec}, and the fact that for any $C > 0$ there is $p_1 \in \nat^*$ such that for $p \geq p_1$, the sum $\sum_{i = 0}^{\infty} (\frac{C}{\sqrt{p}})^i$ converges, we deduce that there is $c > 0$, such that for any $k \in \nat$, there is $C > 0$, such that for $p \geq p_1$, the following estimate for the Schwartz kernel holds
		\begin{equation}\label{eq_gp_dec_inf_sum}
			\Big|  T_p(y_1, y_2) \Big|_{\ccal^k(Y \times Y)} \leq C p^{m + \frac{k - 1}{2}} \cdot \exp \big(- c \sqrt{p} \cdot \dist_X(y_1, y_2) \big).
		\end{equation}
		Moreover, under assumption (\ref{eq_comp_vol_omeg}), we can replace $p^{m + \frac{k - 1}{2}}$ by $p^{m - 1 + \frac{k}{2}}$.
		We conclude by Theorem \ref{thm_logbk_exp_dc}, Lemmas \ref{lem_bnd_prod_a}, \ref{lem_ext_op_inf_sum} and (\ref{eq_gp_dec_inf_sum}).
	\end{proof}
	\begin{proof}[Proof of Theorem \ref{thm_ext_as_exp}]
		Let us fix $k \in \nat$ and establish Theorem \ref{thm_ext_as_exp} for it.
		By Lemmas \ref{lem_apl_tayl_exp}, \ref{lem_gp_dec}, we conclude that  for any $r \in \nat$, $y_0 \in Y$, there are $J_{r, T}(Z_Y, Z'_Y) \in \enmr{F_{y_0}}$ polynomials in $Z_Y, Z'_Y \in \real^{2m}$, satisfying the same properties as in Theorem \ref{thm_ext_as_exp}, such that for $F_{r, T} := J_{r, T} \cdot \mathscr{P}_m$, the following holds.
		There are $\epsilon, c > 0$, $p_1 \in \nat^*$,  such that for any $k, l, l' \in \nat$, there is $C > 0$, such that for any $y_0 \in Y$, $p \geq p_1$, $Z_Y, Z'_Y \in \real^{2m}$, $|Z_Y|, |Z'_Y| \leq \epsilon$, $\alpha, \alpha' \in \nat^{2m}$, $|\alpha|+|\alpha'| \leq l$, the following bound holds
		\begin{multline}\label{eq_apl_ass_tay_tp}
			\bigg| 
				\frac{\partial^{|\alpha|+|\alpha'|}}{\partial Z_Y^{\alpha} \partial Z'_Y{}^{\alpha'}}
				\bigg(
					\frac{1}{p^m} T_p\big(\psi_{y_0}(Z_Y), \psi_{y_0}(Z'_Y) \big)
					-
					\sum_{r = 0}^{k}
					p^{-\frac{r}{2}}						
					F_{r, T}(\sqrt{p} Z_Y, \sqrt{p} Z'_Y) 
					\kappa_{Y}^{-\frac{1}{2}}(Z_Y)
					\kappa_{Y}^{-\frac{1}{2}}(Z'_Y)
				\bigg)
			\bigg|_{\ccal^{l'}(Y)}
			\\
			\leq
			C p^{- \frac{k + 1 - l}{2}}			
			\Big(1 + \sqrt{p}|Z_Y| + \sqrt{p} |Z'_Y| \Big)^{Q''_{k, l, l'}}
			\exp\Big(- c \sqrt{p} |Z_Y - Z'_Y| \Big),
		\end{multline}
		where $Q'''_{k, l, l'} := 6(8(n+2)(2k+1) +  l') + 2l$.
		Moreover, we have
		\begin{equation}\label{eq_j01t_nul0}
			J_{0, T} = 0.
		\end{equation}
		And under assumption (\ref{eq_comp_vol_omeg}), we even have
		\begin{equation}\label{eq_j01t_nul}
			J_{1, T} = 0.
		\end{equation}
		Now, by Lemma \ref{lem_ac_ext_op_tay}, (\ref{eq_ext_as_exp_I_op}) and (\ref{eq_apl_ass_tay_tp}), we deduce the asymptotic expansion (\ref{eq_ext_as_exp}) for 
		\begin{equation}\label{eq_jre_form_comm_e}
			J_r^E 
			:=
			J_{r, I}^{E}
			+
			\sum_{r_1 = 0}^{r}
			\mathcal{K'''}_{n, m}[J_{r_1, I}^{E}, J_{r - r_1, T}].
		\end{equation}
		From (\ref{eq_jogo_form1}), (\ref{eq_j01t_nul0}) and (\ref{eq_jre_form_comm_e}), we deduce (\ref{eq_je0_exp}).
		Moreover, under assumption (\ref{eq_comp_vol_omeg}), from  (\ref{eq_jogo_form2}), (\ref{eq_j01t_nul0}), (\ref{eq_j01t_nul}) and (\ref{eq_jre_form_comm_e}), we deduce (\ref{eq_je1_exp}).
	\end{proof}
	\begin{comment}
	\begin{rem}
		For Theorems \ref{thm_ext_as_exp}, \ref{thm_berg_perp_off_diag}, one can give an alternative proof, relying only on the relation between $\ext_p$ and $\res_p^*$ instead of Lemma \ref{lem_ext_op_inf_sum}.
		The details will appear in the subsequent paper \cite{FinToeplImm}.
		It seems, however, that such a proof doesn't generalize to study the extension operators associated to sections vanishing up to an order which grows with the exponent of the line bundle.
	\end{rem}
	\end{comment}
	\begin{proof}[Proof of Theorem \ref{thm_high_term_ext}]
		The main idea of the proof is to compare the Schwartz kernels of $\ext_{p}$ and $\ext_{p}^{0}$.
		Let us denote $K_p := \ext_p - \ext_p^{0}$. 
		From (\ref{eq_ext0_op}), remark that the Schwartz kernel, $\ext_{p}^{0}(x, y)$, $x = (y', Z_N) \in \mathbb{B}_Y^X(r_{\perp})$, $y \in Y$, of $\ext_{p}^{0}$, evaluated with respect to $dv_Y$, equals to
		\begin{equation}\label{eq_bergm_triv_extension}
			\ext_{p}^{0}(x, y)
			=
			\rho \Big(\frac{|Z_N|}{r_{\perp}} \Big)
			\cdot
			\exp \Big(- p \frac{\pi}{2} |Z_N|^2 \Big)
			\cdot
			B_{p}^Y (y', y).
		\end{equation}
		From this and Theorem \ref{thm_bk_off_diag}, we conclude that there are $c_1, C_1 > 0$, $p_1 \in \nat^*$, such that for any $p \geq p_1$, $x \in X$, $y \in Y$, the following estimate holds
		\begin{equation}\label{eq_ext_kp}
			\Big|  \ext_{p}^{0}(x, y) \Big| \leq C_1 p^{m} \exp \big(- c_1 \sqrt{p} \dist(x, y) \big).
		\end{equation}
		From Theorem \ref{thm_ext_exp_dc} and (\ref{eq_ext_kp}), we conclude that there are $c_2, C_2 > 0$, $p_1 \in \nat^*$,  such that for any $p \geq p_1$, $x \in X$, $y \in Y$, the following estimate holds
		\begin{equation}\label{eq_ext_kp}
			\Big|  K_p(x, y) \Big| \leq C_2 p^{m + \frac{k}{2}} \exp \big(- c_2 \sqrt{p} \dist(x, y) \big).
		\end{equation}
		Let us now denote by $J^{\phi}_{r, Y}(Z_Y, Z'_Y)$, $r \in \nat$, the polynomials from Theorem \ref{thm_berg_dailiuma}, applied for $X := Y$.
		From Theorems \ref{thm_ext_as_exp}, \ref{thm_berg_dailiuma}, we deduce that for polynomials 
		\begin{equation}\label{eq_jrke}
			J_{r, K}^E(Z, Z'_Y) := J_r^E(Z, Z'_Y) -  \sum_{r_0 = 0}^{r}  \kappa_{N, [r - r_0]}^{\frac{1}{2}}(Z_Y) \cdot J^{\phi}_{r_0, Y}(Z_Y, Z'_Y), \quad r \in \nat,
		\end{equation}
		and the functions $F_{r, K}^E := J_{r, K}^E \cdot \mathscr{E}_{n, m}$ over $\real^{2n} \times \real^{2m}$, the following holds.
		There are $\epsilon, c > 0$, $p_1 \in \nat^*$, such that for any $k \in \nat$, there is $C_3 > 0$, such that for any $y_0 \in Y$, $p \geq p_1$, $Z = (Z_Y, Z_N)$, $Z_Y, Z'_Y \in \real^{2m}$, $Z_N \in \real^{2(n - m)}$, $|Z|, |Z'_Y| \leq \epsilon$, we have
		\begin{multline}\label{eq_kp_as_exp}
			\bigg| 
					\frac{1}{p^m} K_p \big(\psi_{y_0}(Z), \psi_{y_0}(Z'_Y) \big)
					-
					\sum_{r = 0}^{k}
					p^{-\frac{r}{2}}						
					F_{r, K}^E(\sqrt{p} Z, \sqrt{p} Z'_Y) 
					\kappa_{X}^{-\frac{1}{2}}(Z)
					\kappa_{Y}^{-\frac{1}{2}}(Z'_Y)
			\bigg|
			\\
			\leq
			C_3 p^{- \frac{k + 1}{2}}
			\Big(1 + \sqrt{p}|Z| + \sqrt{p} |Z'_Y| \Big)^{Q^1_{k, 0, 0}}
			\exp\Big(- c \sqrt{p} \big( |Z_Y - Z'_Y| + |Z_N| \big) \Big).
		\end{multline}
		From (\ref{eq_je0_exp}), (\ref{eq_jopr_calc}) and (\ref{eq_jrke}), we deduce 
		\begin{equation}\label{eq_jk0_exp}
			J_{0, K}^E(Z, Z'_Y) = 0.
		\end{equation}
		Moreover, under assumption (\ref{eq_comp_vol_omeg}), from (\ref{eq_je1_exp}), (\ref{eq_j1pr_calc}) and (\ref{eq_jrke}), we deduce 
		\begin{equation}\label{eq_jk1_exp}
			J_{1, K}^E(Z, Z'_Y) =  J_{1}^E(Z, Z'_Y).
		\end{equation}
		In particular, from (\ref{eq_kp_as_exp}) and (\ref{eq_jk0_exp}), we see that for some $C_4 > 0$, we can improve (\ref{eq_ext_kp}) as follows
		\begin{equation}\label{eq_ext_kp_impro}
			\Big|  K_p(x, y) \Big| \leq C_4 p^{m - \frac{1}{2}} \exp \big(- c \sqrt{p} \dist(x, y) \big).
		\end{equation}
		\par 
		From (\ref{eq_ext_kp_impro}) and the use of Young's inequality for integral operators as in Section \ref{sect_trace}, we deduce that there are $C_5 > 0$, $p_1 \in \nat^*$, such that for any $p \geq p_1$, we have  $\| K_p \| \leq \frac{C_5}{p^{\frac{n - m+ 1}{2}}}$, which implies (\ref{eq_ext_as}).
		In an analogous way, from (\ref{eq_je1_exp}), (\ref{eq_ext_kp}), (\ref{eq_kp_as_exp}), (\ref{eq_jk0_exp}) and (\ref{eq_jk1_exp}), under additional assumptions (\ref{eq_comp_vol_omeg}) and $A = 0$, we see that in  (\ref{eq_ext_as}), one can replace $p^{\frac{n - m+ 1}{2}}$ by $p^{\frac{n - m + 2}{2}}$.
		\par 
		Now, an easy calculation, using (\ref{eq_kappan}), shows that for any $g \in L^2(Y, \iota^*(L^p \otimes F))$, we have
		\begin{equation}\label{eq_norm_1stterm_er}
			\big\| \ext_p^{0} g \big\|_{L^2(dv_X)}
			=
			\big\| f(p, y) \cdot B_p^Y g \big\|_{L^2(dv_Y)},
		\end{equation}
		where the function $f :\nat \times Y \to \real$ is defined as 
		\begin{equation}\label{eq_func_f_defn}
			f(p, y)^2 := \int_{\real^{2(n - m)}} \kappa_N(y, \sqrt{p} Z_N) \cdot  \exp (- p \pi |Z_N|^2 ) \rho \Big(  \frac{|Z_N|}{r_{\perp}} \Big)^2 dZ_{2m + 1} \wedge \cdots \wedge dZ_{2n}.
		\end{equation}
		From the calculation of Gaussian integral, as $p \to \infty$, we have
		\begin{equation}\label{eq_calc_fp_ex}
			f(p, y)^2 = \frac{\kappa_N(y)}{p^{n - m}}  + O \Big( \frac{1}{p^{n - m + \frac{1}{2}}} \Big).
		\end{equation}
		From (\ref{eq_ext_as}), (\ref{eq_norm_1stterm_er}), (\ref{eq_func_f_defn}) and (\ref{eq_calc_fp_ex}), we deduce (\ref{eq_norm_asymp}).
		\par 
		Now, consider the Toeplitz operator $T_{\kappa_N, p} : L^2(Y, \iota^*( L^p \otimes F)) \to L^2(Y, \iota^*( L^p \otimes F))$, given by
		\begin{equation}
			T_{\kappa_N, p} g := B_p^Y ( \kappa_N \cdot B_p^Y g).
		\end{equation}
		Then, we clearly have
		\begin{equation}
			\big\langle T_{\kappa_N, p} g,  g  \big\rangle_{L^2(dv_Y)}
			=
			\big\langle \kappa_N \cdot B_p^Y g,  B_p^Y  g  \big\rangle_{L^2(dv_Y)}.
		\end{equation}
		Thus, by (\ref{eq_norm_1stterm_er}) and (\ref{eq_calc_fp_ex}), we have 
		\begin{equation}\label{eq_top_final_1}
			\big\| \ext_p^{0} \big\|_{L^2(dv_X)}
			=
			\frac{1}{p^{\frac{n - m}{2}}} 
			\big\| T_{\kappa_N, p} \big\|^{\frac{1}{2}} + O \Big( \frac{1}{p^{\frac{n -m + 1}{2}}} \Big).
		\end{equation}
		\par Recall that for compact manifolds $Y$, Bordemann-Meinrenken-Schlichenmaier \cite[Theorem 4.1]{BordMeinSchl} (for $(F, h^F)$ trivial) and Ma-Marinescu \cite[Theorem 3.19, (3.91)]{MaMarToepl} (for any $(F, h^F)$) established that there is $C > 0$, such that
		\begin{equation}\label{eq_top_final_2}
			\sup_{y \in Y} \kappa_N(y) - \frac{C}{\sqrt{p}}
			\leq
			\big\| T_{\kappa_N, p} \big\| \leq \sup_{y \in Y} \kappa_N(y).
		\end{equation}
		We argue that the proof from \cite{MaMarToepl} continues to hold for non-compact manifolds as well.
		Indeed, recall that the lower bound of (\ref{eq_top_final_2}) was proved in \cite{MaMarToepl} using the asymptotic expansion of the peak section, localized at the point where the supremum of $\kappa_N$ is achieved.
		As the asymptotic expansion of the peak section is based on the asymptotic expansion of the Bergman kernel and the exponential bound on it, and both those results continue to hold on manifolds of bounded geometry by the results of \cite{MaMarOffDiag}, cf. Theorems \ref{thm_bk_off_diag}, \ref{thm_berg_dailiuma}, we see that (\ref{eq_top_final_2}) continues to hold in full generality, see \cite[\S 7.5]{MaHol}, when the supremum of $\kappa_N$ is achieved.
		Now, if the supremum is not achieved, then the same proof gives us that for any $\epsilon > 0$, there is $p_1 \in \nat^*$, such that for $p \geq p_1$, we have
		\begin{equation}\label{eq_top_final_232}
			\sup_{y \in Y} \kappa_N(y) - \epsilon
			\leq
			\big\| T_{\kappa_N, p} \big\|
		\end{equation}
		We deduce (\ref{eq_norm_asymp}) by (\ref{eq_top_final_1}), (\ref{eq_top_final_2}) and (\ref{eq_top_final_232}).
		\par 
		Now it is only left to prove that if $A \neq 0$, then under additional assumption (\ref{eq_comp_vol_omeg}), one can not replace $p^{\frac{n - m+ 1}{2}}$ by $p^{\frac{n - m + 2}{2}}$.
		For this, remark that as long as $A \neq 0$, by (\ref{eq_je1_exp}), the operator, acting on $\comp^n$ with the kernel $F_1^E(Z, Z'_Y)$, has non-zero norm.
		Then, by the calculations, similar to (\ref{eq_norm_1stterm_er}), we see that the operator, acting on $\comp^n$ with the kernel $F_1^E(\sqrt{p} Z, \sqrt{p} Z'_Y)$, has norm of order $\frac{1}{p^{\frac{n - m}{2}}}$, as $p \to \infty$.
		We deduce from this and Theorem \ref{thm_ext_as_exp} that if $A \neq 0$, then under additional assumption (\ref{eq_comp_vol_omeg}), one can not replace $p^{\frac{n - m+ 1}{2}}$ by $p^{\frac{n - m + 2}{2}}$.
	\end{proof}
	\begin{proof}[Proof of Corollary \ref{cor_ot_higher}]
		First of all, since $B_p^X \circ \ext_p = \ext_p$, we clearly have
		$
			\big\| \ext_p - B_p^X \circ \ext_p^{0} \big\| \leq \big\| \ext_p - \ext_p^{0} \big\|
		$.
		Now, the statement is a direct consequence of this, Proposition \ref{prop_der_bound}, applied for $B_p^X \circ \ext_p^{0} g$, and (\ref{eq_norm_asymp}).
	\end{proof}
	\begin{proof}[Proof of Theorem \ref{thm_opt_linf_bnd}]
		We let $K_p = \ext_p - \ext_p^{0}$.
		From Corollary \ref{cor_exp_bound_int} and (\ref{eq_ext_kp_impro}), there are $C > 0$, $p_1 \in \nat^*$, such that for any $p \geq p_1$, $f \in H^0_{(2)}(Y, \iota^*( L^p \otimes F))$, we have
		\begin{equation}\label{eq_ext_as1111}
			\big\| K_p f \big\|_{L^{\infty}(X)} \leq \frac{C}{\sqrt{p}}  \cdot \big\|  f \big\|_{L^{\infty}(Y)}.
		\end{equation}
		Remark also that by construction, we have
		\begin{equation}\label{eq_ext_as1212}
			\big\| \ext_p^0 f \big\|_{L^{\infty}(X)} = \big\|  f \big\|_{L^{\infty}(Y)}.
		\end{equation}
		From (\ref{eq_ext_as1111}) and (\ref{eq_ext_as1212}) we deduce Theorem \ref{thm_opt_linf_bnd}.
	\end{proof}

\bibliography{bibliography}

		\bibliographystyle{abbrv}

\Addresses

\end{document}